\documentclass{article}
\usepackage{bm}
\usepackage{tikz,mathpazo}
\usepackage{pgf}
\usepackage[top=2.5cm, bottom=2.5cm, left=3cm, right=3cm]{geometry}   
\usepackage{indentfirst}
\usepackage{graphicx}
\usepackage{graphics}
\usepackage[toc,page,title,titletoc,header]{appendix}
\usepackage{bbm}
\usepackage{listings}  
\usepackage{amsmath}
\usepackage{setspace} 
\usepackage{indentfirst}
\usepackage{caption}
\usepackage{multirow} 
\usepackage{lipsum,multicol}
\usepackage{pdfpages}
\usepackage{float}
\usepackage{amsthm}
\usepackage{pdfpages}
\usepackage{url}
\usepackage{colortbl}
\usepackage{subfig}
\usepackage{epsfig}
\usepackage{epstopdf}
\usetikzlibrary{shapes.geometric, arrows}
\usepackage{fancyhdr}
\usepackage{abstract}
\usepackage{tikz,mathpazo}
\usetikzlibrary{shapes.geometric, arrows}
\usepackage{amssymb}
\usepackage{latexsym}
\usepackage{verbatim}
\usepackage[numbers]{natbib} 
\usepackage{booktabs}
\usepackage{enumitem}
\usepackage{siunitx}
\usepackage{makecell}
\usepackage{bbding}
\usepackage{algorithm}  
\usepackage{algpseudocode}

\usepackage{hyperref}
\hypersetup{
    colorlinks=true,
    linkcolor=blue,
    filecolor=magenta,      
    urlcolor=cyan,
    citecolor = cyan,
}

\allowdisplaybreaks[4]

\newcommand{\norm}[1]{\left\Vert #1\right\Vert}
\newcommand{\bb}[1]{\mathbb{#1}}

\newcommand{\conv}[0]{\mathrm{conv}\,}%convex hull
\newcommand{\X}{{ \ca{X} }}

\newcommand{\ca}[1]{\mathcal{#1}}

\newcommand{\Diag}[0]{\mathrm{Diag}}

\newcommand{\tp}{^\top}

\newcommand{\D}{D}

\newcommand{\Rnd}{\mathbb{R}^{n\times d}}
\newcommand{\Rn}{\mathbb{R}^n}

\newcommand{\Rm}{\mathbb{R}^m}

\newcommand{\Rmd}{\mathbb{R}^{m\times d}}

\newtheorem{theo}{Theorem}[section]
\newtheorem{lem}[theo]{Lemma}
\newtheorem{prop}[theo]{Proposition}

\newtheorem{coro}[theo]{Corollary}

\newtheorem{defin}[theo]{Definition}
\newtheorem{rmk}[theo]{Remark}
\newtheorem{assumpt}[theo]{Assumption}

\usepackage[figuresright]{rotating}
\usepackage{pdflscape}

\usepackage{caption}
\usepackage{algorithm}
\usepackage{algpseudocode}
\usepackage{longtable}
\usepackage{appendix}
\usepackage[figuresright]{rotating}
\usepackage{pdflscape}

\numberwithin{equation}{section}

\title{Convergence of Decentralized Stochastic Subgradient-based Methods for Nonsmooth Nonconvex functions}

\author{Siyuan Zhang\thanks{State Key Laboratory of Scientific and Engineering Computing, Academy of Mathematics and Systems Science, Chinese Academy of Sciences, and University of Chinese Academy of Sciences, China (e-mail: zhangsiyuan@amss.ac.cn)}, ~ 
Nachuan Xiao\thanks{School of Data Science, The Chinese University of Hong Kong, Shenzhen, China.  (e-mail: xncxy@cuhk.edu.cn).},
~ and  ~ 
Xin Liu\thanks{State Key Laboratory of Scientific and Engineering Computing, Academy of Mathematics and Systems Science, Chinese Academy of Sciences, and University of Chinese Academy of Sciences, China (e-mail: liuxin@lsec.cc.ac.cn)}}

\begin{document}
\maketitle

\begin{abstract}%   <- trailing '%' for backward compatibility of .sty file
In this paper, we focus on the decentralized stochastic subgradient-based methods in minimizing nonsmooth nonconvex functions without Clarke regularity, especially in the decentralized training of nonsmooth neural networks. We propose a general framework that unifies various decentralized subgradient-based methods, such as decentralized stochastic subgradient descent (DSGD), DSGD with gradient-tracking technique (DSGD-T), and DSGD with momentum (DSGD-M). To establish the convergence properties of our proposed framework, we relate the discrete iterates to the trajectories of a continuous-time differential inclusion, which is assumed to have a coercive Lyapunov function with a stable set $\mathcal{A}$. We prove the asymptotic convergence of the iterates to the stable set $\mathcal{A}$ with sufficiently small and diminishing step-sizes. These results provide first convergence guarantees for some well-recognized of decentralized stochastic subgradient-based methods without Clarke regularity of the objective function. Preliminary numerical experiments demonstrate that our proposed framework yields highly efficient decentralized stochastic subgradient-based methods with convergence guarantees in the training of nonsmooth neural networks. 

\end{abstract}


\section{Introduction}
In this paper, we consider the following decentralized optimization problem (DOP) on an undirected, connected graph $\mathtt{G} = (\mathtt{V}, \mathtt{E})$ containing $d$ agents,
\begin{equation}
\tag{DOP}
\label{Prob_DOP}
\begin{aligned}
\min_{{\bm x}_1, {\bm x}_2, \ldots, {\bm x}_d \in \Rn} & \quad   \sum_{i=1}^d f_i({\bm x}_i), \\
\text { s. t. } & \quad  {\bm x}_{i}={\bm x}_{j}, \quad \forall(i, j) \in \mathtt{E} .
\end{aligned}
\end{equation}
Here, for each $i \in [d]:=\{1,2,\ldots,d\}$, ${\bm x}_{i}$ refers to the local variable of agent $i$. Moreover, the set of nodes $\mathtt{V}= [d]:= \{1, \ldots, d\}$ represents the set of agents, while the set of edges $\mathtt{E}\subseteq \mathtt{V}\times \mathtt{V}$ represents the set of communication links. An edge $(i,j)\in \mathtt{E}$ refers to that agents $i$ and $j$ are neighbors and can communicate with each other. Furthermore, for any $i \in [d]$, $f_i: \mathbb{R}^{n} \to \mathbb{R}$ is a locally Lipschitz continuous (possibly nonconvex and nonsmooth) cost function exclusively known to the agent $i$ and takes the empirical-risk formulation \cite{shapiro2021lectures}. More precisely, for any $i \in [d]$, there exists $F_i: \Rn \times \Omega \to \bb{R}$ such that 
\begin{equation}\label{eq:emrisk}
f_i({\bm x}):= \frac{1}{|\mathcal{S}_{i}|}\sum_{s_{l}\in \mathcal{S}_{i}} F_{i}({\bm x}; {\bm s}_{l}),
\end{equation}
where ${\bm s}_{l}$ is a random data vector supported on a local finite set $\mathcal{S}_{i}$. In real-world applications, the size of $\mathcal{S}_{i}$ is often sufficiently large, in the sense that computing the exact (sub)gradient of $f_i$ is very costly, which makes stochastic (sub)gradient-based methods inevitable. Problem \eqref{Prob_DOP} has wide applications in machine learning \cite{jain2017non, bottou2018optimization}, signal processing \cite{mateos2012distributed, cohen2016distributed} and control \cite{bolognani2014distributed}, etc.

\subsection{Existing works on decentralized stochastic optimization}

Motivated by stochastic gradient descent (SGD), the decentralized stochastic gradient descent (DSGD) \cite{nedic2009distributed, jiang2017collaborative, lian2017can} is one of the most fundamental methods for solving \eqref{Prob_DOP}. In DSGD, each agent locally averages its model parameters with those of its neighbors and updates along the direction of the negative local stochastic gradient, which can be expressed as
\begin{equation}
    \label{Eq_intro_DSGD0}
    {\bm x}_{i, k+1}= \sum_{j = 1}^d {\bm W}(i,j){\bm x}_{j, k} -  \frac{\eta_k}{|\mathcal{B}_{i,k}|} \sum_{l \in \mathcal{B}_{i,k} } \nabla F_{i}({\bm x}_{i,k}; {\bm s}_{l}) ,
\end{equation}
where the subset $\mathcal{B}_{i,k}\subseteq \mathcal{S}_{i}$ refers to the mini-batch selection of the samples. Therefore,  for any $i\in [d]$, the term $\frac{1}{|\mathcal{B}_{i, k}|} \sum_{l \in \mathcal{B}_{i, k}} \nabla F_{i}({\bm x}_{i,k}; {\bm s}_{l})$ refers to the stochastic gradient of $f_i$ at ${\bm x}_{i,k}$, In addition, we define $\xi_{i,k+1} :=  \frac{1}{|\mathcal{B}_{i,k}|}\sum_{l \in \mathcal{B}_{i,k} } \nabla F_{i}({\bm x}_{i,k}; {\bm s}_{l})- \nabla f_{i}({\bm x}_{i,k})$ as the corresponding evaluation noise. ${\bm W} \in \bb{R}^{d\times d}$ represents the mixing matrix corresponding to the graph $\mathtt{G} = (\mathtt{V}, \mathtt{E})$. That is, ${\bm W}$ is symmetric, doubly stochastic, and ${\bm W}(i,j) = 0$ whenever $(i,j) \notin \mathtt{E}$ (see Section \ref{mix} for details).

Based on DSGD, a wide range of variants have been developed for solving \eqref{Prob_DOP}. Some works introduce auxiliary variables within DSGD to track the global stochastic gradient of the objective function, such as
GNSD \cite{lu2019gnsd} and DSGT \cite{pu2021distributed}. To accelerate convergence, some recent works \cite{yu2019linear,gao2020periodic} propose the decentralized momentum SGD by integrating the heavy-ball momentum \cite{polyak1964some} into DSGD. Moreover, \cite{gao2023distributed} combines the stochastic gradient tracking method with momentum acceleration, {{\cite{chen2023convergence} proposes a general framework for decentralized adaptive methods such as Adam, AdaGrad, and AMSGrad, and show that under suitable conditions, for nonconvex smooth objectives, if the original adaptive method converges, then its decentralized counterpart also converges.}} In addition, \cite{huang2023distributed, huang2023distributed2} introduce random reshuffling (RR) \cite{de2020random} sampling techniques (i.e., randomly samples an index $i_k$ of $[N]$ without replacement at each step $k$ in agent $i$, ensuring that each index appears exactly once per epoch.) into DSGD and DSGT, and provide theoretical evidence that RR outperforms with-replacement sampling. 
A series of studies \cite{lian2017can, george2019distributed, yuan2021decentlam, zhang2019decentralized, xin2021improved} have established theoretical guarantees for above decentralized SGD-type methods with continuously differentiable nonconvex objective functions. {{Recently, a line of works \cite{Yan2023CompressedDP, Bylinkin2024AcceleratedMW, condat2024locodl, xu2023parallel, sun2024role, Liu2025CompressedDM} have focused on developing distributed SGD-type methods that reduce communication and handle asynchrony and heterogeneity. For instance, \cite{condat2024locodl} couples local training  with communication compression to obtain a communication-efficient distributed optimization method with theoretical guarantees with strongly convex objectives. \cite{xu2023parallel} analyzes asynchronous adaptive stochastic gradient methods, showing that AMSGrad-type schemes can be parallelized under bounded staleness while retaining convergence guarantees with smooth objective functions. From the perspective of federated learning, \cite{sun2024role} studies the role of server momentum and develop a more general server-momentum framework, which highlights how momentum can mitigate the adverse effects of heterogeneous data and system asynchrony with smooth objective functions.}}

Much attention has also been directed towards decentralized nonsmooth optimization. Some existing works, such as ProxDGD \cite{zeng2018nonconvex}, NEXT \cite{di2016next}, and PG-EXTRA \cite{shi2015proximal} develop decentralized proximal gradient type methods for \eqref{Prob_DOP}, under the assumption that $f_i$ is composed of a differentiable term $g_{i}$ and a convex regularization term $r_{i}$ for all $i \in [d]$. In this context, the objective function is weakly convex and thus is Clarke regular. Several existing works \cite{di2019distributed, niu2021distributed, xin2021stochastic, xiao2023one, zhou2025decentralized} extend the convergence analysis of Prox-DGD and NEXT  from the deterministic case to the stochastic case, under additional assumptions regarding boundedness and heterogeneity. Additionally, \cite{chen2021distributed} proposes stoDPSM for $\rho$-weakly convex objectives and shows the global convergence while employing the Moreau envelope for measuring the stationarity. \cite{koloskova2021improved} improves the convergence analysis for DSGT and proves its convergence for weakly convex objectives. From the view of distributed gradient flow, \cite{swenson2022distributed}  establishes the global convergence of DSGD for Clarke regular function with decreasing consensus parameter. 
The detailed comparisons between the existing
results on decentralized SGD-based methods in nonsmooth nonconvex optimization are listed in Table \ref{tab:2}.

Despite achieving nice theoretical results, all the aforementioned works are only capable of handling nonsmooth but (Clarke) regular \cite{clarke1990optimization} objective functions in \eqref{Prob_DOP}, while some of these works evaluate the efficiency of their developed methods in the decentralized training of neural networks. However, nonsmooth activation functions, including ReLU and leaky ReLU, have served as popular building blocks for modeling neural networks, resulting in loss functions that are nonsmooth and lack Clarke regularity. To the best of our knowledge, \cite{sahinoglu2024online} is the only recent work that investigates the convergence of decentralized SGD-based method without requiring Clarke regularity.  They propose a decentralized first-order and zero-order online algorithm called Me-DOL, but it requires several strong assumptions, such as global Lipschitz function and $L^2$ bounded stochastic subgradients. As a result, how to establish convergence properties for decentralized stochastic subgradient-based methods in nonsmooth, nonconvex settings remains an area for further exploration.

\begin{table}
    \fontsize{8}{12}\selectfont  
    \begin{tabular}{|c|c|c|c|c|}
    \hline
    {Method} &  {\makecell[c]{{Update}\\{scheme}}} & {Step-sizes} & {Conditions on $f_i$} & {Conditions on noises} \\
    \hline
\makecell[c]{{stoDPSM}\\{\cite{chen2021distributed}}}   &  GD & \makecell[c]{{Square summable,}\\ {  $\eta_{k+1}/\eta_{k} \to 1$}} &  \makecell[c]{{$\rho$-weakly convex,}\\  {bounded subgradient}}  & \makecell[c]{{WRS,} \\ {bounded second moment}} \\     
    \hline  
\makecell[c]{{S-NEXT}\\{\cite{di2019distributed}}}   &  GT & {Square summable} &  \makecell[c]{{Proximal, weakly convex,}\\  {bounded subgradient}}  & WRS, bounded noise \\  
    \hline  
\makecell[c]{{ProxGT}\\{\cite{xin2021stochastic}}}  &  GT & Constant &  \makecell[c]{{Proximal, weakly convex,}\\  {mean-squared smoothness}}  &  WRS, bounded variance   \\  
    \hline  
{Prox-DASA} \cite{xiao2023one}   &  {GD/GT} & {$\eta_{k}=\Theta(\frac{1}{\sqrt{K}}), k\in[K]$} &  {Proximal, weakly convex}  &  {{WRS, bounded variance}}   \\ 
    \hline
{\makecell[c]{{DEPOSITUM}\\{\cite{zhou2025decentralized}}}}   &  {GT-M} & {$\eta_{k}=\Theta(\frac{1}{\sqrt{K}}), k\in[K]$} & {Proximal, weakly convex} &  {WRS, bounded variance}  \\  
    \hline  
\makecell[c]{{DGF}\\ {\cite{swenson2022distributed}}}   &  GD & $\eta_{k}=\Theta(k^{-p}), p\in (0,1]$ &  \makecell[c]{{Clarke regular, weakly coercive,}\\  {decreasing consensus parameter}}  & {WRS, bounded variance}   \\  
    \hline
\makecell[c]{{ME-DOL}\\{\cite{sahinoglu2024online}}}   &  GD & $\eta_{k}=\Theta(\frac{1}{\sqrt{K}}), k\in[K]$ &  Global Lipschitz continuous &  \makecell[c]{{WRS, bounded variance,}\\{bounded second moment} }  \\  
    \hline
 Our work \eqref{Eq_Framework}   &  \makecell[c]{{GD/GT}\\ {/GD-M}} &  $\eta_k=o(1/\log(k))$  & \makecell[c]{{Path-differentiable,} \\{Coercive}}  &   \makecell[c]{{WRS, conditionally}\\ {controlled noises}}  \\  
    \hline
Our work \eqref{Eq_Framework}   &  \makecell[c]{{GD/GT}\\ {/GD-M}} & diminishing $\eta_{k}$ & \makecell[c]{{Path-differentiable,} \\{Coercive}}  &   RR   \\  
    \hline
\end{tabular}
\caption{A brief comparison of existing decentralized SGD-based methods for nonsmooth nonconvex optimization. Here, ``GD", ``GT" and ``GD-M", "GT-M" are abbreviations of  ``gradient descent", ``gradient tracking", ``gradient descent with momentum technique" and ``gradient tracking with momentum technique"; ``WRS" and ``RR" 
are abbreviations of  ``with-replacement sampling" and ``random reshuffling"; The term ``conditionally controlled noise" means that the evaluation noise is bounded conditional on $\mathcal{F}_k=\sigma(\{{\bm x}_{i, l}, l\leq k, i \in [d]\})$. The concept "path-differentiable" is defined in Definition \ref{def:path}.}
\label{tab:2}
\end{table}

\subsection{Existing works on stochastic subgradient-based methods for nonsmooth nonconvex optimization}

In the training of nonsmooth neural networks, one major challenge comes from computing the (stochastic) subgradients of the loss function. In a wide range of deep learning packages, such as TensorFlow, PyTorch, and JAX, automatic differentiation (AD) algorithms have become the default choice for computing subdifferential. Utilizing the chain rule, AD algorithms recursively construct the so-called ``subgradients" through the composition of Jacobians of each block of the neural network, yet they neglect the underlying nonsmoothness. As the chain rule fails for loss functions lacking Clarke regularity, the so-called ``subgradient" may not necessarily be contained in its Clarke subdifferential (see examples in \cite{bolte2020mathematical}). 
This limitation consequently renders the theoretical analyses of existing literature inapplicable to such functions.

To characterize the behavior of AD algorithms, \cite{bolte2021conservative} introduces the concept of path-differentiable functions, which constitute a broad function class, including semi-algebraic functions, semi-analytic functions, and functions whose graphs are definable in some o-minimal structures \cite{davis2020stochastic}. More importantly, the set of path-differentiable functions is general enough to cover almost all loss functions in neural network tasks, encompassing most of the loss functions without Clarke regularity. The study \cite{bolte2021conservative} further demonstrates that when applied to path-differentiable loss functions, the outputs of AD algorithms are contained within a conservative field, which is a generalization of the Clarke subdifferential. Conservative field admits chain rules for path-differentiable functions, thereby preserving several good properties. For a comprehensive discussion of this concept, readers are directed to Section \ref{sec:conser}.

Based on the concepts of path-differentiable functions and conservative field, some recent works \cite{bolte2021conservative, 
 castera2021inertial, pauwels2021incremental, xiao2023adam, ruszczynski2020convergence,le2024nonsmooth,xiao2023convergence, bolte2022long} leverage the ordinary differential equation (ODE) approach \cite{benaim2005stochastic, benaim2006dynamics, borkar2009stochastic,duchi2018stochastic} to study the behavior of stochastic subgradient-based methods in the non-distributed setting.  Although these studies have achieved significant progress in establishing convergence properties for stochastic subgradient-based methods, the theoretical results of many approaches cannot avoid infinitely many spurious critical points due to the gap between the conservative field and the Clarke subdifferential \cite{bianchi2022convergence, bolte2021nonsmooth}. Moreover, extending existing results to the multi-agent scenario is non-trivial, as the update schemes of all the agents are coupled through local communication under decentralized settings. Given the convergence properties that need improvement and limited existing results in the multi-agent scenario, developing a unified framework for establishing convergence properties for decentralized stochastic subgradient-based methods is of great importance and worth exploring.

\subsection{A general framework for decentralized stochastic subgradient-based methods}\label{sec:dsm}

We define the collection of the local variables, evaluation noises and  noiseless gradients in \eqref{Eq_intro_DSGD0} by matrix ${\bm X}_k $,  $\Xi_{k+1}$, and ${\bm H}_k$, respectively: 
\begin{equation}
\begin{aligned}
    {\bm X}_k := & \left[{\bm x}_{1,k}, {\bm x}_{2,k}, \ldots, {\bm x}_{d,k} \right],\\
    \Xi_{k+1} := & \left[\xi_{1,k+1}, \xi_{2,k+1}, \ldots, \xi_{d,k+1} \right],\\
    {\bm H}_k := & [\nabla f_{1}({\bm x}_{1,k}), \ldots, \nabla f_d({\bm x}_{d,k})].
\end{aligned}
\end{equation}
Then for any $i\in [d]$, we denote $\Phi^{\mathrm{GD}}_{i}: {\bm x} \mapsto \{\nabla f_{i}({\bm x})\}$ as a singleton-valued mapping from $\mathbb{R}^{n}$ to $2^{\mathbb{R}^n}$, which characterize the noiseless update directions in the $i$-th agent in the DSGD method. With these notations, the update scheme of DSGD in \eqref{Eq_intro_DSGD0} can be reformulated as follows:
\begin{equation}
    \label{Eq_intro_DSGD1}
    \begin{aligned}
        &{\bm X}_{k+1} = {\bm X}_k {\bm W} - \eta_k ({\bm H}_k + \Xi_{k+1}) \\
        \in{}& {\bm X}_k {\bm W} - \eta_k (\left[\Phi^{\mathrm{GD}}_{1}({\bm x}_{1,k}),  \ldots, \Phi^{\mathrm{GD}}_{d}({\bm x}_{d,k}) \right]  + \Xi_{k+1}). 
    \end{aligned}
\end{equation}

In light of the reformulated update scheme \eqref{Eq_intro_DSGD1}, we consider the following unified framework for decentralized stochastic subgradient-based methods, 
    \begin{equation}
        \label{Eq_Framework}
        \tag{DSM}
        {\bm Z}_{k+1} ={\bm Z}_k {\bm W}- \eta_k ({\bm H}_k + \Xi_{k+1}). 
    \end{equation}
    Here, ${\bm Z}_k = [{\bm z}_{1,k}, \ldots, {\bm z}_{d,k}] \in \Rmd$ denotes the collection of local variables (e.g. model parameters, momentum, etc.),  ${\bm H}_k  \in \Rmd$  refers to the collection of noiseless update directions, $\Xi_{k+1}\in \mathbb{R}^{m\times d}$ refers to the collection of evaluation noises, and $\eta_k >0$ denotes the step-size. 
    
Moreover, the noiseless update directions in \eqref{Eq_Framework} is characterized by a family of locally bounded graph-closed set-valued mappings $\{\Phi_i: i\in [d]\}$ that maps $\Rm$ to the collections of subsets of $\Rm$. More precisely, there exists a sequence of nonnegative real numbers $\{\delta_k\}$ such that 
\begin{equation}\label{relation_H_PHI}
    \frac{1}{d}{\bm H}_{k}{\bm 1_{d}}  \in \conv\left( \frac{1}{d} \sum_{i = 1}^d \Phi_i^{\delta_k}({\bm z}_{i, k}) \right), \quad \forall k \in \mathbb{N}.
\end{equation}
Here, $\Phi_{i}^{\delta}({\bm x}):= \{ {\bm y} \in \mathbb{R}^n: \exists {\bm z}, \text{ s. t. } \norm{{\bm z}-{\bm x}}< \delta, \inf_{{\bm w}\in \Phi_i({\bm z})}\norm{{\bm y}-{\bm w}} < \delta\}$. \footnote{The superscript is used to denote expansion only when applied to set-valued mappings.} Therefore, it is easy to verify that \eqref{Eq_Framework} encloses DSGD method in \eqref{Eq_intro_DSGD0}.

\subsection{Contributions}

The main contributions of our paper can be summarized as follows.

We propose a unified framework \eqref{Eq_Framework} that encompasses a variety of decentralized stochastic subgradient-based methods.  We verify that decentralized stochastic subgradient descent (DSGD), DSGD with gradient-tracking technique (DSGD-T) and  DSGD with momentum (DSGD-M) all fit into the framework \eqref{Eq_Framework}. In particular, a decentralized variant of SignSGD \cite{bernstein2018signsgd}, denoted as DSignSGD, is also enclosed by \eqref{Eq_Framework}.

We establish the convergence guarantees for our framework \eqref{Eq_Framework}. By introducing the following differential inclusion 
\begin{equation}
         \label{Eq_intro_DI}
         \frac{\mathrm{d}{\bm z}}{\mathrm{d}t} \in  -\conv\left( \frac{1}{d}\sum_{i = 1}^d \Phi_i({\bm z}) \right),
\end{equation}
that admits a coercive Lyapunov function $\psi$ and stable set $\mathcal{A}$, we prove that, under mild conditions with controlled noises, the sequence $\{{\bm Z}_k\}$ generated by \eqref{Eq_Framework} is uniformly bounded. Then we prove that the iterates $\{{\bm z}_{i,k}\}, i\in [d]$ asymptotically reach consensus and approximate the trajectories of the differential inclusion \eqref{Eq_intro_DI}, hence converging to the stable set $\mathcal{A}$ under sufficiently small and diminishing step-size.

Based on our theoretical results of the proposed framework \eqref{Eq_Framework}, we prove that DSGD, DSGD-M, DSGD-T, and DSignSGD achieve global convergence in the random reshuffling case and high-probability global convergence in the with-replacement sampling case. We also show that with random initialization, DSGD, DSGD-M and DSignSGD can avoid the spurious critical points in the sense of conservative field in minimizing path-differentiable functions. Although these algorithms are widely used in the training of nonsmooth neural networks, existing theoretical results are limited to Clarke regular objective functions. Our theoretical findings first fill the gap in solving Problem \eqref{Prob_DOP} with objectives lacking Clarke regularity.

Preliminary numerical experiments show the efficiency of analyzed methods
and exhibit the superiority of our proposed method DSignSGD in comparison with DSGD and DSGD-M, hence demonstrating the flexibility and great potential of our framework.

\subsection{Organizations}

The rest of this paper is organized as follows. In Section \ref{sec:2}, we present some basic notations and definitions. In Section \ref{sec:3}, we stipulate some assumptions for our proposed framework \eqref{Eq_Framework}, and establish convergence results under different noise settings by linking it with the dynamics of its corresponding differential inclusion. In Section \ref{sec:4}, we demonstrate that framework \eqref{Eq_Framework} covers major existing decentralized stochastic subgradient-based methods, which inherit the convergence properties from \eqref{Eq_Framework} and avoid infinitely many spurious critical points. Numerical experiments regarding analyzed decentralized subgradient-based methods are presented in Section \ref{sec:6}. In the last section, we draw conclusions and discuss possible future research directions.

\section{Preliminaries}\label{sec:2}

\subsection{Notations}\label{sec:notation}

We denote $\langle\cdot, \cdot\rangle$ as the standard inner product, while  $\|\cdot\|$ as the $\ell_2$-norm of a vector or spectral norm of a matrix. We refer to 
 $\|\cdot\|_{1}$ as  $\ell_1$-norm of a vector and $\|\cdot\|_{\mathrm{F}}$  as Frobenius norm of a matrix. 
$\mathbb{B}({\bm x}, \delta):=\left\{\tilde{{\bm x}} \in \mathbb{R}^n:\|\tilde{{\bm x}}-{\bm x}\|^2\leq \delta^2\right\}$ refers to the ball centered at ${\bm x}$ with radius $\delta$. For a given set $\mathcal{Y}, \operatorname{dist}({\bm x}, \mathcal{Y})$ denotes the distance between ${\bm x}$ and a set $\mathcal{Y}$, i.e. $\operatorname{dist}({\bm x}, \mathcal{Y}):=\arg \min_{y \in \mathcal{Y}}\|{\bm x}-{\bm y}\|,$ $\conv \mathcal{Y}$ denotes the convex hull of $\mathcal{Y}$, and $\mathcal{Y}^d$ denotes $d$-fold Cartesian product. The notation $\otimes$ stands for the Kronecker product. $\Delta_m:= \{(\lambda_0, \ldots, \lambda_m): \lambda_i \geq 0, \sum_{i=0}^{m}\lambda_i =1\}$ stands for the  simplex of dimension $m$.

For any positive sequence  $\left\{\eta_k\right\}$, let $\lambda_{\eta}(0):=0, \lambda_{\eta}(i):=\sum_{k=0}^{i-1} \eta_k$, and $\Lambda_{\eta}(t):=\sup \left\{k \in\mathbb{N}:\right. $ $\left. t \geq \lambda_{\eta}(k)\right\}.$
More explicitly, $\Lambda_{\eta}(t)=p$, if $\lambda_{\eta}(p) \leq t<\lambda_{\eta}(p+1)$. We define the set-valued mapping sign : $\mathbb{R}^n \rightrightarrows \mathbb{R}^n$ as follows:
$$(\operatorname{sign}({\bm x}))_i= \begin{cases}\{-1\}, & {\bm x}_i<0;\\ {[-1,1],} & {\bm x}_i=0;\\ \{1\}, & {\bm x}_i>0.\end{cases}$$
It is easy to verify that $\operatorname{sign}({\bm x})$ is the Clarke subdifferential of the nonsmooth functions ${\bm x} \mapsto\|{\bm x}\|_1$. 

Moreover, we denote $\mathcal{N}_{i}$ as the set of neighbor agents of agent $i$ together with $\{i\}$, $\mathbb{R}_{+}$ as the set of all nonnegative real numbers. For any $N>0,[N]:=\{1, \ldots, N\}$.  And let ${\bm 1}_{d}\in \mathbb{R}^d$ represent a vector of all $1$'s, ${\bm e}_{i}\in \mathbb{R}^{d}$ represent $[0, \ldots, 1, \ldots, 0]^{\top}$, where $1$ is the $i$-th component. For two integers $i$ and $j$, $i \land j$ represents $\min\{i,j\}$. 

We denote $(\Omega, \mathcal{F}, \mathbb{P})$ as the probability space. We use $\sigma({\bm X})$ to denote the sigma-algebra generated by the random variable ${\bm X}$. We say $\left\{\mathcal{F}_k\right\}_{k \in \mathbb{N}}$ is a filtration if $\left\{\mathcal{F}_k\right\}$ is a collection of $\sigma$-algebras that satisfies $\mathcal{F}_0\subseteq \mathcal{F}_1\subseteq \cdots \subseteq \mathcal{F}_{\infty} \subseteq \mathcal{F}$.  A stochastic series $\{\xi_{k}\}$ is called a martingale with respect to a filtration $\{\mathcal{F}_{k}\}$, if $\{\xi_{k}\}$ is adapted to the filtration $\{\mathcal{F}_{k}\}$ and $\mathbb{E}[\xi_{k+1} | \mathcal{F}_{k} ] = \xi_{k}$, $ \forall k \in \mathbb{N}$. Moreover, A stochastic series  $\{\xi_{k}\}$  is referred to as a martingale difference sequence, if $\{\xi_{k}\}$ is adapted to the filtration $\{\mathcal{F}_{k}\}$ and $\mathbb{E}[\xi_{k+1} | \mathcal{F}_{k} ] = 0$ holds for all $k\in \mathbb{N}$. 

In addition, we define the summation function $f$ of \eqref{Prob_DOP} as
\begin{equation}
    \label{Eq_defin_f}
    f({\bm x}) := \frac{1}{d} \sum_{i = 1}^d f_i({\bm x}).
\end{equation}

\subsection{Mixing matrix}\label{mix}
Mixing matrix ${\bm W}=[{\bm W}(i, j)] \in \mathbb{R}^{d \times d}$ determines the structure of the communication network and plays an important role in averaging the local information of neighbor agents. Generally, we assume the mixing matrix ${\bm W}$  satisfies the following properties, which are standard in the literature \cite{xiao2004fast, sundhar2010distributed, lian2017can}.
\begin{defin}\label{Def_mixing_matrix}
Given a connected graph $\mathtt{G} = (\mathtt{V}, \mathtt{E})$, we say ${\bm W} \in \bb{R}^{d\times d}$ is a mixing matrix of $\mathtt{G}$, if it satisfies the following conditions. 
\begin{enumerate}
\item ${\bm W}$ is symmetric.
\item ${\bm W}$ is doubly stochastic, namely, ${\bm W}$ is nonnegative and ${\bm W} {\bm 1}_d={\bm W}^{\top} {\bm 1}_d={\bm 1}_d$.
\item ${\bm W}(i, j)=0$, if and only if $i \neq j$ and $(i, j) \notin \mathtt{E}$.
\end{enumerate}
\end{defin}
With a given graph $\mathtt{G}$, there are various approaches to choose its corresponding mixing matrix, such as Laplacian-based constant edge weight matrix \cite{xiao2004fast} and Metropolis weight matrix \cite{xiao2006distributed}. We refer interested readers to \cite{nedic2018network, shi2015extra, wang2022decentralized, wang2025decentralized} for more details about how to choose the mixing matrix.

Based on \cite[Perron-Frobenius Theorem]{pillai2005perron}, we derive a direct corollary, which illustrates the fundamental spectral property of mixing matrix ${\bm W}$. 
\begin{coro}
    \label{Prop_preliminary_mixing_matrix}
    For any mixing matrix ${\bm W} \in \bb{R}^{d\times d}$ that corresponds to a connected graph $\mathtt{G}$, all the eigenvalues of ${\bm W}$ lie in $(-1,1]$, and ${\bm W}$ has a single eigenvalue at $1$ that admits ${\bm 1_{d}}$ as its eigenvector. 
\end{coro}

\subsection{Nonsmooth analysis and conservative field}\label{sec:conser}

A set-valued mapping $D: \mathbb{R}^{n} \rightrightarrows \mathbb{R}^n $ is a mapping from $\mathbb{R}^n$ to set of subsets of $\Rn$. The graph of $D$ is given by $\operatorname{graph} D = \{ ({\bm x},{\bm z}) \mid {\bm x}\in \mathbb{R}^n, {\bm z}\in D({\bm x})\}$. 
The set-valued mapping is said to be graph-closed (or has closed graph), if $\operatorname{graph} D$ is a closed subset of $\Rn \times \Rn$. $D$ is said to be locally bounded, if for any ${\bm x}\in \mathbb{R}^n$, there exists a neighborhood $U_{{\bm x}}$ of ${\bm x}$ such that $\sup_{{\bm z}\in D({\bm y}), {\bm y}\in U_{{\bm x}}} \| {\bm z}\|< + \infty$. In addition, $D$ is convex-valued, if $D({\bm x})$ is a convex subset of $\Rn$ for any ${\bm x} \in \Rn$.

\begin{defin}\cite{clarke1990optimization}
For any given locally Lipschitz continuous function $f: \mathbb{R}^{n} \to \mathbb{R}$, and any ${\bm x}\in \mathbb{R}^n$, the generalized directional derivative of  $f$ at ${\bm x}$ along the direction ${\bm d} \in \mathbb{R}^n$, is defined as  
\begin{equation*}
  f^{\circ}({\bm x};{\bm d}):=  \limsup_{{\bm y}\to {\bm x}, t\downarrow 0}\frac{f({\bm y}+t{\bm d})-f({\bm y})}{t}.
\end{equation*}
The Clarke subdifferential of $f$ at ${\bm x}\in \mathbb{R}^n$, denoted by $\partial f({\bm x})$, is defined as 
\begin{equation*}
\partial f({\bm x}):= \left\{ {\bm u} \in \mathbb{R}^n:  f^{\circ}({\bm x};{\bm d}) \geq \langle {\bm u}, {\bm d} \rangle, \forall {\bm d}\in \mathbb{R}^n \right\}.
\end{equation*}
\end{defin}

Notice that $\partial f$ is a convex-valued graph-closed locally bounded set-valued mapping from $\Rn$ to the subsets of $\Rn$. Based upon the concept of generalized directional derivative, we present the definition of Clarke regular functions.

\begin{defin}\cite{clarke1990optimization}
We say that $f$ is Clarke regular at ${\bm x} \in \Rn$, if for every direction
${\bm d}\in \Rn$, the one-sided directional derivative 
\begin{equation*}
    f^{*}({\bm x};{\bm d}):= \lim_{t\downarrow 0}\frac{f({\bm x}+t{\bm d})-f({\bm x})}{t}
\end{equation*}
exists, and $f^{*}({\bm x};{\bm d})= f^{\circ}({\bm x};{\bm d}).$ 
\end{defin}

Next, we present a brief introduction to the conservative field, which characterizes the output of AD algorithms in differentiating nonsmooth neural networks.
\begin{defin}[Conservative field]\label{def:conservative}
    Let $D: \mathbb{R}^{n} \rightrightarrows \mathbb{R}^n$ be a nonempty set-valued mapping. We call $D$ a conservative field if it is compact-valued and graph-closed, and for any absolutely continuous loop $\gamma: [0,1] \rightarrow \mathbb{R}^n$ satisfying $\gamma(0)=\gamma(1)$, it holds that 
        \begin{equation*} 
        \int_0^1\max_{{\bm v} \in D(\gamma(t))}\langle\dot{\gamma}(t), {\bm v}\rangle \mathrm{d} t=0.
        \end{equation*}
\end{defin}

It is remarkable to note that any conservative field is locally bounded \cite[Remark 3]{bolte2021conservative}. We now introduce the definition of the path-differentiable function corresponding
to the conservative field.

\begin{defin}\label{def:path}
Let $D$ be a conservative field, we say a function $f$ is a potential for $D$ or $f$ is path-differentiable, provided that $f$ can be formulated as 
$$  f({\bm x})={f({\bm x}_{0})}+\int_0^t\langle\dot{\gamma}(s), D(\gamma(s))\rangle \mathrm{d} s,
$$
for any absolutely continuous curve $\gamma$ with $\gamma(0)={\bm x}_{0}$ and $\gamma(t)={\bm x}$. We also say that $D$ is a conservative field for $f$, denoted as $D_{f}$. 
\end{defin}
 As illustrated in \cite[Corollary 1]{bolte2021conservative}, for any given path-differentiable function $f$, $\partial f$ is a conservative field for $f$. Therefore, the concept of the conservative field can be regarded as an extension of the Clarke subdifferential. More importantly, when conservative field $D_f({\bm x})$ is convex-valued, ${\bm 0} \in \partial f({\bm x})$ implies that $0 \in D_f({\bm x})$. 

In fact, the class of path-differentiable functions is general enough to cover the most objective functions in real-world problems. The well-known Clarke regular functions mentioned in \cite{clarke1990optimization} and semi-algebraic functions mentioned in  \cite{ojasiewicz1965EnsemblesS} are all path-differentiable functions. In \cite{davis2020stochastic, bolte2021conservative}, the authors identify definable functions (whose graphs are definable in an o-minimal structure) as an important class of path-differentiable functions. Actually, most activation functions and loss functions in deep neural networks are definable functions, including sigmoid, softplus, ReLU, leaky ReLU, hinge loss, etc. 

Furthermore, it is important to highlight that definability is preserved under finite summation and composition \cite{davis2020stochastic}. Consequently, for any neural network constructed by definable blocks, its loss function is definable, thus it is a path-differentiable function. Additionally, it is worth noting that the Clarke subdifferential of definable functions is itself definable \cite{bolte2021conservative}. Hence, in the case of a neural network constructed by definable blocks, the conservative field corresponding to the AD algorithms is also definable. The following proposition demonstrates that the definability of both $f$ and its conservative field $D_f$  leads to the nonsmooth Morse–Sard property \cite{bolte2007clarke}. 

\begin{prop}[Theorem 5 in \cite{bolte2021conservative}]\label{prop:finite}
 Let $f$ be a path-differentiable function that admits $\D_f$ as its conservative field. Suppose both $f$ and $\D_f$ are definable over $\mathbb{R}^n$, then the set $\left\{{f({\bm x})}:{\bm 0}\in \D_f({\bm x}) \right\}$ is finite.
\end{prop}

Finally,  based on the concept of the conservative field,  we present the definition of the critical points for the optimization problem \eqref{Prob_DOP}.
\begin{defin}
    \label{Defin_critical_points}
    Let $f$ in \eqref{Eq_defin_f} be a path-differentiable function that admits $\D_f$ as its convex-valued conservative field, then we say ${\bm X}\in \Rnd$ is a $\D_f$-critical point of \eqref{Prob_DOP}, if ${\bm X} = \frac{1}{d}{\bm X}{\bm 1_{d}} {\bm 1}_d\tp$ and ${\bm 0} \in \D_f(\frac{1}{d}{\bm X}{\bm 1_{d}})$. Furthermore,  we say ${\bm X}\in \Rnd$ is a Clarke (or $\partial f$)-critical point of \eqref{Prob_DOP}, if ${\bm X} = \frac{1}{d}{\bm X}{\bm 1_{d}} {\bm 1}_d\tp$ and ${\bm 0} \in \partial f(\frac{1}{d}{\bm X}{\bm 1_{d}})$.
\end{defin}

\subsection{Differential inclusion and stochastic subgradient methods}
In this part, we introduce fundamental concepts and theories associated with differential inclusion, which is crucial for establishing the convergence properties for stochastic subgradient methods. 

\begin{defin}
Let $D: \mathbb{R}^n \rightrightarrows \mathbb{R}^n$ be a set-valued mapping. An absolutely continuous curve $\gamma: \mathbb{R}_{+} \rightarrow \mathbb{R}^n$ is called a solution (or trajectory) to the  differential inclusion 
\begin{equation}\label{eq:diff}
\frac{d {\bm x}}{d t} \in D({\bm x}),
\end{equation}
(or a trajectory of $D$) with initial point ${\bm x}$, provided that $\gamma(0)= {\bm x}$ and $\gamma'(t)\in D(\gamma(t))$
holds for almost every $t\in \mathbb{R}_{+}$. 
\end{defin}

\begin{defin}[Lyapunov function]
Let $\mathcal{B} \subset \mathbb{R}^n$ be a closed set. A continuous function $\psi: \mathbb{R}^n \rightarrow \mathbb{R}$ is referred to as a Lyapunov function for the differential inclusion \eqref{eq:diff} with the stable set $\mathcal{B}$, if for any solution $\gamma$ to \eqref{eq:diff} and any $t>0$,  it holds that 
\begin{equation*}
    \psi(\gamma(t))\leq \psi(\gamma(0)).
\end{equation*}
Moreover, whenever  $\gamma(0) \notin \mathcal{B}$, it holds for any $t> 0$ that 
\begin{equation*}
    \psi(\gamma(t)) < \psi(\gamma(0)).
\end{equation*}
\end{defin}

\begin{defin}
For any set-valued mapping $D: \mathbb{R}^n \rightrightarrows \mathbb{R}^n$ and any $\delta \geq0$, we denote $D^\delta: \mathbb{R}^n \rightrightarrows \mathbb{R}^n$ as
$D^\delta({\bm x})=\cup_{{\bm y} \in \mathbb{B}({\bm x}, \delta)}(D({\bm y})+\mathbb{B}({\bm 0}, \delta))$. 
\end{defin}

\begin{defin}[Perturbed solution]\label{def:pert}
An absolutely continuous curve $\gamma: \mathbb{R}_{+} \rightarrow \mathbb{R}^n $ is said to be a perturbed solution of differential inclusion \eqref{eq:diff}, if the following two conditions hold:
\begin{itemize}
    \item There exists a  locally integrable function $w: \mathbb{R}_{+} \rightarrow \mathbb{R}^n$, such that $$
        \lim_{t \rightarrow + \infty} \sup_{0\leq t_{0} \leq T}\left\|\int_t^{t+t_{0}} w(s) d s\right\|=0
        $$
    holds for any $T>0$.
    \item For above locally integrable $w$, there exists $\delta: \mathbb{R}_{+} \rightarrow \mathbb{R}$ such that $\lim_{t \rightarrow\infty} \delta(t)=0$ and
    $\dot{\gamma}(t)-w(t) \in D^{\delta(t)}(\gamma(t))$. 
 
\end{itemize}
\end{defin}

Now we consider an iterative sequence corresponding to the differential inclusion \eqref{eq:diff},
\begin{equation}\label{eq:sto_app}
{\bm x}_{k+1}= {\bm x}_{k} +\eta_{k}(D^{\delta_k}({\bm x}_{k})+ \xi_{k+1}),
\end{equation}
where $\{\eta_k\}$ is a non-summable positive sequence of step-sizes, $\{\delta_{k}\}$ is a nonnegative sequence, and $\xi_{k+1}$ is a random noise when evaluating $D({\bm x}_k)$. A continuous-time interpolated process $u: \mathbb{R}_{+} \rightarrow \mathbb{R}^n$  induced by \eqref{eq:sto_app} is defined as 
\begin{equation*}
u(\lambda_{\eta}(k) +s) = {\bm x}_{k} +  \frac{{\bm x}_{k+1}-{\bm x}_k}{\eta_{k}} s, \quad s \in [0,\eta_{k}).
\end{equation*}
Here, $\lambda_{\eta}(0):=0$ and $\lambda_{\eta}(k):= \sum_{i=0}^{k-1} \eta_{i}, k \geq 1$.

In what follows, we summarize and present several lemmas, which respectively describe the conditions under which the interpolated process of the sequence $\{{\bm x}_{k}\}$ is a perturbed solution of \eqref{eq:diff}, how the sequence $\{{\bm x}_{k}\}$ remains uniformly bounded, and a result on global convergence.

\begin{lem}[Extension of Proposition 1.3 in \cite{benaim2005stochastic}]\label{lem:extension_perturbed}
Let $D:\mathbb{R}^{n} \rightrightarrows \mathbb{R}^n$ be a locally bounded, graph-closed and convex-valued set-valued mapping. Assume that the following hold:
\begin{enumerate}
\item For any $T>0$, 
\begin{equation*}
\lim_{s \to \infty} \sup_{s\leq i \leq \Lambda_{\eta}(\lambda_{\eta}(s)+T)} \norm{\sum_{k=s}^{i} \eta_{k} \xi_{k+1}} =0,
\end{equation*}
where $\Lambda_{\eta}(t):= \sup\{k: \lambda_{\eta}(k)\leq t\}$.
\item $\{\delta_{k}\}$ is diminishing, i.e. $\lim_{k \to \infty}\delta_{k} =0$.
\item $\sup_{k\in \mathbb{N}}\|{\bm x}_{k}\|< + \infty$.
\end{enumerate}
Then the interpolated process of sequence $\{{\bm x}_{k}\}$ is a perturbed solution for \eqref{eq:sto_app}.    
\end{lem}

Before exhibiting the results on uniform boundedness and convergence of \eqref{eq:sto_app}, we provide an assumption about differential inclusion \eqref{eq:diff}, which is commonly used in the analysis of ODE-based approaches \cite{benaim2006dynamics, borkar2009stochastic,duchi2018stochastic}.

\begin{assumpt}\label{asp:preliminary}
\begin{enumerate}
\item There exists a locally Lipschitz continuous Lyapunov function $\psi: \Rn \to \bb{R}$  for the differential inclusion \eqref{eq:diff} with stable set $\ca{A}$.  
\item The set $\{\psi({\bm x}): {\bm x} \in \ca{A}\}$ is a finite subset of $\bb{R}$. 
\end{enumerate}     
\end{assumpt}

\begin{lem}[Theorem 3.6 in \cite{xiao2023convergence}]\label{lem:xiao_stability}
Let $D: \mathbb{R}^{n} \rightrightarrows \mathbb{R}^n$ be a locally bounded, graph-closed and convex-valued set-valued mapping, and $\mathcal{X}_0$ be any compact subset of $\mathbb{R}^n$. Suppose Assumption \ref{asp:preliminary} holds and Lyapunov function $\psi$ is coercive.

Then for any given $r > \max\{0, 4\sup_{x \in \mathcal{X}_0 \cup \mathcal{A}}\psi(x)\}$, there exist $\alpha >0$, $ T>0 $ such that whenever $\sup_{k\geq 0}\eta_{k}\leq \alpha$, $\sup_{k\geq 0}\delta_{k}\leq \alpha$, and $\sup_{s\leq i \leq \Lambda_{\eta}(\lambda_{\eta}(s)+T)} \|{\sum_{k=s}^{i} \eta_{k} \xi_{k+1}}\|\leq \alpha, \forall s\geq 0$, the sequence $\{x_k\}$ generated by \eqref{eq:sto_app}  with $x_0 \in \mathcal{X}_0$ is restricted in $\ca{L}_{r}:=\{{\bm x}\in \mathbb{R}^n: \psi({\bm x})\leq r\}$.  
\end{lem}

\begin{lem}[Summary of Theorem 3.6, Proposition 3.27 in \cite{benaim2005stochastic}]
\label{The_convergence_beniam}
Let $D: \Rn \rightrightarrows \Rn$ be a locally bounded, graph-closed, and convex-valued set-valued mapping. Suppose Assumption \ref{asp:preliminary} holds, and $\{{\bm x}_k\}$ is generated by \eqref{eq:sto_app} with any ${\bm x}_{0} \in \mathbb{R}^{n}$. If the interpolated process of $\{{\bm x}_k\}$ is a perturbed solution of \eqref{eq:diff}, then any cluster point of $\{{\bm x}_k\}$ generated by \eqref{eq:sto_app} lies in $\ca{A}$, and the sequence $\{\psi({\bm x}_k)\}$ converges to $\psi({\bm x}^*)$, for some ${\bm x}^* \in \mathcal{A}$. 
\end{lem}

Except for Lemma \ref{The_convergence_beniam}, similar results under slightly different conditions can be found in \cite{borkar2009stochastic,davis2020stochastic,duchi2018stochastic}, while some recent works \cite{bianchi2021closed,bolte2022long} focus on analyzing the convergence of \eqref{eq:sto_app} under relaxed conditions. Interested readers could refer to those works for details.

\section{A General Framework for Decentralized Stochastic Subgradient-based Methods}\label{sec:3}
In this section, we aim to establish the convergence properties of proposed framework \eqref{Eq_Framework} under two noise settings. To this end, it is crucial to ensure the uniform boundedness of iterates $\{{\bm Z}_{k}\}$. A large amount of literature \cite{castera2021inertial,ruszczynski2020convergence,davis2020stochastic, le2024nonsmooth} regards this condition as a prior assumption when analyzing the convergence of SGD-type algorithms in a single-agent setting. Nevertheless, more recent works \cite{bianchi2019constant, bianchi2022convergence, josz2024global, josz2023lyapunov, bolte2024inexact, xiao2023convergence} focus on rigorously establishing the uniform boundedness of the iterative sequence.

Section \ref{basic} introduces the fundamental definitions and assumptions used in our analysis, and then rigorously demonstrate the uniform boundedness and  asymptotic convergence for our framework \eqref{Eq_Framework}. In Sections \ref{subsec:1} and \ref{subsec:2}, we further discuss about the convergence properties of \eqref{Eq_Framework} when the evaluation noise is introduced by random reshuffling and with-replacement sampling, respectively.

\subsection{Basic assumptions and main results}\label{basic}

\eqref{relation_H_PHI} shows that the average update direction $\{\frac{1}{d}{\bm H}_{k}{\bm 1_{d}}\}$ approximates $\frac{1}{d}\sum_{i = 1}^d \Phi_i({\bm z}_{i,k})$. To constrain the behavior of $\{\Phi_i\}_{i=1}^d$, we impose a set of basic assumptions concerning both the framework in Section \ref{sec:dsm} and its associated continuous-time differential inclusion,
\begin{equation}\label{eq:diff_inclu}
    \frac{\mathrm{d}{\bm z}}{\mathrm{d}t} \in - \Phi({\bm z}):=  \conv \left( \frac{1}{d}\sum_{i = 1}^d \Phi_i({\bm z}) \right).
\end{equation}

\begin{assumpt}
\label{Assumption_framework}
\begin{enumerate}
\item The sequence of step-sizes $\{\eta_k\}$ is positive and satisfies $\sum_{k=0}^{\infty} \eta_{k} = +\infty$.
\item The sequences $\{{\bm H}_{k}\}$ and $\{\Xi_{k+1}\}$ are uniformly bounded in $\mathbb{R}^{m\times d}$, whenever $\{{\bm Z}_{k}\}$ is bounded in $\mathbb{R}^{m\times d}$.
\item There exists a locally Lipschitz continuous and coercive Lyapunov function $\psi: \Rm \to \bb{R}$ for the differential inclusion \eqref{eq:diff_inclu} with stable set $\ca{A}$.  
\item The set $\{\psi({\bm z}): {\bm z} \in \ca{A}\}$ is a finite subset of $\mathbb{R}$.
\end{enumerate}
\end{assumpt}

Assumption \ref{Assumption_framework}(1) allows for a flexible choice of the step-size, while Assumption \ref{Assumption_framework}(2) is reasonable and mild, which is commonly employed in various existing works \cite{bolte2021conservative,davis2020stochastic,xiao2023convergence,xiao2023adam, castera2021inertial}.

Assumption \ref{Assumption_framework}(3) clarifies the Lyapunov properties of a differential inclusion with respect to  $-\Phi$, which is standard in the literature \cite{benaim2005stochastic, borkar2008stochastic, bian2009subgradient, davis2020stochastic, bolte2021conservative, josz2023lyapunov} and can be satisfied by most appropriately selected set-valued mappings. Since each $\Phi_i$ is graph-closed and locally bounded, the convex-valued set-valued mapping $\Phi$ is also graph-closed and locally bounded. Then, the coercivity of $\psi$ in Assumption \ref{Assumption_framework}(3) indicates the compactness of any level set $\mathcal{L}_{r}:= \{x\in \mathbb{R}^{n}: \psi(x) \leq r\} $. 

Assumption \ref{Assumption_framework}(4) is similar to the Weak Sard’s condition in \cite{davis2020stochastic}, which stipulates that the set of values of $\psi$ at points outside the stable set is dense in $\mathbb{R}$. As demonstrated in \cite{bolte2021conservative},   Assumption \ref{Assumption_framework}(4) holds when $f$ is a definable function with the selection $\Phi:= \partial f$ and $\psi:= f$.

In the following, we present some basic notations and definitions throughout Section \ref{sec:3}. For any $M>0$, the stopping time $\tau_{M}$ is defined as 
\begin{equation*}
\tau_{M}:= \inf\{k\in \mathbb{N}: \|{\bm Z}_{k}\|>M\},
\end{equation*}
and the upper bound of noisy update direction before stopping time is given by 
\begin{equation*}
\ell_{M}:= \sup_{0\leq k\leq \tau_M} \norm{{\bm H}_{k}+\Xi_{k+1}},   
\end{equation*}
which is finite followed by Assumption \ref{Assumption_framework}(2).
The second largest singular value of mixing matrix ${\bm W}$ is 
\begin{equation*}
\lambda_{2}:= \norm{{\bm W}({\bm I}_d - \frac{{\bm 1}_d{\bm 1}_{d}^{\top}}{d})}.
\end{equation*}
\begin{defin}
For any positive sequence $\{\eta_{k}\}$, we say $\{\eta_{k}\}$ is $\alpha_{\text{ub}}$-\textit{upper-bounded} or is \textit{upper-bounded} by $\alpha_{\text{ub}}$, if $\sup_{k\in\mathbb{N}}\eta_{k}\leq \alpha. $ \end{defin} 

\begin{defin}
For any sequence of vectors (or matrices) $\{\xi_{k}\}$, and given constants $T, \alpha_{\text{ub}}>0$, we say $\{\xi_{k}\}$ is $(\alpha_{\text{ub}}, T, \{\eta_{k}\})$-\textit{controlled}, if for any $s\geq0$, we have 
\begin{equation*}
\sup_{s\leq i\leq \Lambda_{\eta}(\lambda_{\eta}(s)+T)}\left\|\sum_{k=s}^{i} \eta_{k}\xi_{k+1}\right\| \leq \alpha_{\text{ub}}.    
\end{equation*}
Moreover, we say $\{\xi_{k}\}$ is $(\alpha_{\text{ub}}, \alpha, T, \{\eta_{k}\})$-\textit{asymptotically controlled}, if it is $(\alpha_{\text{ub}}, T, \{\eta_{k}\})$-\textit{controlled} and 
\begin{equation*}
\limsup_{s \to \infty} \sup_{s\leq i\leq \Lambda_{\eta}(\lambda_{\eta}(s)+T)}\left\|\sum_{k=s}^{i} \eta_{k}\xi_{k+1}\right\| \leq \alpha.    
\end{equation*}
\end{defin}

With above tools in placement, we first present Proposition \ref{prop:avgZk+1Tau} to demonstrate the average iterative sequence of \eqref{Eq_Framework} before stopping time coincides with the update scheme \eqref{eq:sto_app}.

\begin{prop}\label{prop:avgZk+1Tau}
Suppose Assumption \ref{Assumption_framework} holds, and let $\mathcal{Z}_{0}$ be a compact subset of $\mathbb{R}^{m}$. The sequence $\{{\bm Z}_{k}\}$ is generated by \eqref{Eq_Framework}, and ${\bm Z}_{0}={\bm z}_{0}{\bm 1}_d^{\top}, {\bm z}_{0}\in \mathcal{Z}_{0}$. Then the recursion relation of $\{\frac{1}{d}{\bm Z}_{k \land \tau_M}{\bm 1}_d\}$ can be reformulated as 
\begin{equation}\label{eq:prop_relation}
\frac{1}{d}{\bm Z}_{(k+1)\land \tau_M}{\bm 1}_d \in \frac{1}{d}{\bm Z}_{k \land \tau_M}{\bm 1}_d -\eta_{k}[\Phi^{s_k}(\frac{1}{d}{\bm Z}_{k\land \tau_M}{\bm 1}_d)+\frac{1}{d}\Xi_{k+1}{\bm 1}_d]\mathbb{1}_{\tau_{M}>k},
\end{equation}  
where $s_k:=\delta_{k}+ \|{\bm Z}_{\perp, k\land \tau_{M}}\| \leq \delta_{k}+ \frac{\sup_{0\leq i \leq k\land \tau_M}\eta_{i}  \ell_M}{1-\lambda_2}$.
\end{prop}

\begin{proof}
From \eqref{Eq_Framework}, the update scheme for $\{{\bm Z}_{k \land \tau_{M}}\}$ can be expressed as
\begin{equation}\label{eq:Zk+1TauM01}
{\bm Z}_{(k+1)\land \tau_{M}}  = {\bm Z}_{k\land \tau_{M}}+ [{\bm Z}_{k\land \tau_{M}} ({\bm W}-{\bm I}_d) - \eta_{k}({\bm H}_{k}+{\Xi}_{k+1})]{\mathbb{1}_{\tau_{M}>k}},
\end{equation}
where $\mathbb{1}_{\tau_{M}>k}$ is the indicator function for the event $\{\tau_{M} > k\}$. Alternatively, the update can be written as
\begin{equation}\label{eq:Zk+1TauM02}
{\bm Z}_{(k+1)\land \tau_{M}}= {\bm Z}_{k\land (\tau_{M}-1)}{\bm W} - \eta_{k\land (\tau_{M}-1)}({\bm H}_{k\land (\tau_{M}-1)}+{\Xi}_{(k+1)\land \tau_{M}}).
\end{equation}
Define the consensus projection matrix ${\bm P}:= \frac{1}{d}{\bm 1_{d}} {\bm 1}_{d}^{\top}$ and the disagreement matrix ${\bm Z}_{\perp, k}:= {\bm Z}_{k} ({\bm I}_{d}-{\bm P})$. Then, the orthogonal decomposition holds: 
\begin{equation*}
{\bm Z}_{k}= {\bm Z}_{k} {\bm P} + {\bm Z}_{\perp, k}.
\end{equation*}
Intuitively, ${\bm Z}_{\perp,k}$ measures the dissimilarity among all agents' local variables at $k$-th iteration.

According to \eqref{eq:Zk+1TauM02}, simple calculations yield that
\begin{equation}\label{eq:ZktauM}
\begin{aligned}
\|{\bm Z}_{\perp, (k+1)\land \tau_{M}}\|  ={}& \norm{ ({\bm Z}_{k\land (\tau_{M}-1)}{\bm W} - \eta_{k\land (\tau_{M}-1)}({\bm H}_{k\land (\tau_{M}-1)}+{\Xi}_{(k+1)\land \tau_{M}}))  ({\bm I}_{d}-{\bm P})}\\
\leq{}& \norm{{\bm Z}_{k\land (\tau_{M}-1)} {\bm W}({\bm I}_d - {\bm P})} + \norm{\eta_{k\land (\tau_{M}-1)} ({\bm H}_{k\land (\tau_{M}-1)} + \Xi_{(k+1)\land \tau_{M}})({\bm I}_d - {\bm P})}\\
\leq{}&  \norm{{\bm Z}_{k\land (\tau_{M}-1)}({\bm I}_d - {\bm P})}\norm{{\bm W}({\bm I}_d - {\bm P})} + \norm{\eta_{k\land (\tau_{M}-1)} ({\bm H}_{k\land (\tau_{M}-1)} + \Xi_{(k+1)\land \tau_{M}})}\\
\leq{}& \lambda_{2} \norm{{\bm Z}_{\perp, k\land (\tau_{M}-1)}} + \norm{\eta_{k\land (\tau_{M}-1)} ({\bm H}_{k\land (\tau_{M}-1)} + \Xi_{(k+1)\land \tau_{M}})}\\
\leq {}&\sum_{i=0}^{k\land(\tau_{M}-1)} \lambda_2^{k\land(\tau_{M}-1)-i}\|\eta_{i}({\bm H}_{i}+\Xi_{i+1})\| + \lambda_2^{(k+1)\land \tau_{M}} \|{\bm Z}_{\perp, 0}\|\\
\leq  {}&\frac{\sup_{0\leq i \leq k\land (\tau_M-1)}\eta_{i} \ell_M}{1-\lambda_2},
\end{aligned}
\end{equation}
where the second inequality follows from the fact that $({\bm I}_d -{\bm P}){\bm W}({\bm I}_d -{\bm P}) = {\bm W}({\bm I}_d -{\bm P})$ and the last inequality is due to ${\bm Z}_{\perp, 0}= {\bm Z}_{0}-\frac{{\bm Z}_{0}{\bm 1}_{d}{\bm 1}_{d}^{\top}}{d} = 0$. As a result, we obtain that 
\begin{equation}\label{eq:perp_bound}
\|{\bm Z}_{\perp, l\land \tau_{M}}\|\leq \frac{\sup_{0\leq i \leq k\land \tau_M}\eta_{i} \ell_M}{1-\lambda_2}, \text{ for any }   l\leq k+1.
\end{equation}
We right-multiply $\frac{{\bm 1}_d}{d}$ to \eqref{eq:Zk+1TauM01} to get:     
\begin{equation}\label{eq:recursion_tm}
\frac{1}{d}{\bm Z}_{(k+1)\land \tau_M}{\bm 1}_d = \frac{1}{d}{\bm Z}_{k \land \tau_M}{\bm 1}_d -\eta_{k}(\frac{1}{d}{\bm H}_{k}{\bm 1}_{d}+\frac{1}{d}\Xi_{k+1}{\bm 1}_d)\mathbb{1}_{\tau_{M}>k}.
\end{equation}
Equation \eqref{relation_H_PHI} reveals that $\frac{1}{d}{\bm H}_{k}{\bm 1_{d}}  \in \conv\left( \frac{1}{d} \sum_{i = 1}^d \Phi_i^{\delta_k}({\bm z}_{i, k}) \right)$. According to Carathéodory’s Theorem, it admits the representation 
\begin{equation*}
\frac{1}{d}{\bm H}_{k}{\bm 1_{d}}=\sum_{j=0}^{m}c_j (\frac{1}{d}\sum_{i=1}^{d}{\bm y}_{i,k}^{j}),  \quad {\bm y}_{i,k}^{j}\in \Phi_i^{\delta_k}({\bm z}_{i, k}), c_j \in \Delta_{m}.
\end{equation*}

On the event $\{\tau_{M} > k\}$, combining with \eqref{eq:perp_bound} yields that 
\begin{equation*}
\begin{aligned}
&~\mathrm{dist}\left(\frac{1}{d}{\bm H}_{k}{\bm 1_{d}},  \conv\left(\frac{1}{d}\sum_{i = 1}^d \Phi_i(\frac{1}{d}{\bm Z}_{k\land \tau_M}{\bm 1}_d+\mathbb{B}({\bm 0}, \delta_{k}+ \|{\bm Z}_{\perp, k\land \tau_{M}}\|))\right)\right) \\
\leq &~ \sum_{j=0}^{m} c_{j}{} \mathrm{dist}\left(\frac{1}{d}\sum_{i=1}^{d}{\bm y}_{i,k}^{j},  \frac{1}{d}\sum_{i = 1}^d \Phi_i\left(\frac{1}{d}{\bm Z}_{k\land \tau_M}{\bm 1}_d+\mathbb{B}({\bm 0}, \delta_{k}+ \|{\bm Z}_{\perp, k\land \tau_{M}}\|)\right)\right) \\
\leq &~ \sum_{j=0}^{m} c_{j}\frac{1}{d}\sum_{i=1}^{d} \mathrm{dist}\left({\bm y}_{i,k}^{j}, \Phi_i\left(\frac{1}{d}{\bm Z}_{k\land \tau_M}{\bm 1}_d+\mathbb{B}({\bm 0}, \delta_{k}+ \|{\bm Z}_{\perp, k\land \tau_{M}}\|)\right)\right) \\
\leq &~ \sum_{j=0}^{m} c_{j}\frac{1}{d}\sum_{i=1}^{d} \delta_k = \delta_k, \\
\end{aligned}
\end{equation*}
where the first inequality is due to the Jensen's inequality and the second inequality is followed by the triangle inequality. That is to say, 
\begin{equation*}
\frac{1}{d}{\bm H}_{k}{\bm 1_{d}} \in \Phi^{s_k}\left(\frac{1}{d}{\bm Z}_{k\land \tau_M}{\bm 1}_d\right).
\end{equation*}
Plugging this relation into \eqref{eq:recursion_tm} completes the proof.
\end{proof}

{
\begin{rmk}
It is worth mentioning that the value of $\lambda_2$ (i.e., the second-largest eigenvalue of the mixing matrix ${\bm W}$) is a critical factor in quantifying the effect of network topology on the convergence of decentralized methods. As demonstrated in  \cite{boyd2004fastest}, $\lambda_2$ is widely regarded as a measurement for network connectivity. Proposition \ref{prop:avgZk+1Tau} establishes \eqref{eq:perp_bound}, demonstrating that $\lambda_2$ governs the rate at which $\bm Z_{\perp,k}$ converges to zero. Consequently, the value of $\lambda_2$ determines how rapidly the iterates ${\frac{1}{d}{\bm Z}_{(k+1)\land \tau_M}{\bm 1}_d}$ approach the trajectory of the differential inclusion \eqref{eq:diff_inclu}. For further analysis regarding the influence of network topology in smooth settings, we refer readers to some related works \cite{vogels2023beyond, song2022communication}.

\end{rmk}
}

Based on Proposition \ref{prop:avgZk+1Tau}, Theorem \ref{theo:globalstability} establishes that the sequence $\{{\bm Z}_{k}\}$ remains uniformly bounded, whenever step-size $\{\eta_{k}\}$, $\{\delta_{k}\}$ are upper-bounded by a sufficiently small $\alpha_{ub}>0$, and $\{\Xi_{k+1}\}$ is $(\alpha_{ub}, T, \{\eta_{k}\})$-controlled for some $T>0$. 

\begin{theo}\label{theo:globalstability}
Suppose Assumption \ref{Assumption_framework} holds, and let $\mathcal{Z}_{0}$ be a compact subset of $\mathbb{R}^{m}$. Then  for any $r> \max\{0, 4 \sup_{x \in \mathcal{Z}_0 \cup \mathcal{A}}\psi(x)\}$, there exist $\alpha_{ub}>0, T>0$ such that for any $\alpha_{ub}$-upper-bounded sequences $\{\eta_{k}\}$ and $\{\delta_{k}\}$, and any $(\alpha_{ub}, T, \{\eta_{k}\})$-controlled sequence $\{\Xi_{k+1}\}$, the sequence $\{{\bm Z}_{k}\}$  generated by \eqref{Eq_Framework} with ${\bm Z}_{0}={\bm z}_{0}{\bm 1}_d^{\top}, {\bm z}_{0}\in \mathcal{Z}_{0}$ is restricted in $\mathcal{L}_{2r}^{d}$.
\end{theo}

\begin{proof}
Recall from Proposition \ref{prop:avgZk+1Tau}, we have
\begin{equation*}
\frac{1}{d}{\bm Z}_{(k+1)\land \tau_M}{\bm 1}_d \in \frac{1}{d}{\bm Z}_{k \land \tau_M}{\bm 1}_d -\eta_{k}[\Phi^{s_k}(\frac{1}{d}{\bm Z}_{k\land \tau_M}{\bm 1}_d)+\frac{1}{d}\Xi_{k+1}{\bm 1}_d]\mathbb{1}_{\tau_{M}>k}.
\end{equation*}
For any $M> 0$, employing Lemma \ref{lem:xiao_stability}, we conclude that there exists $\alpha$, $T>0$  such that for  any $\alpha$-upper-bounded sequences $\{\eta_{k}\}$ and $\{s_k\}$, and any $(\alpha, T, \{\eta_{k}\})$-controlled sequence $\{\frac{1}{d}\Xi_{k+1}{\bm 1}_d \mathbb{1}_{\tau_M>k}\}$, the sequence $\{\frac{1}{d}{\bm Z}_{k \land \tau_M}{\bm 1}_d\}$ is restricted in $\mathcal{L}_{r}$. 

Let $L_{1}$ be the Lipschitz constant of $\psi$ in $\mathcal{L}_{r}+\mathbb{B}({\bm 0}, \frac{\alpha \ell_{M}}{1-\lambda_2})$, and $\alpha_{ub}:= \min\{\frac{\alpha}{2},  \frac{\alpha (1-\lambda_2)}{2 \ell_M}, \frac{r (1-\lambda_2)}{L_{1}\ell_{M}}\}$. It can be readily verified that 
\begin{equation*}
\psi({\bm z}_{i,k\land \tau_M}) -\psi(\frac{1}{d}{\bm Z}_{k \land \tau_M}{\bm 1}_d) \leq L_{1} \frac{\sup_{0\leq i \leq k\land \tau_M}\eta_{i} \ell_M}{1-\lambda_2}\leq r,
\end{equation*}
for any $i\in [d]$ and any $\alpha_{ub}$-upper-bounded $\{\eta_k\}$. 

Therefore, for any $M>0$,  any $\alpha_{ub}$-upper-bounded sequences $\{\eta_{k}\}$, $\{\delta_{k}\}$, and any $(\alpha_{ub}, T, \{\eta_{k}\})$-controlled sequence $\{\Xi_{k+1}\}$, it holds that $\{\eta_{k}\}$ and $\{s_{k}\}$ are $\alpha$-upper-bounded, $\{\frac{1}{d}\Xi_{k+1}{\bm 1}_d \mathbb{1}_{\tau_M>k}\}$ is $(\alpha, T, \{\eta_{k}\})$-controlled, thus $\{{\bm z}_{i, k\land \tau_{M}}\}$ is restricted in $\mathcal{L}_{2r}, \forall i \in [d]$. Taking $M > \sup\{\|{\bm Z}\|: {\bm Z}\in \mathcal{L}^d_{2r}\}$, from the definition
of the stopping time $\tau_{M}$, we can derive that $\tau_{M}=+\infty$, then ${\bm z}_{i,k}={\bm z}_{i, k\land \tau_{M}}\in \mathcal{L}_{2r}, \forall i \in [d], k\geq 0$, which completes the proof.  
\end{proof}

The following lemma from \cite[Lemma 3.1]{sundhar2010distributed} concerns the limit property of a convolution-like scalar sequence.

\begin{lem}\label{lem:2} 
Let $\{\gamma_{k}\}$ be a sequence of real numbers. If $0<a<1$ and $\lim_{k \rightarrow \infty} \gamma_k=\gamma$, then it holds that $\lim_{k \rightarrow \infty} \sum_{\ell=0}^k a^{k-\ell} \gamma_{\ell}=\frac{\gamma}{1-a}$.

\end{lem}

Combining Theorem \ref{theo:globalstability} and Lemma \ref{lem:2}, we demonstrate in Theorem \ref{theo:aysm} that the sequence $\{\bm Z_{k}\}$ asymptotically reaches consensus and converges to the stable set $\mathcal{A}$ with diminishing upper-bounded $\{\eta_{k}\}$, $\{\delta_{k}\}$ and $(\alpha_{ub}, 0, T, \{\eta_{k}\})$ asymptotically controlled noises.

\begin{theo}\label{theo:aysm}
Suppose Assumption \ref{Assumption_framework} holds, and let $\mathcal{Z}_{0}$ be a compact subset of $\mathbb{R}^{m}$, $\{{\bm Z}_{k}\}$  be generated by \eqref{Eq_Framework} with ${\bm Z}_{0}={\bm z}_{0}{\bm 1}_d^{\top}, {\bm z}_{0}\in \mathcal{Z}_{0}$. Then there exist $\alpha_{ub}>0, T>0$ such that for any $\alpha_{ub}$-upper-bounded diminishing sequences $\{\eta_{k}\}$ and $\{\delta_{k}\}$, and any $(\alpha_{ub}, 0, T, \{\eta_{k}\})$-asymptotically controlled sequence $\{\Xi_{k+1}\}$, it follows that \begin{equation*}
\lim_{k\to \infty} \mathrm{dist}({\bm Z}_{k}, \{{\bm Z}\in \mathbb{R}^{m \times d}: {\bm Z}= {\bm z} {\bm 1}^{\top}, {\bm z}\in \mathcal{A}\}) =0, 
\end{equation*}
and 
$\psi({\bm z}_{i,k})$ converges for each $i \in [d].$
\end{theo}

\begin{proof}
From Theorem \ref{theo:globalstability}, there exists $r$, $\alpha_{\text{ub}}, T>0$, for any $\alpha_{\text{ub}}$-upper-bounded $\{\eta_k\}$ and $\{\delta_{k}\}$, and any $(\alpha_{ub}, T, \{\eta_{k}\})$-controlled sequence $\{\Xi_{k+1}\}$, the sequence $\{{\bm Z}_{k}\} \subseteq \mathcal{L}_{2r}^{d}$. Taking sufficiently large $M$, \eqref{eq:prop_relation} becomes 
\begin{equation*}
\frac{1}{d}{\bm Z}_{k+1}{\bm 1}_d \in \frac{1}{d}{\bm Z}_{k}{\bm 1}_d -\eta_{k}[\Phi^{s_k}(\frac{1}{d}{\bm Z}_{k}{\bm 1}_d)+\frac{1}{d}\Xi_{k+1}{\bm 1}_d],
\end{equation*}  
where $s_k:= \delta_k + \|{\bm Z}_{\perp, k}\|$. Since $\lim_{k\to \infty}\|{\bm Z}_{\perp, k}\|\leq \lim_{k\to \infty}\sum_{i=0}^{k-1} \lambda_2^{k-1-i}\|\eta_{i}({\bm H}_{i}+\Xi_{i+1})\|=0$, we have $\lim_{k\to \infty} s_{k} =0$. 

Moreover, $(\alpha_{ub}, 0, T, \{\eta_{k}\})$-asymptotically controlled sequence $\{\Xi_{k+1}\}$ implies that for any $T'>0$,
\begin{equation*}
\lim_{s\to \infty}\sup_{s\leq i \leq \Lambda_{\eta}(\lambda_{\eta}(s)+T')}\norm{\sum_{k=s}^{i}\eta_k \frac{1}{d}\Xi_{k+1}{\bm 1}_d}=0,   
\end{equation*}
then it follows that the interpolated process of $\{\frac{1}{d}{\bm Z}_{k}{\bm 1}_d\}$ is a perturbed solution of \eqref{eq:diff_inclu}.

By combining Lemma \ref{The_convergence_beniam}, we can infer that, any cluster point of $\{\frac{1}{d}{\bm Z}_{k}{\bm 1}_d\}$ lies in $\mathcal{A}$ and the sequence $\{\psi(\frac{1}{d}{\bm Z}_{k}{\bm 1}_d)\}$ converges. As $\lim_{k\to \infty}\|{\bm Z}_{\perp, k}\|=0$, we conclude that, any cluster point of $\{{\bm z}_{i,k}\}$ lies in $\mathcal{A}$ and $\{\psi({\bm z}_{i,k})\}$ converges, for any $i\in [d]$. This completes the proof. 
\end{proof}

\subsection{Convergence with random reshuffling}\label{subsec:1}

Given a family of locally bounded, graph-closed set-valued mappings $\{\mathcal{U}_{i,j}: 1\leq i \leq d, 0\leq j \leq N-1\}$, we consider framework \eqref{Eq_Framework} with the following random reshuffled update scheme,
\begin{equation}\label{eq:reshuf}
\frac{1}{d}\Xi_{k+1}\mathbf{1}_d \in \frac{1}{d} \sum_{i=1}^{d} \mathcal{U}_{i, i_{k}}^{\rho_{k}}({\bm z}_{i,k})- \frac{1}{d}{\bm H}_{k}\mathbf{1}_d.
\end{equation}

One can see \eqref{eq:reshuf} as a detailed characterization of the relationship between the update direction and the iterates $\{{\bm z}_{i,k}\}$, thereby serving as a replacement for the role of \eqref{relation_H_PHI}. Hereafter, we make some assumptions on \eqref{eq:reshuf} and step-sizes $\{\eta_{k}\}$ in \eqref{Eq_Framework}.

\begin{assumpt}\label{asp:reshuffling}
\begin{enumerate}
\item For any $k_1, k_2 \in \mathbb{N}$, $\eta_{k_1} = \eta_{k_2}$ holds, if $\lfloor\frac{k_1}{N} \rfloor = \lfloor \frac{k_2}{N} \rfloor$.
\item The sequence of indexes $\{i_k\}$ is chosen from $\{0, \ldots, N-1\}$ by reshuffling. That is, $\{i_k : lN \leq k < (l +1)N\} = \{0 ,\ldots, N-1\}$ holds for any $l\in\mathbb{N}$.
\item The sequence $\{\rho_k\}$ is controlled by $\{\eta_{i}p({\bm Z}_{i})\}_{i\leq k}$, where $p: \mathbb{R}^{m \times d} \to \mathbb{R}$ is a locally bounded function. Moreover,  
$\lim_{k \to \infty} \rho_k =0$, whenever $\lim_{k\to \infty} \eta_{k} =0$ and  $\{{\bm Z}_{k}\}$ is bounded.
\item For any ${\bm z}\in \mathbb{R}^{m}$, $\frac{1}{N}\sum_{j=0}^{N-1} \mathcal{U}_{i,j}({\bm z}) \subseteq \Phi_{i}({\bm z})$.
\end{enumerate}    
\end{assumpt}

Indeed, we can treat each interval $\{k: lN\leq k < (l+1)N\}$ as an epoch. Assumption \ref{asp:reshuffling}(1) reveals that step-sizes $\{\eta_{k}\}$ remain unchanged within each epoch, which is a practical and mild condition. Assumption \ref{asp:reshuffling}(2) illustrates that the sequence of  indexes $\{i_k\}$ is drawn by reshuffling in each epoch. Assumption \ref{asp:reshuffling}(3) shows how the noisy averaged update scheme asymptotically approximates the Minkowski sum of random reshuffled set-valued mappings. The statement $\{\rho_k\}$ is controlled by $\{\eta_{i}p({\bm Z}_{i})\}_{i\leq k}$ means that $\rho_k$ is bounded by a linear combination of the components in $\{\eta_{i}p({\bm Z}_{i})\}_{i\leq k}$. In particular, $\rho_k$ can be an exponential moving average of $\{\eta_{i}p({\bm Z}_{i})\}_{i\leq k}$, given by $\sum_{i=0}^{k}\lambda_2^{k-i}\eta_{i}p({\bm Z}_{i})$. In Assumption \ref{asp:reshuffling}(4), we show the noisy update scheme at agent $i$ is randomly reshuffled from $N$ components of $\Phi_i$, which implies that $\frac{1}{Nd}\sum_{i=1}^{d}\sum_{k=0}^{N-1}\mathcal{U}_{i,i_k}({\bm z}) \subseteq \Phi({\bm z})$, building a bridge between the sequence $\{\frac{1}{d}{\bm Z}_{kN}{\bm 1}_d\}$ and the differential inclusion \eqref{eq:diff_inclu}.

It will be straightforward to see from Lemma \ref{lem:AVGtauM} that, in a certain sense, the averaged evaluation noise is ``hidden" over each epoch.

\begin{lem}\label{lem:AVGtauM}
Suppose Assumption \ref{Assumption_framework} and \ref{asp:reshuffling} hold, let $\mathcal{Z}_{0}$ be any compact set of $\mathbb{R}^{m}$ and ${\bm Z}_{0}={\bm z}_{0}{\bm 1}^{\top}, {\bm z}_{0}\in \mathcal{Z}_0$. $\{{\bm Z}_{k}\}$ is generated by \eqref{Eq_Framework} and the evaluation noise is introduced by random reshuffling \eqref{eq:reshuf}. Then for any $\hat{\varepsilon}\in (0, 1)$ and $M>0$, there exists a constant $\alpha>0$ such that for any $k\geq 0$, any $\eta_{kN}\in(0,\alpha]$, we have
\begin{equation}\label{eq:vanish}
\frac{1}{d}{\bm Z}_{(k+1)N \land \tau'_M}{\bm 1}_d \in  \frac{1}{d}{\bm Z}_{kN \land \tau'_M}{\bm 1}_d-N\eta_{kN} \Phi^{\hat{\varepsilon}}(\frac{1}{d}{\bm Z}_{kN \land \tau'_M}{\bm 1}_d)\mathbb{1}_{\tau'_M>kN}.
\end{equation}
where $\tau'_{M}:= \inf \{\lfloor \frac{k}{N} \rfloor N\in \mathbb{N}: \|{\bm Z}_{k}\|>M\}$.
\end{lem}

\begin{proof}
From Assumption \ref{asp:reshuffling}(3), for any $\hat{\varepsilon}\in (0, 1)$ and any $M>0$, there exists $\alpha_0>0$ such that $\rho_{k}< \frac{\hat{\varepsilon}}{3}$, when $\{\eta_{k}\}$ is upper bounded by $\alpha_0$, and $\|{\bm Z}_{k}\| \leq M+1$. We now introduce an additional notation:
\begin{equation*}
L_M := \sup\left\{ \norm{\mathcal{U}_{i,j}(x)}: \|x\|\leq M+1, i\in [d], 0\leq j\leq N-1 \right\}, 
\end{equation*}  
and for any $\hat{\varepsilon}\in (0, 1)$, define
\begin{equation*}
\alpha:= \min\left\{ \alpha_0, \frac{1}{ 1+(L_M+\frac{\hat{\varepsilon}}{3})}, \frac{\hat{\varepsilon}}{3N (1+L_M)+N\hat{\varepsilon}}, \frac{(1-\lambda_2)\hat{\varepsilon}}{3 \ell_M} \right\}.
\end{equation*}
Then for any $\eta_{kN} \in (0,  \alpha]$ and any $j \in \{0,\ldots, N-1\}$, it follows from \eqref{eq:ZktauM} that 
\begin{equation}\label{eq:ZperptauM_reshuffle}
\begin{aligned}
\|{\bm Z}_{\perp, (kN+j)\land \tau_M}\| & \leq \sum_{i=0}^{(kN+j-1)\land (\tau_M -1)}\lambda_2^{(kN+j-1)\land (\tau_M -1)-i}\|\eta_{i} ({\bm H}_{i}+ \Xi_{i+1})\| + \lambda_2^{(kN+j)\land \tau_M}\|{\bm Z}_{\perp, 0}\|\\
& \leq \frac{\sup_{0\leq i \leq (kN+j-1)\land (\tau_M-1)}\eta_{i} \ell_M}{1-\lambda_2} \leq \frac{\hat{\varepsilon}}{3},
\end{aligned}
\end{equation}
Without loss of generality, assuming $\tau_M\geq kN$, the update scheme \eqref{eq:reshuf} implies that   
\begin{equation}\label{eq:ZtauMreshuffling}
\begin{aligned}
           \norm{\frac{1}{d}{\bm Z}_{(kN+j)\land \tau_M}{\bm 1}_d - \frac{1}{d}{\bm Z}_{{kN\land \tau_M}}{\bm 1}_d} 
        & \leq \sum_{l = 0}^{j\land (\tau_M-kN-1)} \eta_{kN+l}\left\|\frac{1}{d}{\bm H}_{kN+l}{\bm 1}_d +\frac{1}{d}\Xi_{kN+l+1}{\bm 1}_d\right\|\\ 
        &\leq \sum_{l = 0}^{j\land (\tau_M-kN-1)} \eta_{kN+l} (L_M+\frac{\hat{\varepsilon}}{3}) \\
        & \leq \frac{(j+1)\hat{\varepsilon}(L_M+\frac{\hat{\varepsilon}}{3})}{3N (1+L_M)+N \hat{\varepsilon}}  \leq \frac{\hat{\varepsilon}}{3}.\\ 
\end{aligned}
    \end{equation}
As a result, we derive that $\sup_{0\leq j \leq  N-1} \norm{{\bm z}_{i, (kN+j)\land \tau_M} - \frac{1}{d}{\bm Z}_{{kN\land \tau_M}}{\bm 1}_d}  \leq \frac{2}{3}\hat{\varepsilon}$, $\forall i \in [d]$.  

Therefore, for any $k > 0$ and any $0\leq l \leq N-1$, the following inclusion holds: 
\begin{equation}\label{eq:recursion_avgZtauM}
\frac{1}{d}{\bm Z}_{(kN+l+1)\land \tau_M}{\bm 1}_d \in \frac{1}{d}{\bm Z}_{(kN+l)\land \tau_M}{\bm 1}_d - \eta_{kN+l}\frac{1}{d} \sum_{i=1}^d\left[\mathcal{U}_{i,i_{kN+l}}^{\frac{\hat{\varepsilon}}{3}}\left(\frac{1}{d}{\bm Z}_{{kN\land \tau_M}}{\bm 1}_d + \mathbb{B}({\bm 0},\frac{2\hat{\varepsilon}}{3})\right)\right]\mathbb{1}_{\tau_M>kN+l}.
\end{equation}
Moreover, the definition of $\tau'_{M}$ implies the following relations:
\begin{equation*}
\tau'_M>kN \Leftrightarrow \tau_M > kN+l, \quad \forall 0\leq l \leq N-1,
\end{equation*}
and 
\begin{equation*}
kN\land \tau'_{M}  \leq (kN+l+1)\land \tau_M,  \quad \forall 0\leq l \leq N-1.
\end{equation*}
Recursively utilizing \eqref{eq:recursion_avgZtauM} gives us that 
    \begin{equation*}
        \begin{aligned}
        & \frac{1}{d}{\bm Z}_{(k+1)N\land \tau'_M}{\bm 1}_d \\ \in{}& \frac{1}{d}{\bm Z}_{kN\land \tau'_M }{\bm 1}_d - \sum_{l = kN}^{(k+1)N-1} \eta_{kN} \frac{1}{d} \sum_{i=1}^d \left[\mathcal{U}_{i, i_l}(\frac{1}{d}{\bm Z}_{kN \land \tau'_M }{\bm 1}_d + \mathbb{B}({\bm 0},\hat{\varepsilon}))+\mathbb{B}({\bm 0},\frac{\hat{\varepsilon}}{3}) \right]\mathbb{1}_{\tau'_M >kN}\\
         \subseteq{} & \frac{1}{d}{\bm Z}_{kN\land \tau'_M}{\bm 1}_d - N\eta_{kN} \left( \frac{1}{d} \sum_{i=1}^d\frac{1}{N}\sum_{l=0}^{N-1}\ca{U}_{i,l}(\frac{1}{d}{\bm Z}_{kN \land \tau'_M}{\bm 1}_d + \mathbb{B}({\bm 0},\hat{\varepsilon}))  +\mathbb{B}({\bm 0},\frac{\hat{\varepsilon}}{3}) \right)\mathbb{1}_{\tau'_M >kN}\\
        \subseteq{}& \frac{1}{d}{\bm Z}_{kN\land \tau'_M}{\bm 1}_d - N\eta_{kN}  \Phi^{\hat{\varepsilon}} \left(\frac{1}{d}{\bm Z}_{kN\land \tau'_M}{\bm 1}_d \right)\mathbb{1}_{\tau'_M >kN}, 
        \end{aligned}
    \end{equation*}  
which is our desired result.
\end{proof}

Lemma \ref{lem:AVGtauM} illustrates that $\{ 
\frac{1}{d}{\bm Z}_{kN \land \tau'_M}{\bm 1}_d\}$ can be viewed as an inexact discretization of differential inclusion \eqref{eq:diff_inclu}. Then based on Lemma \ref{lem:xiao_stability}, we establish the uniform boundedness and asymptotic convergence of $\{{\bm Z}_{k}\}$ when adopting \eqref{eq:reshuf} under Assumption \ref{asp:reshuffling}. As the evaluation noise in \eqref{eq:vanish} vanishes over each epoch and is thus asymptotically controlled, Theorems \ref{thm:reshuffing} and \ref{theo:reshuff} can be viewed as special cases of Theorems \ref{theo:globalstability} and \ref{theo:aysm}.

\begin{theo}\label{thm:reshuffing}
Let $\mathcal{Z}_{0}$ be any compact set of $\mathbb{R}^{m}$. Suppose Assumption \ref{Assumption_framework} and \ref{asp:reshuffling} hold, and ${\bm Z}_{0} = {\bm z}_{0}{\bm 1}^{\top}_d,$ where ${\bm z}_{0}\in \mathcal{Z}_{0}$. Then for any $r> \max\{0, 4\sup_{{\bm z} \in \mathcal{Z}_0 \cup \mathcal{A}}\psi({\bm z})\}$, there exist $\alpha>0$, such that for any $\alpha$-upper-bounded  $\{\eta_{k}\}$, the sequence $\{{\bm Z}_{k}\}$ generated by \eqref{Eq_Framework} with \eqref{eq:reshuf}  is restricted in $\mathcal{L}_{2r}^d$.      
\end{theo}

\begin{proof}
For any $r> \max\{0, 4\sup_{{\bm z} \in \mathcal{Z}_0 \cup \mathcal{A}}\psi({\bm z})\}$,
we set $L_{2}$ to be the Lipschitz constant of $\psi$ in $\mathcal{L}_{r}+\mathbb{B}(0,r)$. According to Lemma \ref{lem:xiao_stability}, there exists $\alpha_{1}\in (0, \frac{3r}{2L_2})$ such that for any $\alpha_{1}$-upper bounded $\{\hat{\eta}_{k}\}$ and $\{\hat{\delta}_{k}\}$, the sequence generated by  
\begin{equation*}
\hat{\bm x}_{k+1} \in \hat{\bm x}_k - \hat{\eta}_k \Phi^{\hat{\delta}_{k}}(\hat{\bm x}_k), 
\end{equation*}
satisfies $\{\hat{\bm x}_{k}\}\subseteq \mathcal{L}_{r}$.  Given the iterative sequence $\{\frac{1}{d}{\bm Z}_{kN \land \tau'_M}{\bm 1}_d\}$, Lemma \ref{lem:AVGtauM} guarantees that for any $\delta_{kN} \in (0, \alpha_1)$, there exists $\alpha\in (0, \frac{\alpha_1}{N})$ such that whenever $\{\eta_{kN}\}$ is $\alpha$-upper bounded, the following inclusion holds:
\begin{equation*}
\frac{1}{d}{\bm Z}_{(k+1)N \land \tau'_M}{\bm 1}_d \in  \frac{1}{d}{\bm Z}_{kN \land \tau'_M}{\bm 1}_d-N\eta_{kN} \Phi^{\delta_{kN}}(\frac{1}{d}{\bm Z}_{kN \land \tau'_M}{\bm 1}_d)\mathbb{1}_{\tau'_M>kN}.
\end{equation*}
Hence, by Lemma \ref{lem:xiao_stability},  the sequence $\{\frac{1}{d}{\bm Z}_{kN \land \tau'_M}{\bm 1}_d\} \subseteq \mathcal{L}_r$. 

Without loss of generality, we assume $\tau_M> kN$. Following arguments analogous to those in \eqref{eq:ZperptauM_reshuffle}-\eqref{eq:ZtauMreshuffling}, we deduce that 
\begin{equation*}
\begin{aligned}
&\norm{{\bm z}_{i, (kN+j)\land \tau_M}- \frac{1}{d}{\bm Z}_{kN\land \tau'_M}{\bm 1}_d} \\
\leq{}& \norm{{\bm z}_{i, (kN+j)\land \tau_M} - \frac{1}{d}{\bm Z}_{(kN+j)\land \tau_M}{\bm 1}_d} + \norm{\frac{1}{d}{\bm Z}_{(kN+j)\land \tau_M}{\bm 1}_d- \frac{1}{d}{\bm Z}_{kN}{\bm 1}_d} \\
<{}& \frac{2\alpha_1}{3} < \frac{r}{L_2}.
\end{aligned}
\end{equation*}
Consequently,
\begin{equation}\label{eq:psi_zitauM}
\psi({\bm z}_{i, (kN+j)\land \tau_M}) \leq \psi(\frac{1}{d}{\bm Z}_{kN\land \tau'_M}{\bm 1}_d) + L_2\norm{{\bm z}_{i, (kN+j)\land \tau_M}- \frac{1}{d}{\bm Z}_{kN\land \tau'_M}{\bm 1}_d} \leq  2r, 
\end{equation}
for any $i\in[d], j\in \{0,\ldots, N-1\}$. 

To finish the proof, it suffices to  demonstrate that $\tau_M = +\infty$, for any $M>\sup\{\|{\bm Z}\|: {\bm Z}\in \mathcal{L}_{2r}^d\}$. If not, then there exists $k_1\in \mathbb{N}$ with $k_1 N \leq \tau_M < (k_1+1)N$. By \eqref{eq:psi_zitauM}, this implies ${\bm Z}_{\tau_M}\in \mathcal{L}_{2r}^d$. However, since $M>\sup\{\|{\bm Z}\|: {\bm Z}\in \mathcal{L}_{2r}^d\}$, the definition of $\tau_M$ forces $\norm{{\bm Z}_{\tau_M}}> M$, contradicting ${\bm Z}_{\tau_M}\in \mathcal{L}_{2r}^d$.

\end{proof}

\begin{theo}\label{theo:reshuff}
Suppose Assumption \ref{Assumption_framework} and \ref{asp:reshuffling} hold. Let $\mathcal{Z}_{0}$ be any compact set of $\mathbb{R}^{m}$, $\{{\bm Z}_{k}\}$ be generated by $\eqref{Eq_Framework}$ with \eqref{eq:reshuf}, and ${\bm Z}_{0} = {\bm z}_{0}{\bm 1}^{\top}_d, {\bm z}_{0}\in \mathcal{Z}_{0}$. Then there exists $\alpha_{\text{rr}}>0$, such that for any $\alpha_{\text{rr}}$-upper-bounded diminishing $\{\eta_{k}\}$ in \eqref{Eq_Framework}, it follows that 
\begin{equation*}
\lim_{k\to \infty} \mathrm{dist}({\bm Z}_{k}, \{{\bm Z}\in \mathbb{R}^{m \times d}: {\bm Z}= {\bm z} {\bm 1}^{\top}, {\bm z}\in \mathcal{A}\}) =0, 
\end{equation*}
and 
$\psi({\bm z}_{i,k})$ converges for each $i \in [d].$
\end{theo}

\begin{proof}
Recall from Theorem \ref{thm:reshuffing} that there exists $\alpha>0$, if $\{{\bm Z}_{k}\}$ iterates by \eqref{Eq_Framework} and \eqref{eq:reshuf} with $\alpha$-upper-bounded sequence $\{\eta_{kN}\}$ and initial point ${\bm Z}_{0} = {\bm z}_{0}{\bm 1}^{\top}_d, {\bm z}_{0}\in \mathcal{Z}_{0}$, then $\{{\bm Z}_{k}\}$ is restricted in $\mathcal{L}_{2r}^d$ for some $r>0$. Moreover, by Assumption \ref{asp:reshuffling}(3), we can further select a $\alpha_{\text{rr}} \in (0, \alpha)$ such that $\rho_k\leq 1$ for $\alpha_{\text{rr}}$-upper-bounded sequence $\{\eta_{kN}\}$.

We now analyze the case when $\{\eta_{kN}\}$ is upper bounded by $\alpha_{\text{rr}}$. With a slight abuse of notation, let $L:= \sup\left\{ \norm{\mathcal{U}_{i,j}(x)}: x \in \mathcal{L}_{2r}+\mathbb{B}({\bm 0},1), i\in [d], 0\leq j\leq N-1 \right\}$. 
Then for any $i\in [d]$ and $j \in \{0, \ldots, N-1\}$, we have
\begin{equation}\label{eq:zjzkn1}
\begin{aligned}
\norm{{\bm z}_{i, kN+j} - \frac{1}{d}{\bm Z}_{kN}{\bm 1}_d}   & \leq \norm{\frac{1}{d}{\bm Z}_{kN+j}{\bm 1}_d - \frac{1}{d}{\bm Z}_{kN}{\bm 1}_d} + \norm{{\bm Z}_{\perp, kN+j}}\\
& \leq \sum_{l = 0}^{j} \eta_{kN+l}\left\|\frac{1}{d}{\bm H}_{kN+l}{\bm 1}_d +\frac{1}{d}\Xi_{kN+l+1}{\bm 1}_d\right\|+ \norm{{\bm Z}_{\perp, kN+j}}\\ 
&\leq \sum_{l = 0}^{j} \eta_{kN+l} (L+1) +  \norm{{\bm Z}_{\perp, kN+j}}\\
& \leq N(L+1)\eta_{kN} + \max_{0\leq j \leq N-1}\left\{\norm{{\bm Z}_{\perp, kN+j}}\right\} \\
% & \leq \frac{2}{3} \hat{\varepsilon}_{kN}.
\end{aligned}
\end{equation}

According to Assumption \ref{Assumption_framework}(2), $\|{\bm H}_{k}+\Xi_{k+1}\|$ is uniformly bounded as $\{{\bm Z}_{k}\}$ is restricted in $\mathcal{L}_{2r}^d$.  Plugging Lemma \ref{lem:2}, we have 
\begin{equation}\label{eq:consensus_thm1}
\lim_{k\to \infty}\|{\bm Z}_{\perp, kN+j}\|\leq \lim_{k\to \infty}\sum_{i=0}^{kN+j} \lambda_2^{kN+j-i}\|\eta_{i}({\bm H}_{i}+\Xi_{i+1})\| =0.
\end{equation}

Define $\hat{\varepsilon}_{kN}:=  N(L+1)\eta_{kN} + \max_{0\leq j \leq N-1}\{\|{\bm Z}_{\perp, kN+j}\|\} + \max_{0\leq j\leq N-1}\{\rho_{kN+j}\} $. Similar to the arguments in Lemma \ref{lem:AVGtauM}, we can achieve that
\begin{equation}\label{eq:upsc}
\frac{1}{d}{\bm Z}_{(k+1)N}{\bm 1}_d \in \frac{1}{d}{\bm Z}_{kN}{\bm 1}_d - N\eta_{kN}  \Phi^{\hat{\varepsilon}_{kN}} \left(\frac{1}{d}{\bm Z}_{kN}{\bm 1}_d \right) . 
\end{equation}
Since $\{\eta_{kN}\}$ is diminishing,  Assumption \ref{asp:reshuffling}(3) indicates that $\hat{\varepsilon}_{kN} \to 0$ as $k\to \infty$. 

Indeed, the update scheme \eqref{eq:upsc} corresponds to  an inexact subgradient descent method corresponding for the differential inclusion \eqref{eq:diff_inclu}. Applying Lemma \ref{lem:extension_perturbed}, we obtain that the interpolated process of $\{\frac{1}{d}{\bm Z}_{kN}{\bm 1}_d\}$ is a perturbed solution of \eqref{eq:diff_inclu}. Then Lemma \ref{The_convergence_beniam} illustrates that 
\begin{equation*}
\lim_{k\to \infty} \mathrm{dist}(\frac{1}{d}{\bm Z}_{kN}{\bm 1}_d , \mathcal{A}) =0, \textit{ and } \psi(\frac{1}{d}\sum_{i=1}^{d}{\bm z}_{i,kN}) \textit{ converges},
\end{equation*}
and \eqref{eq:zjzkn1} further implies that 
\begin{equation}\label{eq:singletoA}
\lim_{k\to \infty} \mathrm{dist}({\bm z}_{i,kN+j}, \mathcal{A}) =0, \textit{ and } \psi({\bm z}_{i,kN+j}) \textit{ converges, for any }  0\leq j \leq N-1.
\end{equation}
Combining \eqref{eq:consensus_thm1} and \eqref{eq:singletoA} leads to the desired result.    
\end{proof}

\subsection{Convergence under with-replacement sampling}\label{subsec:2}

In this part, we handle the circumstance that the evaluation noise $\{\Xi_{k+1}\}$ arises from with-replacement sampling. Let  $(\Omega, \mathcal{F}, \mathbb{P})$ be the probability space. We consider the following expression of $\{\Xi_{k+1}\}$:
\begin{equation}\label{eq:chi}
\Xi_{k+1} = \chi({\bm Z}_{k}, \xi_{k}) = [\chi_{1}({\bm z}_{1,k}, \xi_{1,k}), \ldots,  \chi_{d}({\bm z}_{d,k}, \xi_{d, k})]
\end{equation}
where each $\chi_i: \mathbb{R}^{m} \times \Omega \to \mathbb{R}^{m}$ is a Borel function, $\{\xi_{i, k}\}_{k\geq 0}$ is a sequence of sampling data picked in $\Omega$ at agent $i$. With-replacement sampling means that the sampling data $\{\xi_{i, k}\}_{k\geq 0}$ are  randomly drawn from the same probability distribution $\mathcal{P}_{i}$, 
which allows  $\xi_{i, k}$ to occur multiple times.

We stipulate the following assumptions on framework \eqref{Eq_Framework} with \eqref{eq:chi}.

\begin{assumpt}
\label{Assumption_stochastic}
	\begin{enumerate}
		\item For any ${\bm Z} \in \mathbb{R}^{m\times d}$,  $\bb{E}_{\xi}[\chi({\bm Z}, \xi)] = 0$. 
		Moreover, there exists a locally bounded function $q: \mathbb{R}^{m\times d} \to \bb{R}_+$, such that 
        \begin{equation*}\norm{\chi({\bm Z}, \xi)} \leq q({\bm Z}), \quad \text{ for almost every } \xi \in \Omega^d.
        \end{equation*}
		\item The sequence of step-sizes $\{\eta_k\}$ satisfies
\begin{equation*}
			\sum_{k = 0}^{\infty} \eta_k = +\infty, \quad \lim_{k\to\infty} \eta_k \log(k) = 0. 
		\end{equation*}
	\end{enumerate}    
\end{assumpt}

Recall from Proposition \ref{prop:avgZk+1Tau}, we have
\begin{equation}\label{eq:avgTam}
\frac{1}{d}{\bm Z}_{(k+1)\land \tau_M}{\bm 1}_d \in \frac{1}{d}{\bm Z}_{k \land \tau_M}{\bm 1}_d -\eta_{k}[\Phi^{s_k}(\frac{1}{d}{\bm Z}_{k\land \tau_M}{\bm 1}_d)+\frac{1}{d}\Xi_{k+1}{\bm 1}_d]\mathbb{1}_{\tau_{M}>k}.
\end{equation}  
Under Assumption \ref{Assumption_stochastic}(1), the sequence $\{\frac{1}{d}\Xi_{k+1}{\bm 1}_d \mathbb{1}_{\tau_M> k}\}$ is a uniformly bounded martingale difference sequence with respect to $\mathcal{F}_{k}:= \sigma(\{{\bm Z}_{i}, i\leq k\})$. With the $o(1/\log(k))$ step-sizes $\{\eta_{k}\}$ and any given $\iota>0$, the following lemma illustrates that sequence $\{\frac{1}{d}\Xi_{k+1}{\bm 1}_d \mathbb{1}_{\tau_M> k}\}$ is $(\iota, T, \eta_{k})$-controlled with arbitrary high probability. 

\begin{lem}\label{lem:high_prop_control}
Suppose Assumption \ref{Assumption_framework} and \ref{Assumption_stochastic} hold. Then for any $\varepsilon, T, \iota >0$, there exists $\beta_0>0$, such that for any $\beta_0$-upper-bounded sequence $\{\eta_{k}\}$, it holds that 
\begin{equation*}\label{withreplace1}
\mathbb{P}\left(\left\{\exists s\geq 0, \text{ such that } \sup_{s\leq i \leq \Lambda_{\eta}(\lambda_{\eta}(s)+T)} \norm{\sum_{k=s}^{i} \eta_{k} \frac{1}{d}\Xi_{k+1}{\bm 1}_d \mathbb{1}_{\tau_M>k}}\geq \iota \right\} \right)   \leq \varepsilon.
\end{equation*}
    
\end{lem}

Lemma \ref{lem:high_prop_control} is a direct corollary of \cite[Proposition 4.2]{benaim2006dynamics} and \cite[Proposition A.5]{xiao2023convergence}, hence we omit its proof for simplicity. Based on this lemma, we present Theorem \ref{theo:withreplace_thm1} and \ref{theo:with} to show the uniform boundedness and convergence properties of framework \eqref{Eq_Framework} together with evaluation noise \eqref{eq:chi}.

\begin{theo}\label{theo:withreplace_thm1}
Suppose Assumption \ref{Assumption_framework} and \ref{Assumption_stochastic} hold, and let $\mathcal{Z}_{0}$ be any compact set of $\mathbb{R}^{m}$, ${\bm Z}_{0}={\bm z}_{0}{\bm 1}_d^{\top}, {\bm z}_{0}\in \mathcal{Z}_{0}$.
The sequence $\{{\bm Z}_{k}\}$ is generated by \eqref{Eq_Framework}, and the evaluation noise is introduced by with-replacement sampling \eqref{eq:chi}.
Then for any $r> \max\{0, 4\sup_{x\in \mathcal{Z}_0\cup \mathcal{A}}\psi(x)\}$ and  any $\varepsilon \in (0,1)$, there exists $\beta>0$, such that for any $\beta$-upper-bounded sequence $\{\eta_{k}\}$ and $\{\delta_{k}\}$,  $\{{\bm Z}_{k}\}$  is restricted in $\mathcal{L}_{2r}^{d}$ with probability at least $1-\varepsilon$. 
\end{theo}

Let $\Omega_0$ denote the event that the sequence $\{\frac{1}{d}\Xi_{k+1}{\bm 1}_d \mathbb{1}_{\tau_M>k}\}$ is $(\iota, T, \eta_k)$-controlled. 
Lemma \ref{lem:high_prop_control} gives us that 
\begin{equation*}
\mathbb{P}(\Omega_0)   \geq 1-\varepsilon.
\end{equation*}
Applying Theorem \ref{theo:globalstability} over event $\Omega_0$ directly yields Theorem \ref{theo:withreplace_thm1}. Hence, we omit the proof for clarity.

\begin{theo}\label{theo:with}
Suppose Assumption \ref{Assumption_framework} and \ref{Assumption_stochastic} hold, $\{\delta_k\}$ in \eqref{relation_H_PHI} satisfies $\lim_{k\to \infty} \delta_k = 0$. Let $\mathcal{Z}_{0}$ be any compact set of $\mathbb{R}^{m}$, and ${\bm Z}_{0}= {\bm z}_{0}{\bm 1}_{d}^{\top}, {\bm z}_{0}\in \mathcal{Z}_0$.  Then for any $\varepsilon\in (0,1)$, there exists $\beta>0$, such that for any $\{\eta_{k}\}, \{\delta_k\}$ upper-bounded by $\beta$, and any $\{{\bm Z}_{k}\}$ generated by \eqref{Eq_Framework} with evaluation noise \eqref{eq:chi}, we have 
\begin{equation*}
\mathbb{P}\left(\lim_{k\to \infty} \mathrm{dist}({\bm Z}_{k}, \{{\bm Z}\in \mathbb{R}^{m \times d}: {\bm Z}= {\bm z} {\bm 1}^{\top}, {\bm z}\in \mathcal{A}\}) =0 \right) \geq 1-\varepsilon,   
\end{equation*}
and $\mathbb{P}(\psi({\bm z}_{i,k}) \textit{ converges})\geq 1-\varepsilon$ for any $i\in [d]$.
\end{theo}

\begin{proof}
Denote the event $\Omega_1:=\{\{{\bm Z}_{k}\} \subseteq \mathcal{L}_{2r}^{d}\}$. By Theorem \ref{theo:globalstability} and \ref{theo:withreplace_thm1}, we have 
\begin{equation*}
\Omega_1 \supseteq \Omega_0, \quad 
\mathbb{P}(\Omega_1) \geq \mathbb{P}(\Omega_0)  \geq 1-\varepsilon.
\end{equation*}
Notice that $\{\frac{1}{d}\Xi_{k+1}{\bm 1}_d\}$ is a uniformly bounded martingale difference sequence over $\Omega_1$ and $\eta_{k} \sim o(1/\log(k))$. Invoking \cite[Proposition 1.4, Remark 1.5]{benaim2005stochastic}, we obtain that 
\begin{equation*}\label{eq:noisecontrolled}
\lim_{s\to \infty}\sup_{s\leq i \leq \Lambda_{\eta}(\lambda_{\eta}(s)+T)}\norm{\sum_{k=s}^{i}\eta_k \frac{1}{d}\Xi_{k+1}{\bm 1}_d}=0  
\end{equation*}
holds almost surely over both $\Omega_1$ and $\Omega_0$, for any $T>0$. 

Employing Theorem \ref{theo:aysm} over  $\Omega_0$, we completes the proof. 
    
\end{proof}

\section{Convergence Properties of Decentralized Stochastic Subgradient-based Methods}\label{sec:4}

In this section, we demonstrate the flexibility of our proposed framework \eqref{Eq_Framework} by showing that it encloses a wide range of popular decentralized stochastic subgradient-based methods, including DSGD, decentralized generalized SGD with momentum (DGSGD-M), and DSGD with gradient-tracking technique (DSGD-T).  More importantly, our theoretical analysis provides convergence guarantees for these methods in the minimization of nonsmooth path-differentiable functions without Clarke regularity. Additionally, Theorem \ref{thm:DSGD2}, \ref{thm:DSGDM2}  establish that DSGD and DGSGD-M can avoid spurious critical points.

To facilitate analysis, we need the following assumptions on problem \eqref{Prob_DOP}. 
\begin{assumpt}\label{Assumption_obj}
\begin{enumerate}
\item For each $i \in [d]$ and ${\bm s}_{l}\in \mathcal{S}_{i}$,  $F_i(\cdot; {\bm s}_{l})$ is a definable function, which admits a definable conservative field $D_{F_i (\cdot; {\bm s}_{l})}$.

\item The objective function $f$ is coercive. 
\end{enumerate}
\end{assumpt}

As discussed in Section \ref{sec:conser}, the class of definable functions is wide enough to enclose the loss functions of common nonsmooth neural networks. Moreover, \cite{bolte2021conservative} demonstrates that there exists a definable conservative field for any definable loss function constructed by definable blocks, within which the results yielded by Automatic Differentiation (AD) algorithms are contained. Henceforth, Assumption \ref{Assumption_obj}(1) is mild and reasonable in practical scenarios.

For conciseness, we denote $D_{F_{i,l}}:= D_{F_i(\cdot; {\bm s}_{l})}$. According to \cite[Corollary 4]{bolte2021conservative} and Assumption \ref{Assumption_obj}(1), 
\begin{equation}\label{eq:defin_cons}
    D_{f_i}({\bm x}) := \frac{1}{|\mathcal{S}_{i}|}\sum_{l=1}^{|\mathcal{S}_{i}|}  D_{F_{i,l}}({\bm x}).
\end{equation}
is a definable conservative field for $f_{i}$. Recall that $f({\bm x}) = \frac{1}{d} \sum_{i = 1}^d  f_i({\bm x})$ defined in \eqref{Eq_defin_f}, we can choose its definable conservative field $\D_f$ as
\begin{equation}
    D_f({\bm x}) :=  \conv\left(\frac{1}{d} \sum_{i = 1}^d D_{f_i}({\bm x}) \right). 
\end{equation}

\subsection{Decentralized stochastic subgradient descent}
In this subsection, we establish the global convergence for decentralized stochastic subgradient descent (DSGD) by showing that it fits into our proposed framework \eqref{Eq_Framework}. The detailed algorithm of DSGD is presented in Algorithm 1. Intuitively, for each $i \in [d]$, the consensus term $\sum_{j\in \mathcal{N}_{i}}{\bm W}(i,j) {\bm x}_{j,k}$ enforces  ${\bm x}_{i,k}$ towards the mean over all the agents (i.e., $\frac{1}{d} \sum_{i \in [d]} {\bm x}_{i,k}$). Moreover, as ${\bm d}_{i,k}$ refers to the (mini-batch) stochastic subgradient of $f_i$ at ${\bm x}_{i,k}$, the term $-\eta_k {\bm d}_{i,k}$ can be viewed as a single SGD step to minimize $f_i$ in the $i$-th agent.

\begin{algorithm}
\caption{DSGD for solving \eqref{Prob_DOP}. }
\label{alg:dsgd}
\begin{algorithmic}[1] 
\Require Initial point ${\bm x}_0 \in \mathbb{R}^{n}$ and a mixing matrix ${\bm W}$.
\For{all $i\in [d]$ in parallel}
\State Set $k \gets 0$. Initialize ${\bm x}_{i,k}= {\bm x}_0 $
\While{not terminated}
    \State Randomly select a mini-batch $\mathcal{B}_{i, k} \subseteq \mathcal{S}_{i}$;
    \State Compute ${\bm d}_{i,k} \in \frac{1}{|\mathcal{B}_{i,k}|} \sum_{{\bm s}_l \in \mathcal{B}_{i,k}}  D_{F_{i,l}}({\bm x}_{i,k})$;
    \State Choose the step-size $\eta_k$;
    \State Communicate and update the local variables 
    $${\bm x}_{i,k+1} = \sum_{j\in \mathcal{N}_{i}}{\bm W}(i,j) {\bm x}_{j,k} - \eta_k {\bm d}_{i,k};$$
    \State Set $k \gets k + 1$;
\EndWhile
\EndFor
\State \Return ${\bm X}_k:= [{\bm x}_{1,k}, \ldots, {\bm x}_{d, k}]$;
\end{algorithmic}
\end{algorithm}

In Algorithm \ref{alg:dsgd}, the evaluation noise is introduced by the selection of mini-batch $\mathcal{B}_{i,k}$. Now, we discuss two different scenarios for selecting mini-batches and step-sizes.

\begin{assumpt}[Random reshuffling]\label{Assumption_app1}
\mbox{}
    \begin{enumerate}
    \item $\mathcal{B}_{i,k}$ is selected from set $\mathcal{S}_i$ by random reshuffling, meaning that for every epoch index $l \in \mathbb{N}$, the following holds:
    \begin{equation*}
\bigcup_{k=lN}^{(l+1)N-1}\mathcal{B}_{i,k}= \mathcal{S}_i.
    \end{equation*}
    \item Let $\{\hat{\eta}_{k}\}$ be a prefixed sequence such that $\sum_{k=0}^{\infty} \hat{\eta}_{k} = +\infty,$  and $\lim_{k\to \infty}\hat{\eta}_{k}=0$. Moreover, for any $j_1, j_2 \in \mathbb{N}$, $\hat{\eta}_{j_1} = \hat{\eta}_{j_2}$ holds, if $\lfloor\frac{j_1}{N} \rfloor = \lfloor \frac{j_2}{N} \rfloor$. We set $\eta_{k}= c \hat{\eta}_{k}$, $c>0$ for all $k\in\mathbb{N}$.      
    \end{enumerate}
\end{assumpt}

\begin{assumpt}[With-replacement sampling]\label{Assumption_app2}
\mbox{}
\begin{enumerate}
\item $\mathcal{B}_{i,k}$ is selected from set $\mathcal{S}_i$ by with-replacement sampling for each $k\in \mathbb{N}$.
\item Let $\{\hat{\eta}_{k}\}$ be a prefixed sequence satisfying 
        \begin{equation*}
       \sum_{i = 0}^{\infty} \hat{\eta}_k = +\infty, \text{ and } \lim_{k\to \infty} \hat{\eta}_k \log(k) = 0. 
        \end{equation*}
We set $\eta_{k}= c \hat{\eta}_{k}$, $c>0$ for all $k\in\mathbb{N}$.  
\end{enumerate}
\end{assumpt}

Assumption \ref{Assumption_app1}(1) and \ref{Assumption_app2}(1) employ different sampling techniques, leading to different forms of evaluation noise.  In subsequent analysis, we will demonstrate that they conform to \eqref{eq:reshuf} and \eqref{eq:chi}, respectively.

Assumption \ref{Assumption_app1}(2) and \ref{Assumption_app2}(2) both assume that $\{\eta_{k}\}$ is composed of a scaling parameter $c$ and a fixed diminishing sequence of step-sizes (at most in the rate of $o(1/\log(k))$), hence enabling great flexibility in choosing the step-sizes $\{\eta_k\}$ in practical scenarios.

In the following theorem, we establish the global convergence of Algorithm \ref{alg:dsgd} by showing that Algorithm \ref{alg:dsgd} fits into our proposed framework with $\Phi = \D_f$, $\psi = f$ and $\ca{A} = \{{\bm x} \in \Rn: {\bm 0} \in \D_f({\bm x})\}$. 
\begin{theo}
    \label{thm:dsgd}
    Suppose Assumption \ref{Assumption_obj} holds. Let $\{{\bm X}_k\}$ be the sequence generated by Algorithm \ref{alg:dsgd}. 
    \begin{enumerate}[label={(\arabic*)}]
        \item When Assumption \ref{Assumption_app1} holds, there exists $\alpha_{c}>0$, for any $c\in (0, \alpha_{c})$, it holds that almost surely, any cluster point of $\{{\bm X}_{k}\}$ is a $\D_f$-critical point of \eqref{Prob_DOP} and $\{f({\bm x}_{i,k}): k \in \bb{N}\}$ converges for any $i \in [d]$.
        \item When Assumption \ref{Assumption_app2} holds, for any $\varepsilon>0$, there exists $\alpha_{c}>0$, for any $c\in (0, \alpha_{c})$, 
        \begin{equation*}
            \mathbb{P}\left(\lim_{k \to \infty}\mathrm{dist}({\bm X}_{k}, \{{\bm X}\in \mathbb{R}^{n\times d}: {\bm X}={\bm x}{\bm 1}^{\top}, {\bm 0} \in D_f({\bm x})\}) =0  \right)\geq 1-\varepsilon,
        \end{equation*}
and 
\begin{equation*}
\mathbb{P}(f({\bm x}_{i,k}) \textit { converges })\geq 1-\varepsilon, \textit{ }\forall i \in [d].
\end{equation*}
    \end{enumerate}
\end{theo}
\begin{proof}
As demonstrated in Line 7 of Algorithm \ref{alg:dsgd}, the sequence $\{{\bm X}_{k}\}$ follows the update scheme, 
\begin{equation}\label{eq:thm:dsgd_0}
    {\bm X}_{k+1} = {\bm X}_k {\bm W} - \eta_k {\bm D}_{k}, ~  {\bm D}_k =  [{\bm d}_{1,k}, \ldots, {\bm d}_{d,k}]. 
\end{equation}

Let the filtration $\ca{F}_k := \sigma(\{{\bm X}_l : l\leq k\})$. We set ${\bm Z}_k := {\bm X}_k$, ${\bm H}_k := \bb{E}[{\bm D}_k| \ca{F}_k]$ and $\Xi_{k+1} := {\bm D}_k - {\bm H}_k$. To demonstrate that Algorithm \ref{alg:dsgd} fits into our framework \eqref{Eq_Framework}, it suffices to check the update scheme \eqref{eq:thm:dsgd_0} is a special case of \eqref{Eq_Framework} satisfying Assumption \ref{Assumption_framework}, and verify the validity of  Assumption \ref{asp:reshuffling} or \ref{Assumption_stochastic} with different sampling techniques.  Assumption \ref{Assumption_framework}(1) follows immediately from Assumption \ref{Assumption_app1}(2) or \ref{Assumption_app2}(2); Assumption \ref{Assumption_framework}(2) follows quickly from local boundedness of the conservative field $D_{F_{i,l}}$. Nevertheless, how to verify \eqref{relation_H_PHI} varies across different cases. For case (1) in Theorem \ref{thm:dsgd}, we set  $\Phi_i := \D_{f_i}$. In this case, the nonnegative sequence $\{\delta_{k}\}$ in \eqref{relation_H_PHI} is not subject to any other restrictions, hence we may select a sufficiently large one to meet the practical requirement for ${\bm H}_{k}$ and $\Phi_i$. This choice makes Assumption \eqref{relation_H_PHI} hold trivially. For case (2), we need to carefully choose a diminishing $\{\delta_{k}\}$ so that it conforms to the conditions in Theorem \ref{theo:with}.  

Specifically, we note that $\mathbb{E}[{\bm d}_{i,k}|\mathcal{F}_{k}] \in \mathbb{E}_{\mathcal{B}_{i, k}}[\frac{1}{|\mathcal{B}_{i,k}|} \sum_{{\bm s}_l \in \mathcal{B}_{i,k}}  D_{F_{i,l}}({\bm x}_{i,k})]= D_{f_{i}} ({\bm x}_{i,k})$ in case (2), we can conclude that
\begin{equation*}
    \frac{1}{d}{\bm H}_k {\bm 1}_d \in \frac{1}{d}\sum_{i = 1}^d D_{f_i}({\bm x}_{i,k}) \subseteq \conv\left( \frac{1}{d}\sum_{i = 1}^d D_{f_i}({\bm x}_{i,k}) \right). 
\end{equation*}
Therefore, by choosing $\Phi_i := \D_{f_i}$ and $\delta_{k}:=0$, we have verified the validity of \eqref{relation_H_PHI}.

Furthermore, according to \cite[Section 6]{bolte2021conservative} and the path-differentiability of $f$, it follows that  $f$ is the Lyapunov function of the differential inclusion $\frac{\mathrm{d}{\bm x}}{\mathrm{d}t} \in -\D_f({\bm x})$, 
and admits the stable set $\{{\bm x} \in \Rn: {\bm 0} \in \D_f({\bm x})\}$. In addition, combining Assumption \ref{Assumption_obj}(1) and Proposition \ref{prop:finite} yields $\{f({\bm x}): {\bm 0} \in \D_f({\bm x})\}$ is finite in $\bb{R}$. As a result, we verify the validity of Assumption \ref{Assumption_framework} with $\Phi := \D_f$, $\psi := f$, and $\ca{A} := \{{\bm x} \in \Rn: {\bm 0} \in \D_f({\bm x})\}$. 

With the choice of $\mathcal{U}_{i,j}({\bm x}):=\frac{1}{|\mathcal{B}_{i,j}|} \sum_{{\bm s}_l\in \mathcal{B}_{i,j}} D_{F_{i,l}}({\bm x})$, where $\mathcal{B}_{i,j}, 0\leq j \leq N-1$ are subsets of equal size, and $\rho_k := 0$, Assumption \ref{asp:reshuffling} is satisfied for case (1); With the choice of $\xi_{i, k}:=\{\textit{Draw out } \mathcal{B}_{i,k} \textit{ from }$ $\textit{ agent } i\}$, $\delta_k :=0$ and $\chi_{i}({\bm x}, \xi_{i, k})\in \frac{1}{|\mathcal{B}_{i,k}|} \sum_{{\bm s}_l \in \mathcal{B}_{i,k}}  D_{F_{i,l}}({\bm x}) -D_{f_i}({\bm x})$, Assumption \ref{Assumption_stochastic} is satisfied for case (2). Finally, setting $\mathcal{X}_{0}$ as a compact set which contains ${\bm x}_0$, and leveraging Theorem \ref{theo:reshuff} and \ref{theo:with}, we completes our proof. 
\end{proof}

The following theorem demonstrates that Algorithm \ref{alg:dsgd} provably avoids spurious critical points. For clarity, we defer the full proof to Section \ref{appendix:1}.

\begin{theo}\label{thm:DSGD2}
Suppose Assumption \ref{Assumption_obj} holds. Let $\mathcal{X}_0$ be any compact set of $\mathbb{R}^n$, and sequence $\{{\bm X}_{k}\}$ be generated by Algorithm 1.  
    \begin{enumerate}[label={(\arabic*)}]
        \item When Assumption \ref{Assumption_app1} holds, there exists $\alpha_{c}>0$, a full-measure subset $\mathcal{S}$ of $(0, \alpha_c)$, and a full-measure subset $\mathcal{X}_1$ of $\mathcal{X}_0$, such that for any $c\in \mathcal{S}$ and ${\bm x}_0 \in \mathcal{X}_1$, it holds that almost surely, any cluster point of $\{{\bm X}_{k}\}$ is a Clarke-critical point of \eqref{Prob_DOP} and $\{f({\bm x}_{i,k}): k \in \bb{N}\}$ converges for any $i \in [d]$. 
        \item When Assumption \ref{Assumption_app2} holds, for any $\varepsilon>0$, there exists $\alpha_{c}>0$, a full-measure subset $\mathcal{S}$ of $(0, \alpha_c)$, and a full-measure subset $\mathcal{X}_1$ of $\mathcal{X}_0$, such that for any $c\in \mathcal{S}$ and ${\bm x}_0 \in \mathcal{X}_1$, 
        \begin{equation*}
            \mathbb{P}\left(\lim_{k \to \infty}\mathrm{dist}({\bm X}_{k}, \{{\bm X}\in \mathbb{R}^{n\times d}: {\bm X}={\bm x}{\bm 1}^{\top}, {\bm 0} \in \partial f({\bm x})\}) =0  \right)\geq 1-\varepsilon,
        \end{equation*}
and 
\begin{equation*}
\mathbb{P}(f({\bm x}_{i,k}) \textit { converges })\geq 1-\varepsilon, \textit{ }\forall i \in [d].
\end{equation*}
    \end{enumerate}

\end{theo}

\subsection{Decentralized generalized SGD with momentum }\label{sec:dgsgdm}

Momentum is a technique to accelerate the subgradient descent method in the relevant direction and dampen oscillations. In the single-agent setting, \cite{xiao2023convergence} investigates the global convergence properties of generalized SGD with momentum, where an auxiliary function $\phi: \Rn \to \bb{R}$ is introduced to normalize the update directions of the primal variables. 

In this subsection, motivated by the generalized SGD with momentum analyzed in \cite{xiao2023convergence}, we introduce the decentralized generalized stochastic subgradient descent method with momentum (DGSGD-M) and present the details in Algorithm \ref{alg:DSGD-M}. Different from DSGD in Algorithm \ref{alg:dsgd}, Algorithm \ref{alg:DSGD-M} introduces the auxiliary variables $\{{\bm y}_{i,k}: k\in \bb{N}\}$ to track the first-order moment of the stochastic subgradients of $f_i$. Therefore, Algorithm \ref{alg:DSGD-M} requires two rounds of communication to perform local consensus to $\{{\bm x}_{1,k}, \ldots, {\bm x}_{d,k}\}$ and $\{{\bm y}_{1,k}, \ldots, {\bm y}_{d,k}\}$, respectively. Moreover, similar to the generalized SGD analyzed in \cite{xiao2023convergence}, Algorithm \ref{alg:DSGD-M} introduces an auxiliary function $\phi$ to normalize the update directions of its primal variables $\{{\bm x}_{1,k}, \ldots, {\bm x}_{d,k}\}$.

\begin{algorithm}
\caption{DGSGD-M for solving \eqref{Prob_DOP}.}
\label{alg:DSGD-M}
\begin{algorithmic}[1] % 数字1确保行号是数字
\Require Initial point ${\bm x}_{0}\in \mathbb{R}^n$, initial descent direction ${\bm y}_{0}\in \mathbb{R}^n$, momentum parameter $\tau > 0$, an auxiliary function $\phi:\Rn \to \bb{R}$,  and a mixing matrix ${\bm W}$. 
\For{all $i\in [d]$ in parallel}
\State Set $k \gets 0$. Initialize ${\bm x}_{i,k}={\bm x}_0, {\bm y}_{i,k}={\bm y}_0$; 
\While{not terminated}
    \State Choose the step-size $\eta_{k}$;
    \State Communicate and update the local variable 
    $${\bm x}_{i, k+1} \in \sum_{j\in \mathcal{N}_{i}}{\bm W}(i,j) {\bm x}_{j, k} - \eta_{k} \partial \phi({\bm y}_{i, k}) ;$$
    \State Randomly select a mini-batch $\mathcal{B}_{i,k+1} \subseteq \mathcal{S}_{i}$;
    \State Compute ${\bm d}_{i,k+1} \in \frac{1}{|\mathcal{B}_{i,k+1}|} \sum_{{\bm s}_l \in \mathcal{B}_{i,k+1}}  D_{F_{i,l}}({\bm x}_{i,k+1})$;
    \State Communicate and update local descent direction 
    $${\bm y}_{i, k+1} = (1-\tau \eta_{k}) \sum_{j\in \mathcal{N}_{i}} {\bm W}(i,j) {\bm y}_{j,k}     + \tau \eta_{k} {\bm d}_{i,k+1};$$
    \State Set $k \gets k + 1$;
\EndWhile
\EndFor
\State \Return ${\bm X}_k:= [{\bm x}_{1,k}, \ldots, {\bm x}_{d,k}]$;
\end{algorithmic}
\end{algorithm}

\begin{rmk}
    The auxiliary function $\phi$ in Algorithm \ref{alg:DSGD-M} determines how the update directions of the variables $\{{\bm x}_k\}$ are regularized. Hence, Algorithm \ref{alg:DSGD-M} yields different variants of DSGD-type methods with different choices of the auxiliary function $\phi$. For example, when we choose $\phi({\bm y}) = \frac{1}{2}\norm{{\bm y}}^2$, Algorithm \ref{alg:DSGD-M} corresponds to DSGD with momentum; when we select $\phi({\bm y}) = \norm{{\bm y}}_1$, Algorithm \ref{alg:DSGD-M} employs the sign map to regularize the update directions of the variables $\{{\bm x}_k\}$, which can be viewed as a decentralized variant of SignSGD analyzed in \cite{bernstein2018signsgd}, and denoted as DSignSGD.  
    Interested readers could refer to \cite{xiao2023convergence} for more details on the choices of the auxiliary function $\phi$. 
    
    { Furthermore, recognizing signSGD as a limiting variant of Adam \cite{chen2023symbolic}, we extend our analysis in Appendix \ref{appendix:adam} to demonstrate how Adam-type methods can be adapted to our framework, hence establishing convergence guarantees for these methods under the nonsmooth nonconvex settings.}
\end{rmk}

Let ${\bm Y}_{k}:= [{\bm y}_{1,k}, \ldots, {\bm y}_{d,k}]$ at $k$-th iteration in Algorithm \ref{alg:DSGD-M}. The convergence properties of this algorithm are established in Theorem \ref{thm:DSGD-M}, within which we concern about two different noise settings and step-sizes in Assumption \ref{Assumption_app1} and \ref{Assumption_app2}.

\begin{theo}\label{thm:DSGD-M}
Suppose Assumption \ref{Assumption_obj} holds. $\{{\bm X}_k\}$ is generated by Algorithm \ref{alg:DSGD-M}, and auxiliary function $\phi$ is convex and admits a unique minimizer at ${\bm 0}$. 
    \begin{enumerate}[label={(\arabic*)}]
        \item When Assumption \ref{Assumption_app1} holds, there exists $\alpha_{c}>0$, for any $c\in (0, \alpha_{c})$, it holds that almost surely, any cluster point of $\{{\bm X}_{k}\}$ is a $\D_f$-critical point of \eqref{Prob_DOP} and $\{f({\bm x}_{i,k}): k \in \bb{N}\}$ converges for any $i \in [d]$.
        \item When Assumption \ref{Assumption_app2} holds, for any $\varepsilon>0$, there exists $\alpha_{c}>0$, for any $c\in (0, \alpha_{c})$, 
        \begin{equation*}
            \mathbb{P}\left(\lim_{k \to \infty}\mathrm{dist}({\bm X}_{k}, \{{\bm X}\in \mathbb{R}^{n\times d}: {\bm X}={\bm x}{\bm 1}^{\top}, {\bm 0} \in D_f({\bm x})\}) =0  \right)\geq 1-\varepsilon
        \end{equation*}
and 
\begin{equation*}
\mathbb{P}(f({\bm x}_{i,k}) \textit { converges })\geq 1-\varepsilon, \textit{ }\forall i \in [d].
\end{equation*}
    \end{enumerate}
\end{theo}
\begin{proof}
    From Line 5 and Line 8 of Algorithm \ref{alg:DSGD-M}, we can conclude that the sequences $\{{\bm X}_k\}$ and $\{{\bm Y}_k\}$ follow the update scheme,
    \begin{equation}
        \label{Eq_Thm_DSGD-M}
        \begin{aligned}
            {\bm X}_{k+1} \in{}&  {\bm X}_k {\bm W} - \eta_k {\bm T}_k, ~ {\bm T}_k \in [\partial \phi({\bm y}_{1,k}), \ldots, \partial \phi({\bm y}_{d,k})], \\ 
            {\bm D}_{k+1} ={}& [{\bm d}_{1, k+1}, \ldots, {\bm d}_{d, k+1}],\\
            {\bm Y}_{k+1} ={}& (1-\tau\eta_{k}) {\bm Y}_{k}{\bm W}+ \tau \eta_{k} {\bm D}_{k+1}.
        \end{aligned}
    \end{equation}
    Let the filtration $\ca{F}_k := \sigma(\{{\bm X}_{l}, ~ {\bm Y}_{l} : l\leq k\})$. Then by setting 
\begin{equation*}
{\bm Z}_k := \left[\begin{smallmatrix}
        {\bm X}_{k}\\
        {\bm Y}_{k}
    \end{smallmatrix} \right], {\bm H}_k := \bb{E}\left[\left[\begin{smallmatrix}
        {\bm T}_k\\
        \tau {\bm Y}_{k}{\bm W}- \tau {\bm D}_{k+1}
    \end{smallmatrix} \right] |\ca{F}_k\right], \text{ and } \Xi_{k+1} := \left[\begin{smallmatrix}
        {\bm T}_k\\
        \tau {\bm Y}_{k}{\bm W}- \tau {\bm D}_{k+1}
    \end{smallmatrix} \right] - {\bm H}_k,
\end{equation*}
we can conclude that the update scheme \eqref{Eq_Thm_DSGD-M} can be reshaped as ${\bm Z}_{k+1} = {\bm Z}_k {\bm W} - \eta_k ({\bm H}_k + \Xi_{k+1})$, which aligns with the update format of \eqref{Eq_Framework}.

    Now we check the validity of \eqref{relation_H_PHI} and Assumption \ref{Assumption_framework}, and then show  Assumption \ref{asp:reshuffling} and \ref{Assumption_stochastic} are satisfied for case (1) and (2), respectively. From the local boundedness of $\D_{f_i}$, we can conclude that $\{{\bm D}_k\}$ is uniformly bounded almost surely whenever $\{{\bm X}_{k}\}$ is bounded. Together with the update scheme of ${\bm Y}_k$, we can conclude that $\norm{{\bm Y}_{k+1}} \leq (1-\tau\eta_k)\norm{{\bm Y}_k} + \tau \eta_k\norm{{\bm D}_{k+1}}$, hence there exists a constant $M_Y> 0$ such that $\sup_{k\geq 0} \norm{{\bm Y}_k} \leq \max\{\norm{{\bm Y}_0}, \sup_{k\geq 0}\norm{{\bm D}_k}\} < M_Y$, if $\{{\bm X}_{k}\}$ is bounded. Then combining with local boundedness of $\partial \phi$, it can be referred that $\{{\bm T}_{k}\}$ is uniformly bounded almost surely whenever $\{{\bm X}_{k}\}$ is bounded. This  
    illustrates that Assumption \ref{Assumption_framework}(2) holds. 
 
    Moreover, we set $\Phi_i({\bm x}, {\bm y}) := 
        \left[\begin{smallmatrix}
        \partial \phi({\bm y})\\
        \tau {\bm y} - \tau \D_{f_i}({\bm x})
    \end{smallmatrix} \right]$, then  \eqref{relation_H_PHI} holds automatically for case (1) with sufficiently large $\delta_{k}$. For case (2), we define $\delta_k$ as an exponential moving average of $\{\eta_{j}\|{\bm T}_{j}\|\}_{j\leq k}$, expressed as 
\begin{equation*}
\delta_k := \eta_{k}\|{\bm T}_{k}\| +2\sum_{j=0}^{k-1} \lambda_2^{k-1-j}\eta_{j} \|{\bm T}_{j}\|.
\end{equation*}
From this, we can deduce that $\frac{1}{d}{\bm H}_k {\bm 1}_d \in \frac{1}{d} \sum_{i = 1}^d \Phi^{\delta_k}_i({\bm x}_{i,k}, {\bm y}_{i,k})$ for case (2). 
    
Furthermore,  from the definition of $\Phi_i$ and \eqref{relation_H_PHI}, we attain that $\Phi({\bm x},{\bm y}) = \left[\begin{smallmatrix}
        \partial \phi({\bm y})\\
        \tau {\bm y} - \tau \D_{f}({\bm x})
    \end{smallmatrix} \right]$. 
Then together with \cite[Proposition 4.5]{xiao2023convergence} and the path-differentiability of $f$, we can conclude that $\psi({\bm x}, {\bm y}) = f({\bm x}) + \frac{1}{\tau}\phi({\bm y})$ is the Lyapunov function of the differential inclusion $\left( \frac{\mathrm{d}{\bm x}}{\mathrm{d}t}, \frac{\mathrm{d}{\bm y}}{\mathrm{d}t}  \right) \in -\Phi({\bm x},{\bm y})$, and admits $\{({\bm x}, {\bm y}) \in \Rn \times \Rn: {\bm 0}\in \D_f({\bm x}), {\bm y} = {\bm 0}\}$ as its stable set. Leveraging the definability of $f$ and $\D_f$, it follows that the set   $\{\psi({\bm x}, {\bm y}) \in \Rn \times \Rn: {\bm 0}\in \D_f({\bm x}), {\bm y} = {\bm 0}\}$ is finite in $\bb{R}$. Therefore, Assumption \ref{Assumption_framework} holds for Algorithm \ref{alg:DSGD-M}. 
    
We set $\mathcal{U}_{i,j}({\bm x}, {\bm y}):= \left[\begin{smallmatrix} 
\partial \phi({\bm y}) \\
\tau {\bm y}- \tau \frac{1}{|\mathcal{B}_{i,j}|} \sum_{{\bm s}_l\in \mathcal{B}_{i,j}}D_{F_{i,l}}({\bm x})
\end{smallmatrix} \right]$, where $\mathcal{B}_{i,j}$ for $0\leq j \leq N-1$ are subsets of equal size, and set $\rho_k := \eta_{k}\|{\bm T}_{k}\| +2\sum_{j=0}^{k-1} \lambda_2^{k-1-j}\eta_{j} \|{\bm T}_{j}\|$. Then \eqref{Eq_Thm_DSGD-M} is a special form of \eqref{Eq_Framework} with \eqref{eq:reshuf}. Moreover, Assumption \ref{Assumption_app1} implies that Assumption \ref{asp:reshuffling} is satisfied for case (1). With the choice of $\xi_{i, k}:=\{\textit{Draw out } \mathcal{B}_{i,k} \textit{ from agent } i\}$, and 
\begin{equation*}
\chi_{i}({\bm x}, {\bm y}, \xi_{i, k})\in \left[\begin{smallmatrix}
        0\\
        \frac{- \tau}{|\mathcal{B}_{i,k}|} \sum_{{\bm s}_l \in \mathcal{B}_{i,k}}  D_{F_{i,l}}({\bm x}, {\bm y}) -D_{f_i}({\bm x}, {\bm y})
    \end{smallmatrix} \right],
\end{equation*}
\eqref{Eq_Thm_DSGD-M} is a particular form of \eqref{Eq_Framework} with \eqref{eq:chi}, and Assumption \ref{Assumption_stochastic} is satisfied for case (2). Plugging Lemma \ref{lem:2} and Theorem \ref{theo:withreplace_thm1}, we can prove that $\delta_{k} \to 0$, as $k \to \infty$ with arbitrary high probability \footnote{Notice that Theorem \ref{theo:withreplace_thm1} does not require $\delta_k$ to diminish.}. Employing Theorem \ref{theo:reshuff} and \ref{theo:with}, we can derive the results in Theorem \ref{thm:DSGD-M}. 
    
\end{proof}

Theorem \ref{thm:DSGDM2} illustrates how Algorithm \ref{alg:DSGD-M} avoids  spurious critical points, whose proof is presented in Appendix \ref{appendix:1} for simplicity. { We also provide a numerical example in Appendix \ref{appendix:2} to demonstrate that DSGD-m can avoid spurious critical points.}

\begin{theo}\label{thm:DSGDM2}
Suppose Assumption \ref{Assumption_obj} holds and ${\bm W}$ is a non-singular mixing matrix. $\mathcal{X}_0$ is a compact set of $\mathbb{R}^n$,  sequence $\{{\bm X}_{k}\}$ is generated by Algorithm \ref{alg:DSGD-M}, and auxiliary function $\phi$ is convex which admits a unique minimizer at ${\bm 0}.$  
    \begin{enumerate}[label={(\arabic*)}]
        \item When Assumption \ref{Assumption_app1} holds, there exists $\alpha_{c}>0$, a full-measure subset $\mathcal{S}$ of $(0, \alpha_c)$, and a full-measure subset $\mathcal{K}$ of $\mathcal{X}_0\times \mathbb{R}^{n}$, such that for any $c\in \mathcal{S}$ and $({\bm x}_0, {\bm y}_0) \in \mathcal{K}$, it holds that almost surely, any cluster point of $\{{\bm X}_{k}\}$ is a Clarke-critical point of \eqref{Prob_DOP} and $\{f({\bm x}_{i,k}): k \in \bb{N}\}$ converges for any $i \in [d]$. 
        \item When Assumption \ref{Assumption_app2} holds, for any $\varepsilon>0$, there exists $\alpha_{c}>0$, a full-measure subset $\mathcal{S}$ of $(0, \alpha_c)$, and a full-measure subset $\mathcal{K}$ of $\mathcal{X}_0\times \mathbb{R}^{n}$, such that for any $c\in \mathcal{S}$ and $({\bm x}_0, {\bm y}_0) \in \mathcal{K}$, 
        \begin{equation*}
            \mathbb{P}\left(\lim_{k \to \infty}\mathrm{dist}({\bm X}_{k}, \{{\bm X}\in \mathbb{R}^{n\times d}: {\bm X}={\bm x}{\bm 1}^{\top}, {\bm 0} \in \partial f({\bm x})\}) =0  \right)\geq 1-\varepsilon
        \end{equation*}
and 
\begin{equation*}
\mathbb{P}(f({\bm x}_{i,k}) \textit { converges })\geq 1-\varepsilon, \textit{ }\forall i \in [d].
\end{equation*}
    \end{enumerate}

\end{theo}

\subsection{Decentralized SGD with gradient-tracking technique}\label{sec:dsgt}

The decentralized SGD method with gradient-tracking technique (DSGD-T) is another widely employed stochastic 
(sub)gradient-based method in decentralized learning. Algorithm \ref{alg:dsgt} exhibits the detailed implementations of DSGD-T in minimizing nonsmooth path-differentiable functions.  Different from the DSGD in Algorithm \ref{alg:dsgd}, DSGD-T introduces a sequence of auxiliary variables to track the average of local subgradients over all the agents. To the best of our knowledge, all the existing works focus on analyzing the convergence of DSGD-T with Clarke regular objective functions, and the convergence of DSGD-T for nonsmooth functions without Clarke regularity remains unexplored. 

In this subsection, we prove that DSGD-T can be enclosed by our proposed framework \eqref{Eq_Framework} with \eqref{eq:reshuf} or \eqref{eq:chi}. Therefore, based on our theoretical results exhibited in Theorem \ref{theo:reshuff} and \ref{theo:with}, we can establish the global convergence of DSGD-T in the minimization of nonsmooth path-differentiable functions.

\begin{algorithm}
\caption{DSGD-T for solving \eqref{Prob_DOP}}
\label{alg:dsgt}
\begin{algorithmic}[1] % 数字1确保行号是数字
\Require Initial point ${\bm x}_{0}\in \mathbb{R}^n$, initial descent direction ${\bm V}_{0}=[{\bm v}_{1,0}, \ldots, {\bm v}_{d,0}]$ satisfying ${\bm V}_{0}{\bm 1}_{d} = \sum_{i=1}^{d} {\bm d}_{i, 0} $, where ${\bm d}_{i,0} \in \frac{1}{|\mathcal{B}_{i,0}|} \sum_{{\bm s}_l \in \mathcal{B}_{i,0}}  D_{F_{i,l}}({\bm x}_{i,0}),$ $\mathcal{B}_{i,0}$ is a mini-batch of $\mathcal{S}_{i}$, and a mixing matrix ${\bm W}$. 
\For{all $i\in [d]$ in parallel}
\State Set $k \gets 0$. Initialize ${\bm x}_{i,k}={\bm x}_0$; 
\While{not terminated}
    \State Choose the step-size $\eta_k$;
    \State Communicate and update the local variable 
    $${\bm x}_{i,k+1} = \sum_{j\in \mathcal{N}_{i}}{\bm W}(i,j) {\bm x}_{j,k} - \eta_{k} {\bm v}_{i,k}; $$
    \State Randomly select a mini-batch $\mathcal{B}_{i, k+1} \subseteq \mathcal{S}_{i}$;
    \State Compute ${\bm d}_{i,k+1} \in \frac{1}{|\mathcal{B}_{i,k+1}|} \sum_{{\bm s}_l \in \mathcal{B}_{i,k+1}}  D_{F_{i,l}}({\bm x}_{i, k+1})$;
    \State Communicate and update local descent direction 
    $${\bm v}_{i,k+1} =  \sum_{j\in \mathcal{N}_{i}}{\bm W}(i,j) ({\bm v}_{j,k} + {\bm d}_{j,k+1}  - {\bm d}_{j,k});$$
    \State Set $k \gets k + 1$;
\EndWhile
\EndFor
\State \Return ${\bm X}_k:= [{\bm x}_{1,k}, \ldots, {\bm x}_{d,k}]$;
\end{algorithmic}
\end{algorithm}

In the rest of this subsection, we denote  ${\bm V}_{k} = [{\bm v}_{1,k},$ $ \ldots, {\bm v}_{d,k}]$ at $k$-th iteration in Algorithm \ref{alg:dsgt}. We condense the assertion that DSGD-T can fit in our framework \eqref{Eq_Framework} with different noise settings, along with convergence analysis, into a single theorem.

\begin{theo}\label{thm:dsgt}
Suppose Assumption \ref{Assumption_obj} holds. Let $\{{\bm X}_k\}$ be the sequence generated by Algorithm \ref{alg:dsgt}. 
    \begin{enumerate}[label={(\arabic*)}]
        \item When Assumption \ref{Assumption_app1} holds, there exists $\alpha_{c}>0$, for any $c\in (0, \alpha_{c})$, it holds that almost surely, any cluster point of $\{{\bm X}_{k}\}$ is a $\D_f$-critical point of \eqref{Prob_DOP} and $\{f({\bm x}_{i,k}): k \in \bb{N}\}$ converges for any $i \in [d]$.
        \item When Assumption \ref{Assumption_app2} holds, for any $\varepsilon>0$, there exists $\alpha_{c}>0$, for any $c\in (0, \alpha_{c})$, 
        \begin{equation*}
            \mathbb{P}\left(\lim_{k \to \infty}\mathrm{dist}({\bm X}_{k}, \{{\bm X}\in \mathbb{R}^{n\times d}: {\bm X}={\bm x}{\bm 1}^{\top}, {\bm 0} \in D_f({\bm x})\}) =0  \right)\geq 1-\varepsilon
        \end{equation*}
and 
\begin{equation*}
\mathbb{P}(f({\bm x}_{i,k}) \textit { converges })\geq 1-\varepsilon, \textit{ }\forall i \in [d].
\end{equation*}
    \end{enumerate}
\end{theo}

\begin{proof}
From Line 5 - Line 8 of Algorithm \ref{alg:dsgt}, the sequences $\{{\bm X}_k\}$ and $\{{\bm V}_k\}$ are updated by the following scheme,  
\begin{equation}\label{eq:dsgt}
\begin{aligned}  
    {\bm X}_{k+1} ={}&  {\bm X}_k {\bm W} - \eta_k {\bm V}_k,  ~ {\bm D}_{k+1} ={}[{\bm d}_{1,k+1}, \ldots, {\bm d}_{d,k+1}],\\
    {\bm V}_{k+1} = {}&  ({\bm V}_{k}  + {\bm D}_{k+1}  - {\bm D}_{k}) {\bm W}. \\
\end{aligned}
\end{equation}

Let the filtration $\ca{F}_k := \sigma(\{{\bm X}_l : l\leq k\})$. By setting ${\bm Z}_k := {\bm X}_k$, ${\bm H}_k := \bb{E}[{\bm V}_k|\ca{F}_k]$ and $\Xi_{k+1} := {\bm V}_k - {\bm H}_k$, we can conclude that \eqref{eq:dsgt} coincides with the framework \eqref{Eq_Framework}. 

Then we aim to check the validity of Assumption \ref{Assumption_framework}. Due to the local boundedness of any conservative field, there exists $M_D>0$ such that $\sup_{k\geq 0} \norm{{\bm D}_k} \leq M_D$ almost surely when $\{{\bm X}_k\}$ is restricted in some bounded set. Then by the update scheme of $\{{\bm V}_k\}$ in \eqref{eq:dsgt}, it holds that  
\begin{equation*}
    {\bm V}_{k+1}  =  {\bm D}_{k+1}{\bm W} +  \sum_{i=1}^{k} {\bm D}_{i}({\bm W} - {\bm I}_d){\bm W}^{k+1-i}   - {\bm D}_{0}{\bm W}^{k+1} + {\bm V}_{0}{\bm W}^{k+1}.
\end{equation*}
Then from Corollary \ref{Prop_preliminary_mixing_matrix}, there exists a constant $\alpha \in (0,1)$, such that $\norm{({\bm W} - {\bm I}_d){\bm W}^{k}} \leq 2\alpha^{k}$ holds for any $k\geq 0$. Therefore, almost surely, it holds that 
\begin{equation*}
    \begin{aligned}
        \sup_{k\geq 0} \norm{{\bm V}_{k}}
       \leq{}& \sup_{k\geq 0} \| {\bm D}_{k+1}{\bm W} +  \sum_{i=1}^{k} {\bm D}_{i}({\bm W} - {\bm I}_d){\bm W}^{k+1-i}   - {\bm D}_{0}{\bm W}^{k+1} + {\bm V}_{0}{\bm W}^{k+1} \|\\
        \leq{}& 3\sup_{k\geq 0} \norm{{\bm D}_k} + \left(\sup_{k\geq 0}  \norm{{\bm D}_{k}}\right) \cdot \sum_{i=1}^{k}\norm{({\bm W} - {\bm I}_d){\bm W}^{k+1-i}}\\
        \leq{}& 3M_D  + \sum_{i = 1}^{\infty}  2M_D\alpha^{i} \leq \left(3+ \frac{2}{1-\alpha}\right) M_D. 
    \end{aligned}
\end{equation*}
Therefore, we can conclude that $\{{\bm V}_k\}$ is uniformly bounded almost surely when $\{{\bm X}_k\}$ is bounded, hence verifying the validity of Assumption \ref{Assumption_framework}(2). 

Assumption \ref{Assumption_framework}(1) follows quickly from Assumption \ref{Assumption_app1} or \ref{Assumption_app2}. \eqref{relation_H_PHI} holds vacuously with sufficiently large $\delta_{k}$ and $\Phi_i := \D_{f_i}$ for case (1). Moreover for case (2), followed by $\mathbb{E}[{\bm d}_{i,k}|\mathcal{F}_{k}] \in \mathbb{E}_{\mathcal{B}_{i, k}}[\frac{1}{|\mathcal{B}_{i,k}|} \sum_{{\bm s}_l \in \mathcal{B}_{i,k}}  D_{F_{i,l}}({\bm x}_{i,k})]= D_{f_{i}} ({\bm x}_{i,k})$, it is easy to deduce that 
\begin{equation*}
    \frac{1}{d}{\bm H}_k {\bm 1}_d  = \mathbb{E}[\frac{1}{d}{\bm V}_k {\bm 1}_d| \mathcal{F}_{k}] = \mathbb{E}[\frac{1}{d}{\bm D}_k {\bm 1}_d|\mathcal{F}_{k}]\in \frac{1}{d} \sum_{i = 1}^d \D_{f_i}({\bm x}_{i,k}), \forall k\in \bb{N}.
\end{equation*}
Therefore, by choosing $\delta_{k}:=0$ and $\Phi_i := \D_{f_i}$ for all $i \in [d]$, we verify the validity of \eqref{relation_H_PHI}.

Furthermore, according to \cite[Section 6]{bolte2021conservative}, we can assert that  $f$ is the Lyapunov function of the differential inclusion $\frac{\mathrm{d}{\bm x}}{\mathrm{d}t} \in -\D_f({\bm x})$, 
and admits the stable set $\{{\bm x} \in \Rn: {\bm 0} \in \D_f({\bm x})\}$. In addition, the definability of $f$ and $\D_f$ illustrates that $\{f({\bm x}): {\bm 0} \in \D_f({\bm x})\}$ is finite in $\bb{R}$. As a result, we have verified the validity of Assumption \ref{Assumption_framework} with $\Phi := \D_f$, $\psi := f$, and $\ca{A} := \{{\bm x} \in \Rn: {\bm 0} \in \D_f({\bm x})\}$. 

Following the same construction of $\mathcal{U}_{i,k}({\bm x})$, $\rho_k$, $\xi_{i, k}$ and $\chi_{i}$ as in the proof of Theorem \ref{thm:dsgd}, we could verify Assumption \ref{asp:reshuffling} and \ref{Assumption_stochastic} hold for case (1) and case (2), respectively. Henceforth, the update scheme \eqref{eq:dsgt} is a variant  of \eqref{Eq_Framework} with either \eqref{eq:reshuf} or \eqref{eq:chi}. Utilizing our theoretical results in Theorem \ref{theo:reshuff} and \ref{theo:with}, we completes the proof. 

\end{proof}

\section{Numerical Experiments}\label{sec:6}

\begin{figure*}[!t]
\centering
\subfloat[Test Accuracy]{\includegraphics[width=1.8in]{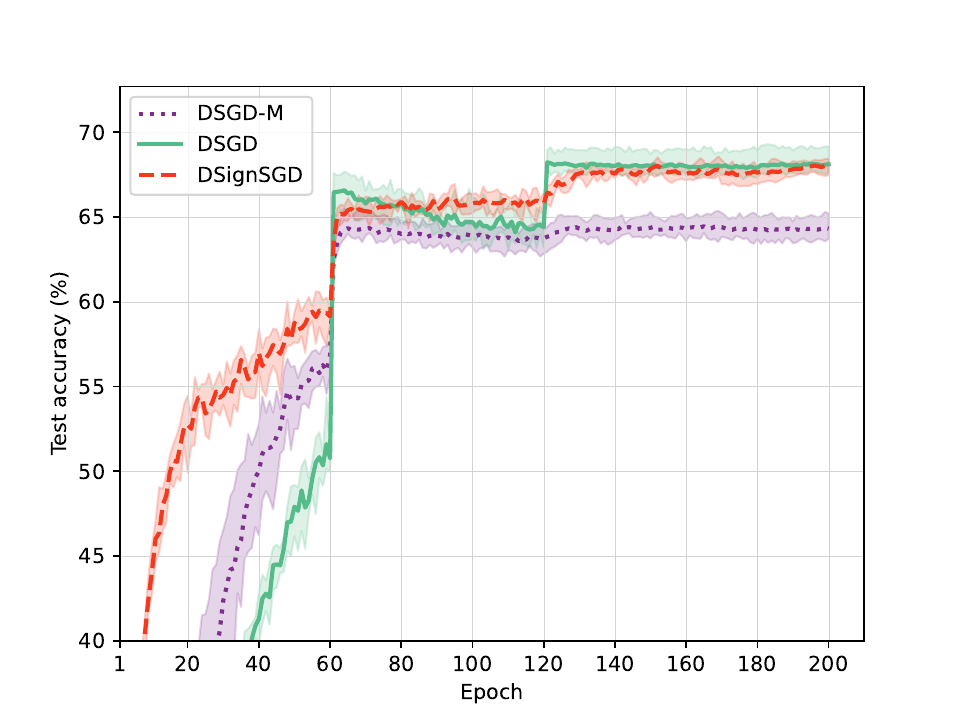}%
\label{fig_1_1}}
\hfil
\subfloat[Train Loss]{\includegraphics[width=1.8in]{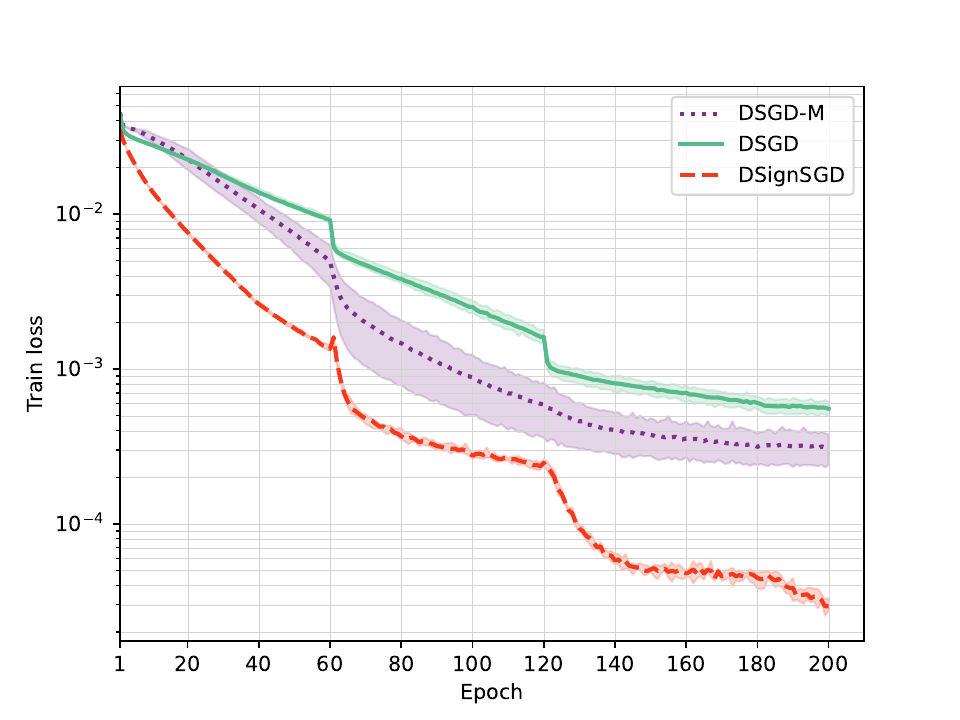}%
\label{fig_1_2}}
\hfil
\subfloat[Consensus Error]{\includegraphics[width=1.8in]{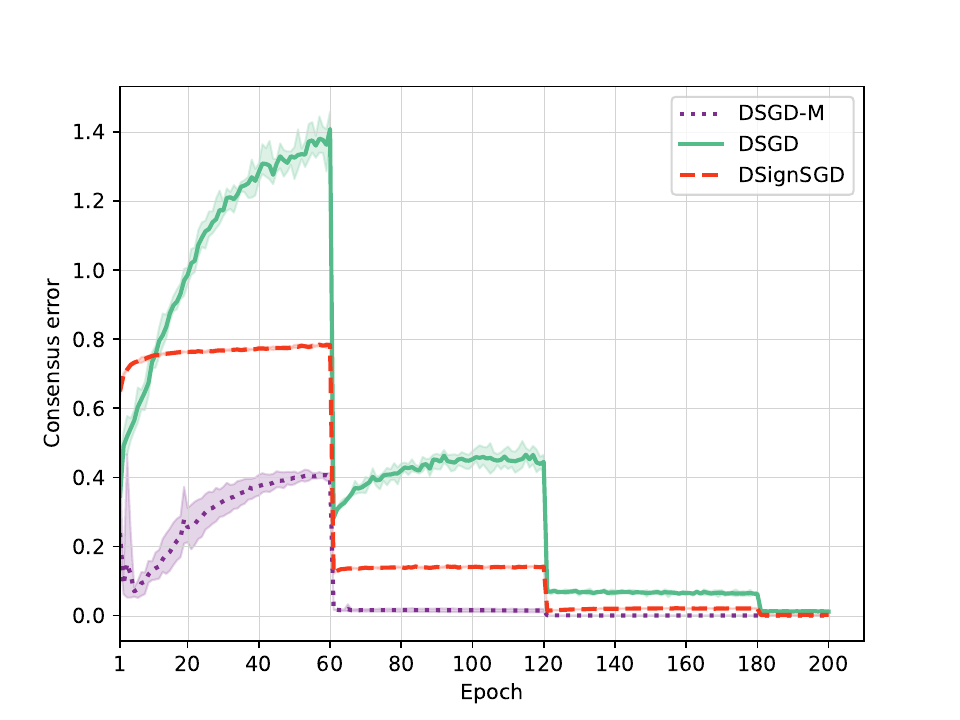}%
\label{fig_1_3}}

\caption{Numerical performance comparison of DSGD, DSGD-M, and DSignSGD in training ResNet50 on CIFAR-100 dataset using random reshuffling strategy.}
\label{fig:dsgd}
\end{figure*}

\begin{figure*}[!t]
\centering
\subfloat[Test Accuracy]{\includegraphics[width=1.8in]{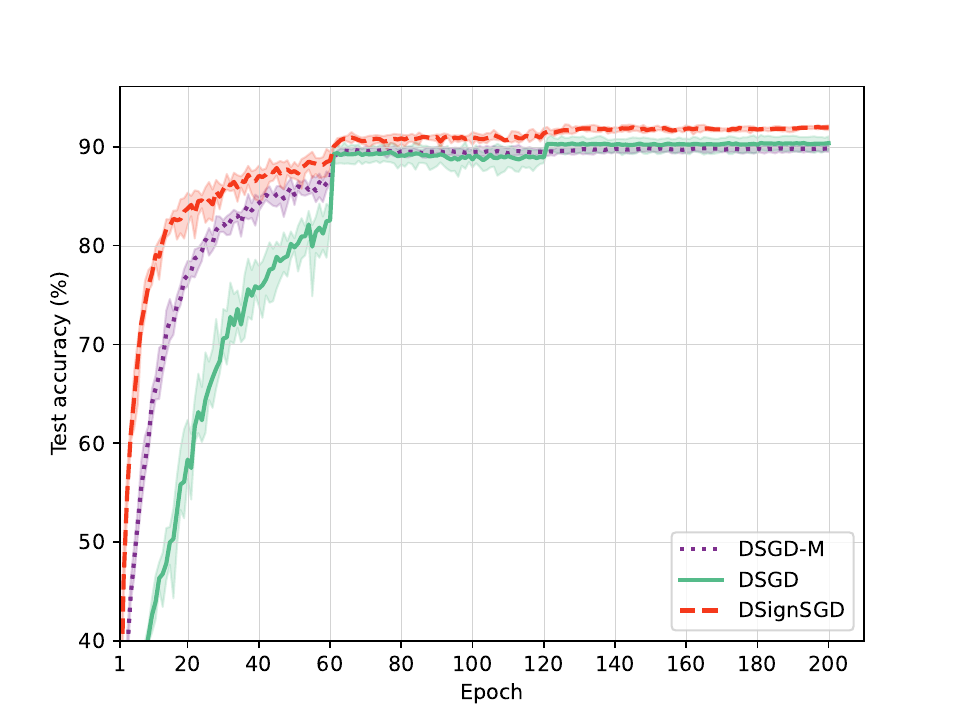}%
\label{fig_2_1}}
\hfil
\subfloat[Train Loss]{\includegraphics[width=1.8in]{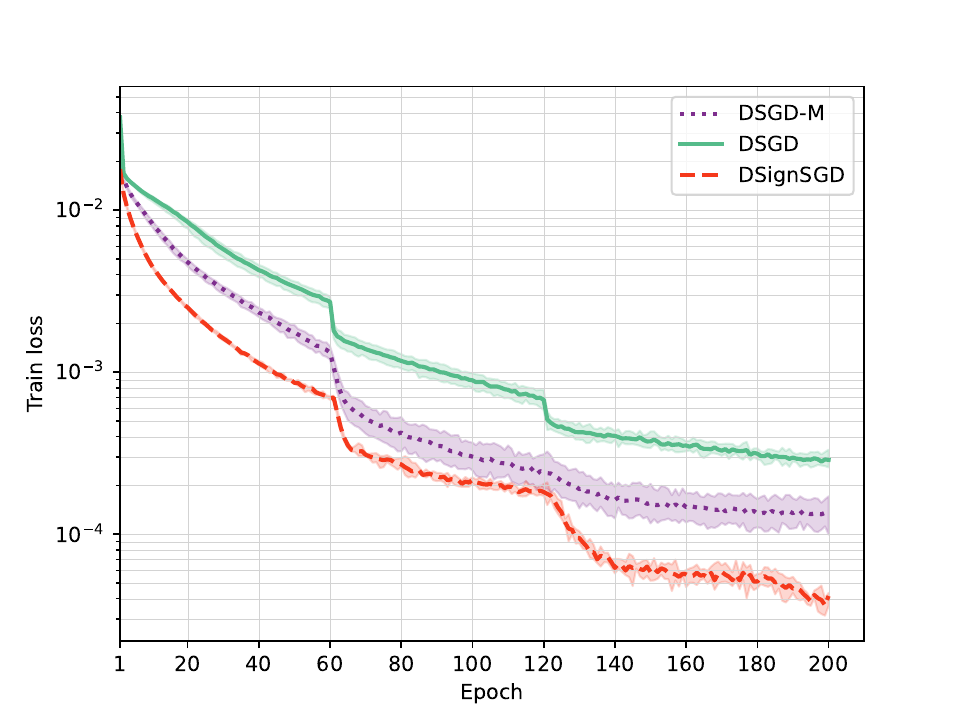}%
\label{fig_2_2}}
\hfil
\subfloat[Consensus Error]{\includegraphics[width=1.8in]{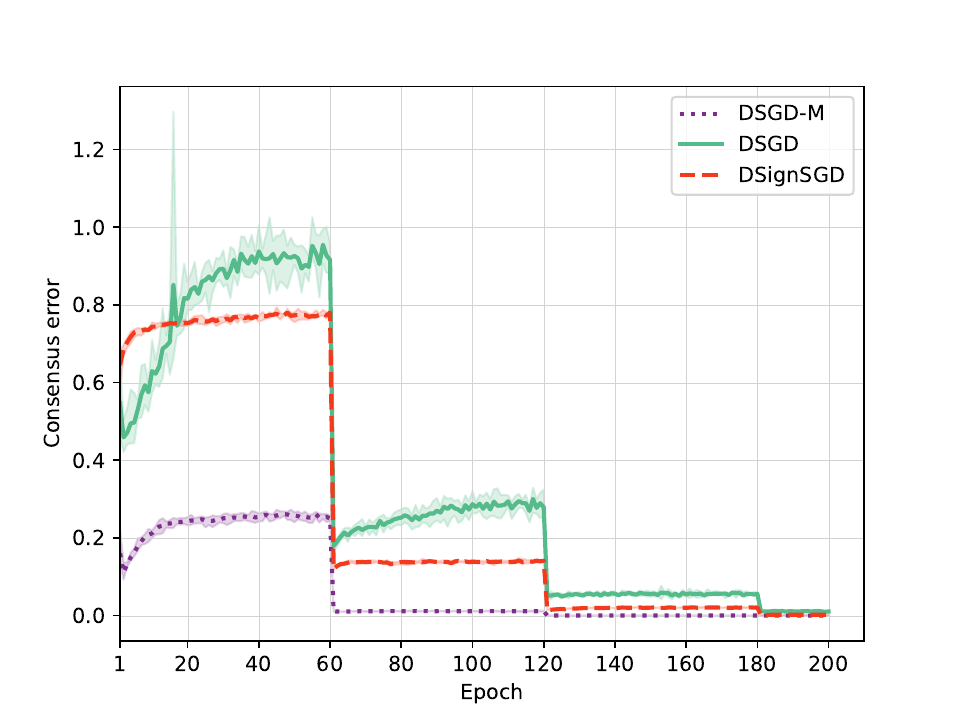}%
\label{fig_2_3}}

\caption{Numerical performance comparison of DSGD, DSGD-M and DSignSGD in training ResNet50 on CIFAR-10 dataset using random reshuffling strategy.}
\label{fig:dsgd2}
\end{figure*}

{

\begin{figure*}[!t]
\centering
\subfloat[Test Accuracy]{\includegraphics[width=1.8in]{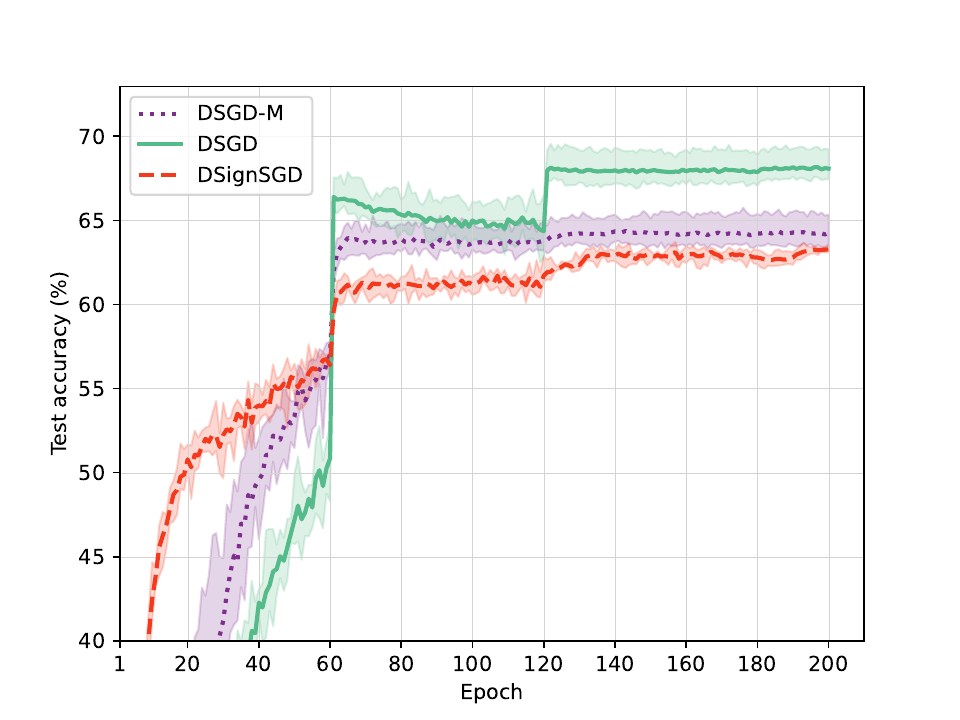}%
\label{fig_7_1}}
\hfil
\subfloat[Train Loss]{\includegraphics[width=1.8in]{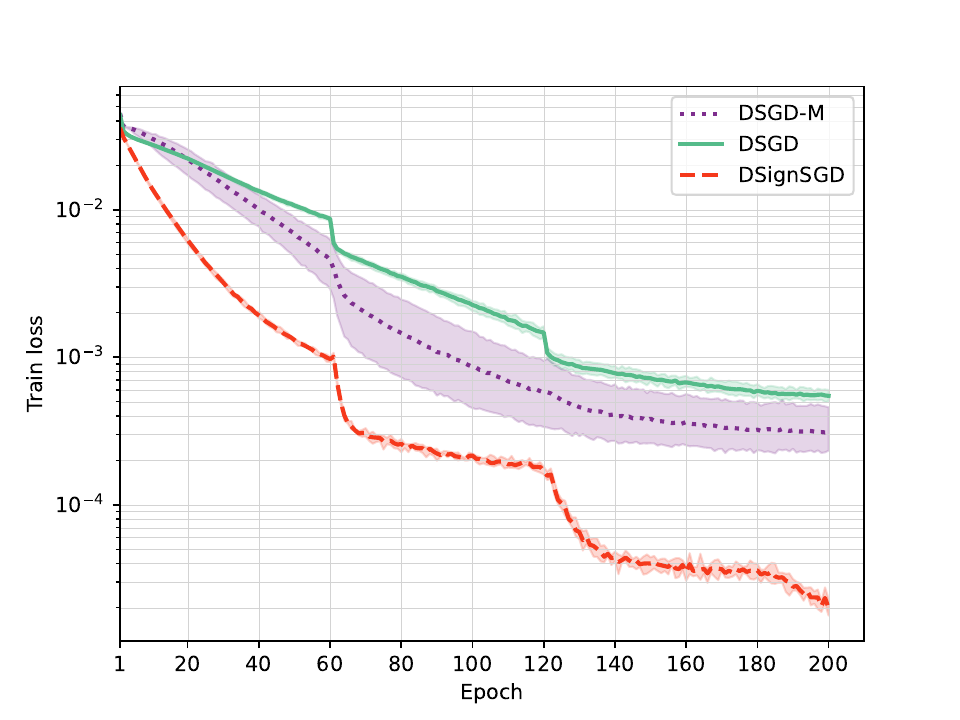}%
\label{fig_7_2}}
\hfil
\subfloat[Consensus Error]{\includegraphics[width=1.8in]{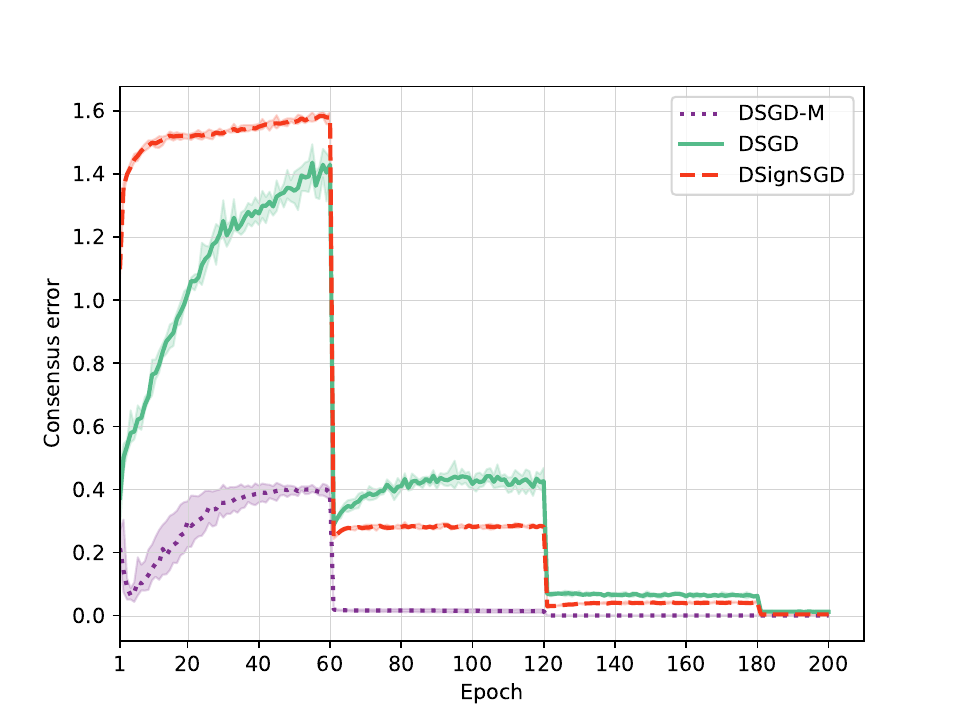}%
\label{fig_7_3}}

\caption{Numerical performance comparison of DSGD, DSGD-M, and DSignSGD in training ResNet50 on CIFAR-100 dataset using with-replacement sampling strategy.}
\label{fig:dsgd3}
\end{figure*}

\begin{figure*}[!t]
\centering
\subfloat[Test Accuracy]{\includegraphics[width=1.8in]{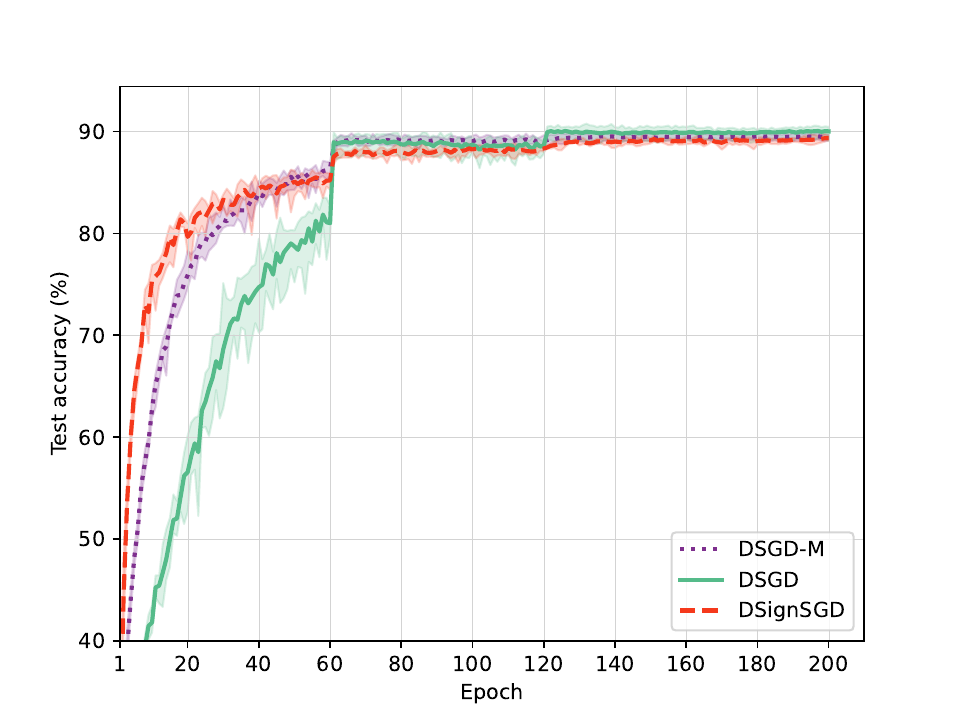}%
\label{fig_6_1}}
\hfil
\subfloat[Train Loss]{\includegraphics[width=1.8in]{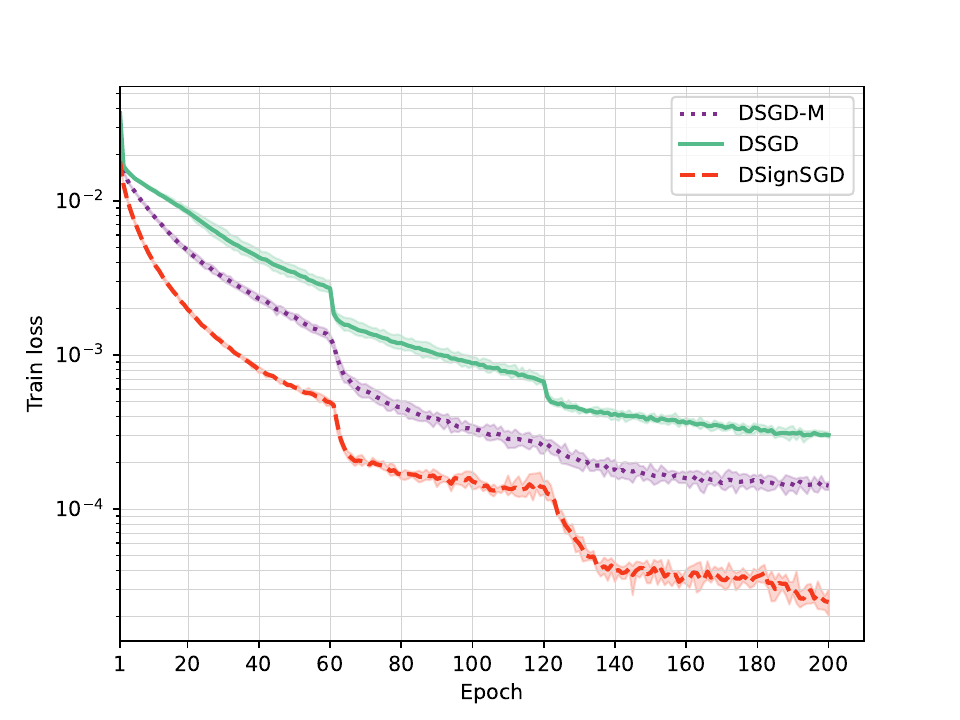}%
\label{fig_6_2}}
\hfil
\subfloat[Consensus Error]{\includegraphics[width=1.8in]{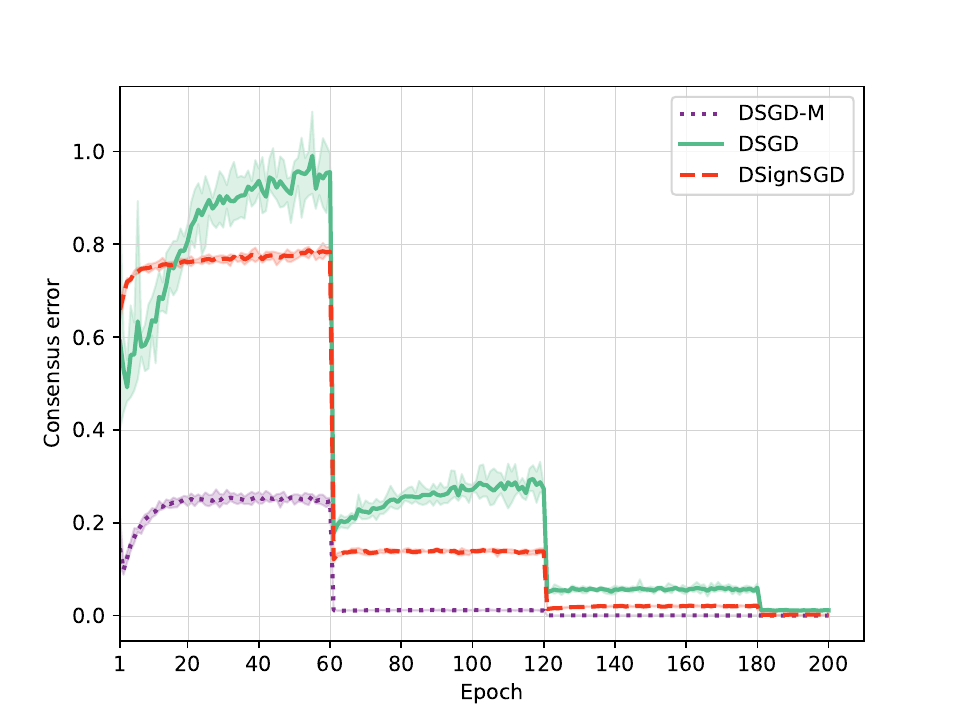}%
\label{fig_6_3}}

\caption{Numerical performance comparison of DSGD, DSGD-M, and DSignSGD in training ResNet50 on CIFAR-10 dataset using with-replacement sampling strategy.}
\label{fig:dsgd4}
\end{figure*}
}

In this section, we evaluate the numerical performance of our analyzed decentralized stochastic subgradient-based methods in Section \ref{sec:4} for decentralized training nonsmooth neural networks. All numerical experiments in this section are conducted on 
a workstation equipped with two Intel(R) Xeon(R) Processors Gold 5317 CPUs (at 3.00GHz, 18M Cache) and 8 NVIDIA GeForce RTX 4090 GPUs under Ubuntu 20.04.1. We implement all decentralized algorithms with Python 3.8 and PyTorch 1.13.0 using NCCL 2.14.3 (CUDA 11.7) as the communication backend. 

\subsection{Experiment settings}
We conduct a series of image classification tasks with CIFAR-10 and CIFAR-100 datasets \cite{krizhevsky2009learning}, which both consist of $50000$ training samples and $10000$ test samples. ResNet-50 \cite{he2016deep} is employed as our nonsmooth neural network. The decentralized network topology is set as ring-structure with $8$ agents, and the mixing matrix is chosen as the Metropolis weight matrix \cite{xiao2006distributed}. In each experiment, we divide the training dataset into 8 equal parts, with each part serving as the local dataset for each agent. During any epoch of training, the local dataset of each agent itself remains unchanged, while the mini-batch is selected by random reshuffling { or with-replacement sampling} to ensure compliance with the sampling scheme in Assumption \ref{Assumption_app1} { or \ref{Assumption_app2}}. 

Moreover, we set the number of epochs as $200$ and choose the batch size as $128$. The strategy for selecting step-size is designed to reduce it three times, specifically at the 60-th, 120-th, and 160-th epochs, with a decay factor of 0.2. 

\subsection{Performance comparison between DSGD, DSGD-M, and DSignSGD}

We first evaluate the numerical performance of DSGD (Algorithm \ref{alg:dsgd}), DSGD-M (Algorithm \ref{alg:DSGD-M} with $\phi = \frac{1}{2}\|{\bm y}\|^{2}$), and 
DSignSGD (Algorithm \ref{alg:DSGD-M} with $\phi=\norm{{\bm y}}_1$). All compared methods are executed five times with varying random seeds.  In each test instance, all these compared methods are initialized with the same
randomly generated initial points. 

In Algorithm \ref{alg:DSGD-M}, we choose the momentum parameter $\tau$ as $\frac{0.1}{\eta_{0}}$, and the initial descent direction is taken as ${\bm Y}_{0}= (1-\tau \eta_{0}) {\bm D}_{0}{\bm W} + \tau \eta_{0} {\bm D}_{0}$. Additionally, we set the initial step-size $\eta_{0}$ of DSGD and DSGD-M to be $0.2$, and the counterpart of DSignSGD to be $10^{-4}$.

We present the numerical results of the above-mentioned decentralized stochastic sub-gradient-based methods in { Figures \ref{fig:dsgd}- \ref{fig:dsgd4}}, including the test accuracy, train loss, and consensus error { under two different sampling strategies.} It is worth mentioning that the consensus error at the $k$-th iteration is measured by $\frac{1}{\sqrt{d}}\|{\bm X}_{k}({\bm I}_{d}-{\bm P})\|_{\mathrm{F}}$. As illustrated in { Figures \ref{fig:dsgd}- \ref{fig:dsgd4}}, all the compared decentralized stochastic subgradient-based methods are able to train the ResNet-50 network to a high accuracy. Moreover, we can conclude from Figure \ref{fig:dsgd}(a) and \ref{fig:dsgd2}(a) that { when we employ the random reshuffling strategy, DSignSGD achieves similar test accuracy as DSGD and better test accuracy compared to DSGD-M, while showing superior performance to DSGD and DSGD-M in the aspect of train loss from Figure \ref{fig:dsgd}(b) and \ref{fig:dsgd2}(b). When using with-replacement sampling, DSGD achieves  higher test accuracy than that of  DSGD-M and DSignSGD, but exhibits the worst training-loss-performance from Figure \ref{fig:dsgd3}- \ref{fig:dsgd4}. }

Furthermore, in all the compared methods, the consensus error diminishes to $0$ throughout the iterations. These results illustrate that our framework is able to yield efficient stochastic subgradient-based methods for decentralized training tasks.

\subsection{Performance comparison between DSGD and DSGD-T}

In this subsection, we evaluate the numerical performance of DSGD (Algorithm \ref{alg:dsgd}) and DSGD-T (Algorithm \ref{alg:dsgt}).
Similar to the settings in Section 5.2, we repeat the numerical experiments for $5$ times with varying random seeds for DSGD and DSGD-T. For each random seed, DSGD and DSGD-T are executed with the same initial point ${\bm x}_0$. The initial step-size is taken as $0.2$ for both algorithms.

The numerical results on the CIFAR-100 and CIFAR-10 datasets are presented in {Figures \ref{fig:dsgt}- \ref{fig:dsgt4}}, respectively. It can be concluded that DSGD-T demonstrates slightly better performance than DSGD in all aspects { under all sampling strategies}, especially in train loss. Therefore, we can conclude that DSGD-T achieves high efficiency in decentralized training tasks while enjoying convergence guarantees from our framework \eqref{Eq_Framework}.

\begin{figure*}[!t]
\centering
\subfloat[Test Accuracy]{\includegraphics[width=1.8in]{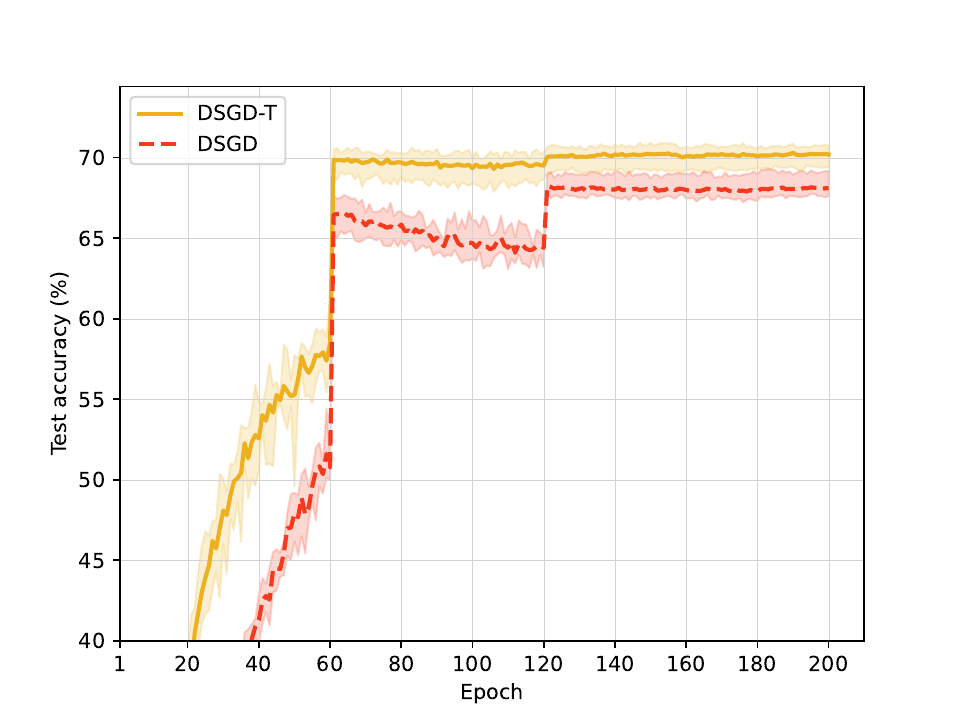}%
\label{fig_3_1}}
\hfil
\subfloat[Train Loss]{\includegraphics[width=1.8in]{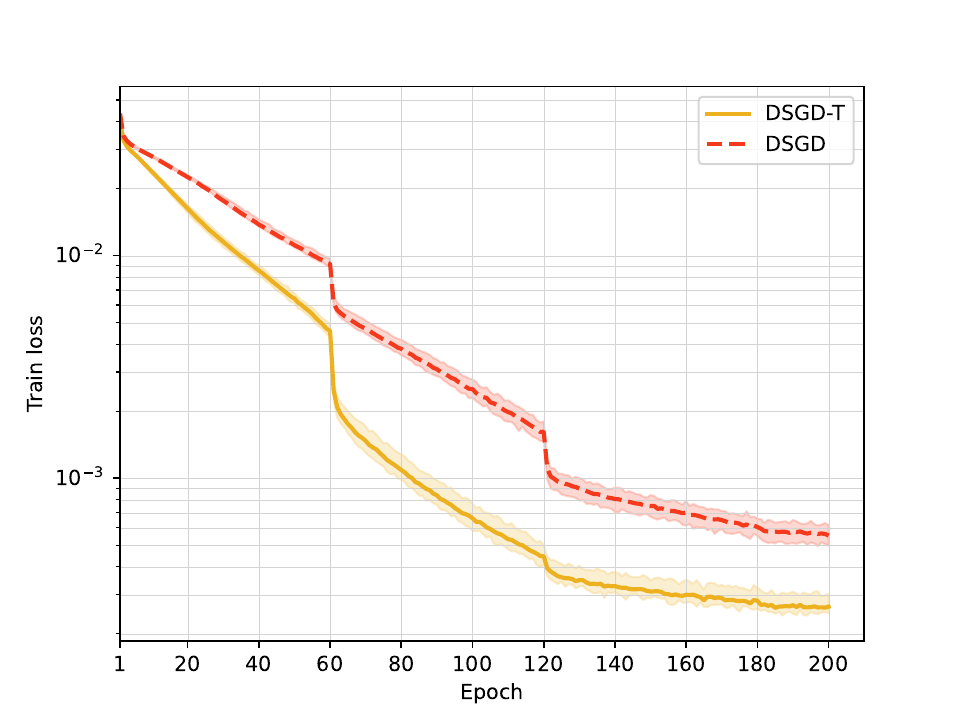}%
\label{fig_3_2}}
\hfil
\subfloat[Consensus Error]{\includegraphics[width=1.8in]{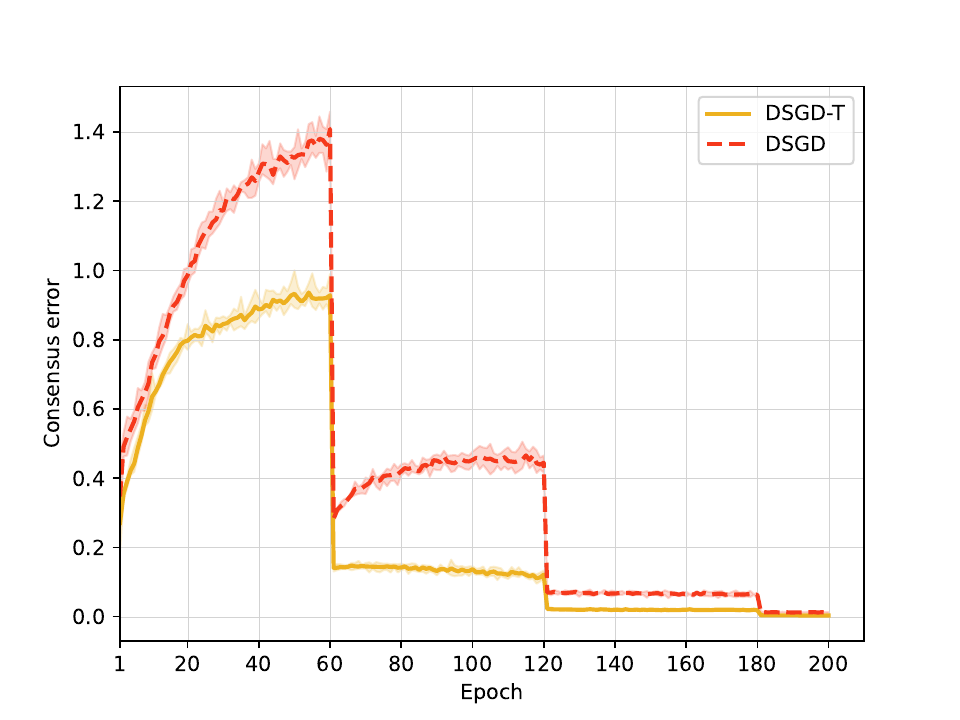}%
\label{fig_3_3}}

\caption{Numerical performance comparison of DSGD and DSGD-T in training ResNet50 on CIFAR-100 dataset using random reshuffling strategy.}
\label{fig:dsgt}
\end{figure*}

\begin{figure*}[!t]
\centering
\subfloat[Test Accuracy]{\includegraphics[width=1.8in]{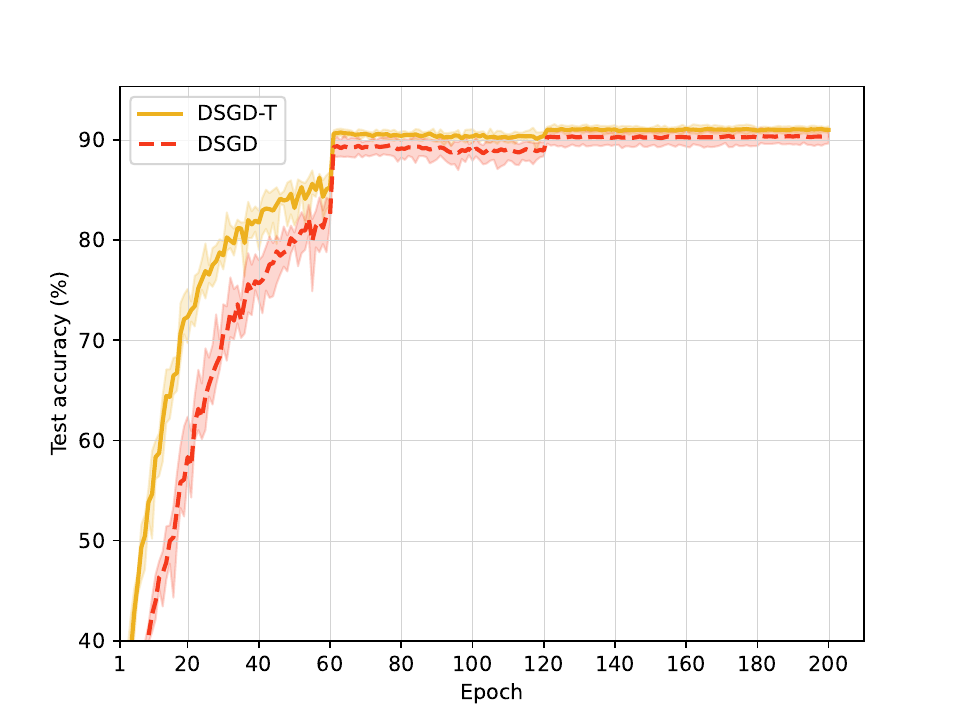}%
\label{fig_4_1}}
\hfil
\subfloat[Train Loss]{\includegraphics[width=1.8in]{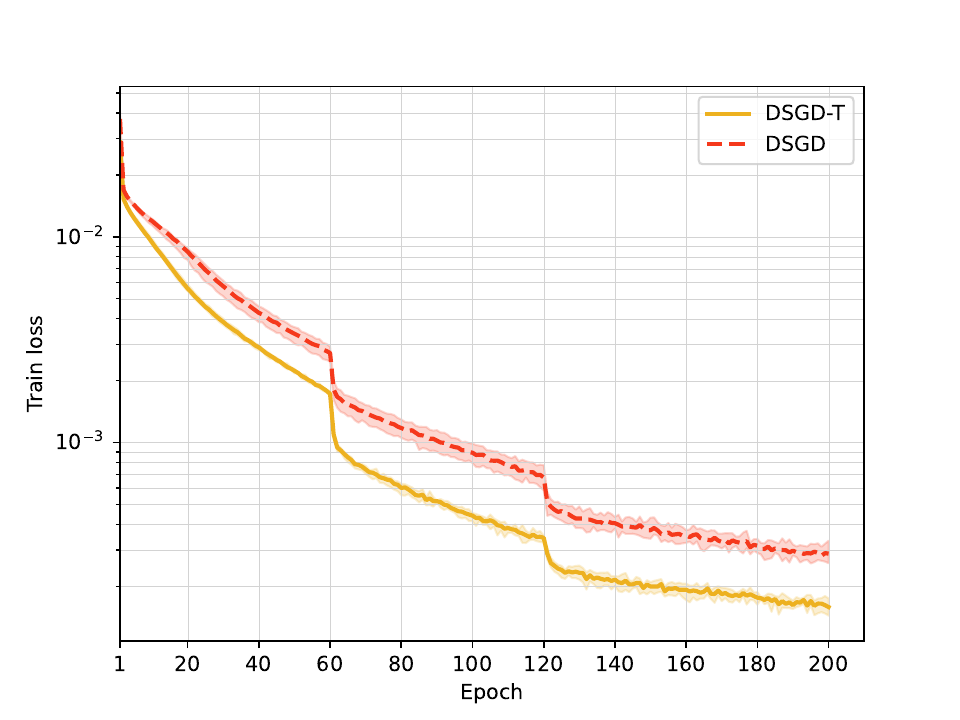}%
\label{fig_4_2}}
\hfil
\subfloat[Consensus Error]{\includegraphics[width=1.8in]{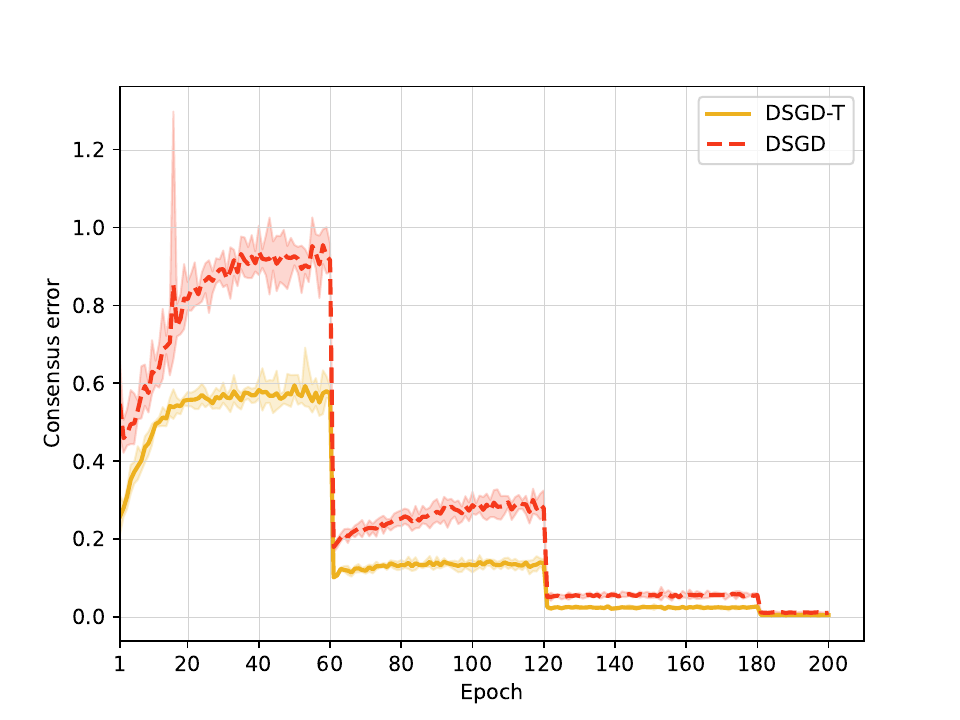}%
\label{fig_4_3}}

\caption{Numerical performance comparison of DSGD and DSGD-T in training ResNet50 on CIFAR-10 dataset using random reshuffling strategy.}
\label{fig:dsgt2}
\end{figure*}

{
\begin{figure*}[!t]
\centering
\subfloat[Test Accuracy]{\includegraphics[width=1.8in]{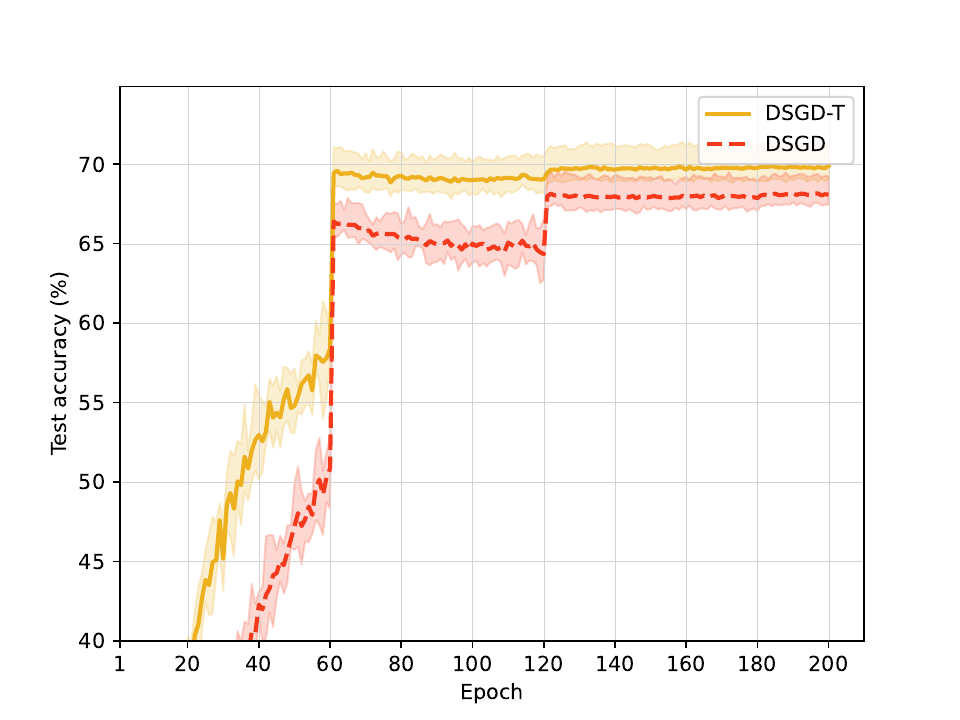}%
\label{fig_8_1}}
\hfil
\subfloat[Train Loss]{\includegraphics[width=1.8in]{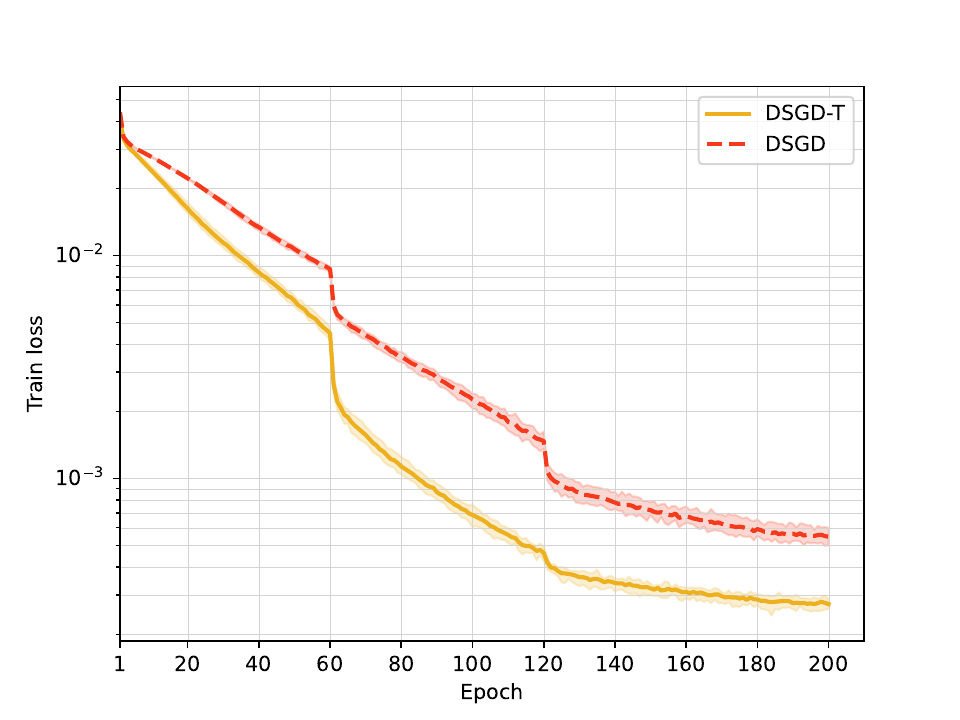}%
\label{fig_8_2}}
\hfil
\subfloat[Consensus Error]{\includegraphics[width=1.8in]{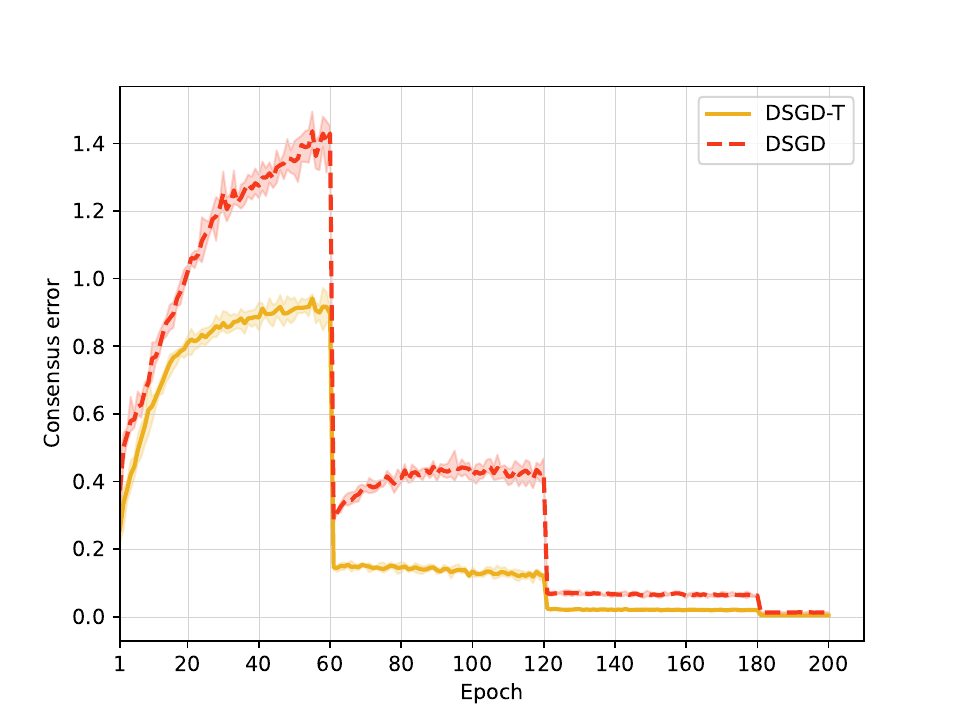}%
\label{fig_8_3}}

\caption{Numerical performance comparison of DSGD and DSGD-T in training ResNet50 on CIFAR-100 dataset using with-replacement sampling strategy.}
\label{fig:dsgt3}
\end{figure*}

\begin{figure*}[!t]
\centering
\subfloat[Test Accuracy]{\includegraphics[width=1.8in]{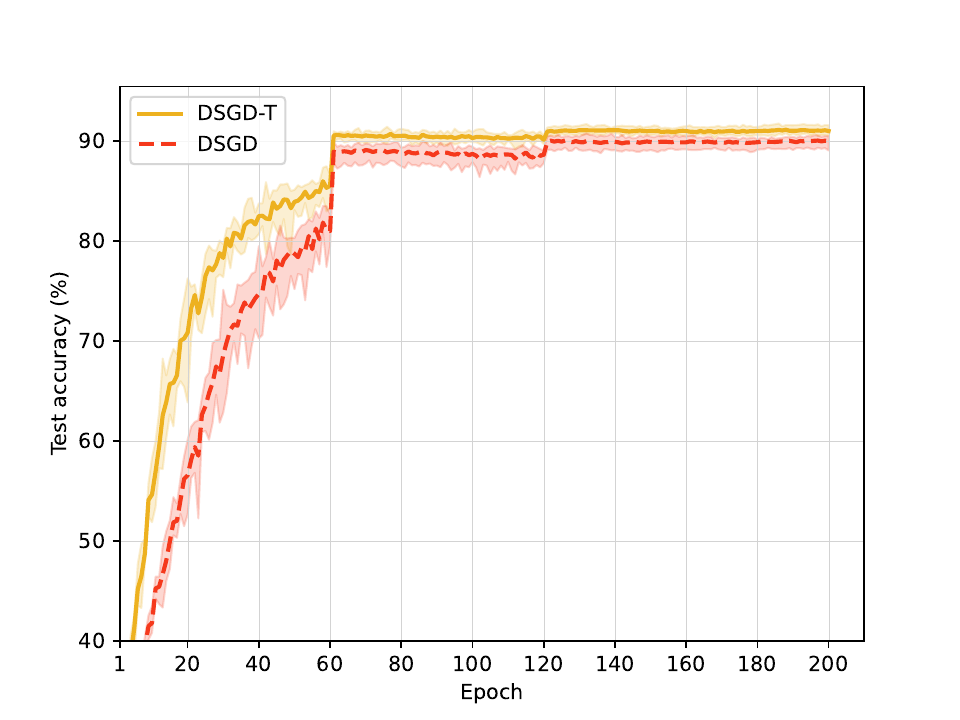}%
\label{fig_9_1}}
\hfil
\subfloat[Train Loss]{\includegraphics[width=1.8in]{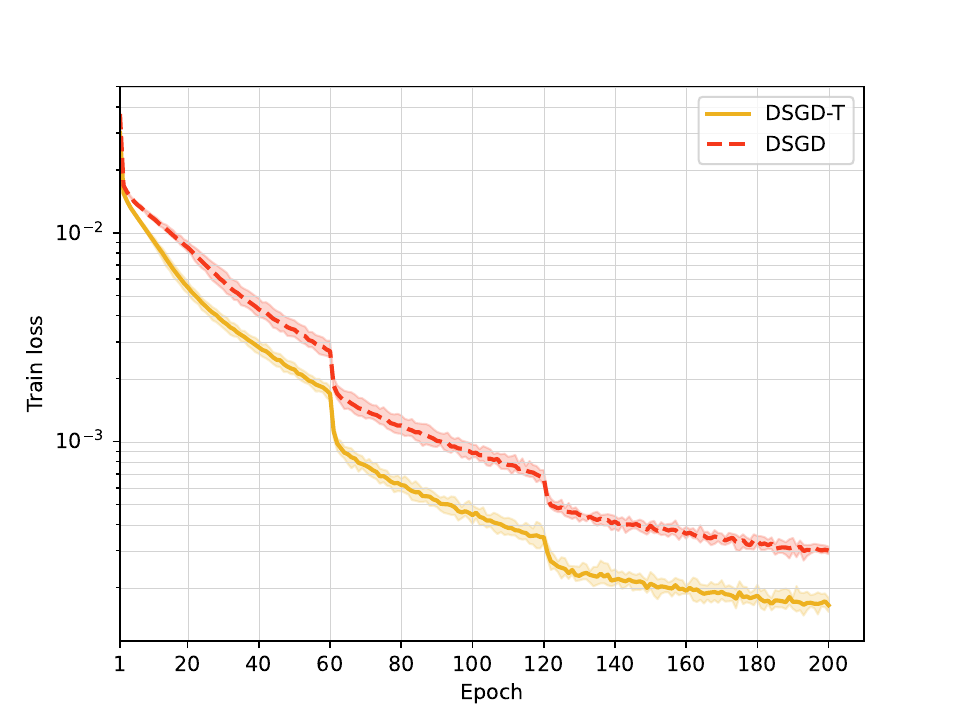}%
\label{fig_9_2}}
\hfil
\subfloat[Consensus Error]{\includegraphics[width=1.8in]{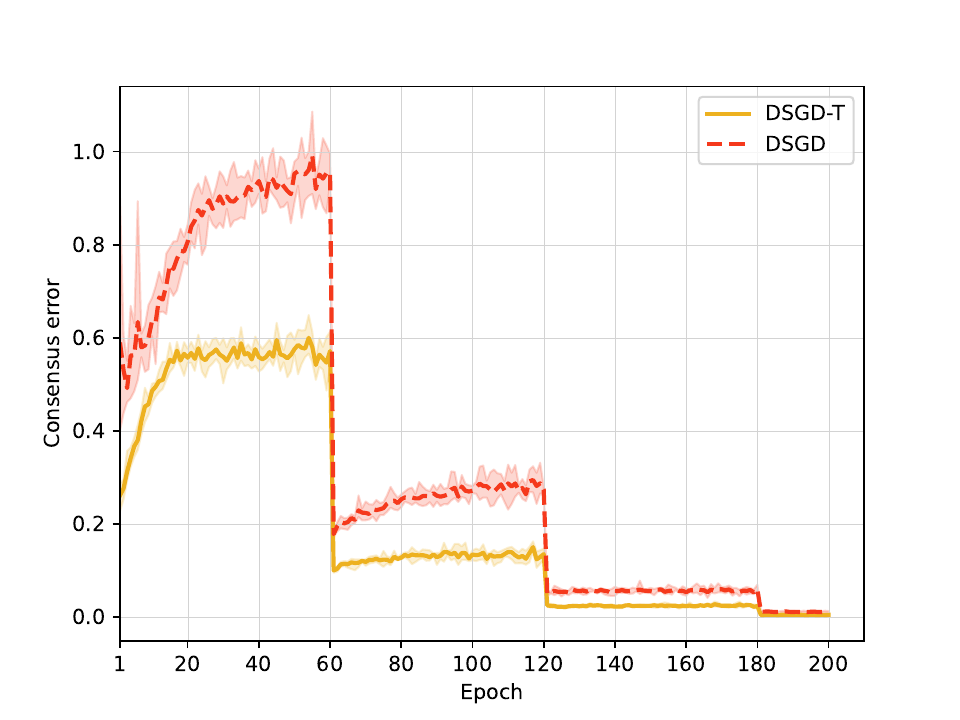}%
\label{fig_9_3}}

\caption{Numerical performance comparison of DSGD and DSGD-T in training ResNet50 on CIFAR-10 dataset using with-replacement sampling strategy.}
\label{fig:dsgt4}
\end{figure*}
}

\section{Conclusion}\label{sec:7}

Decentralized stochastic subgradient-based methods play an important role in the decentralized training of deep neural networks, where the widely employed nonsmooth building blocks result in non-Clarke regular loss functions. As nearly all of the existing works on decentralized optimization focus on the Clarke regular functions, there is a gap between theoretical analysis and implementation in real-world training tasks. Additionally, despite the development of various acceleration techniques to improve the performance of decentralized stochastic subgradient methods, the convergence guarantees of these methods in training tasks are largely absent.

To address these issues, we propose a unified framework  \eqref{Eq_Framework} for analyzing the global convergence of  stochastic subgradient-based methods in nonsmooth decentralized optimization. We conduct consensus analysis and relate the average iterates of  \eqref{Eq_Framework} to the trajectories of its corresponding noiseless differential inclusion. Then, under two different evaluation noise settings, we establish convergence properties of \eqref{Eq_Framework}, in the sense that $\{{\bm z}_{i,k}\}$ asymptotically reaches consensus and converges to the stable set of the corresponding noiseless differential inclusion.

Moreover, we show our proposed framework \eqref{Eq_Framework} is general enough to enclose the implementations of various acceleration techniques for decentralized stochastic subgradient methods, including momentum and gradient-tracking techniques. In particular, we demonstrate that a wide range of decentralized stochastic subgradient-based methods, such as DSGD, DSGD-M, and DSGD-T, together with two common sampling schemes, are all encompassed by  our proposed framework. Therefore, our theoretical results, for the first time, provide the global convergence guarantees for these decentralized stochastic subgradient-based methods in nonsmooth nonconvex optimization. Furthermore, based on our proposed framework, we develop a new method named DSignSGD, where sign-mapping is employed to regularize the update directions. Preliminary numerical results verify the validity of our developed theoretical results and exhibit the high efficiency of these decentralized stochastic subgradient-based methods enclosed in our framework \eqref{Eq_Framework}.

Our work sets several directions for future research in the field of nonsmooth nonconvex decentralized optimization. {  Existing ODE approaches  are limited to establishing asymptotic convergence properties or the convergence rate for the continuous-time trajectories of the corresponding differential inclusion \cite{NEURIPS2022_a9077da4, bolte2025inexact}. Consequently, deriving explicit convergence rates for various nonsmooth decentralized stochastic subgradient methods (such as DSGD, DSGT, and their momentum variants) remains a significant challenge. While our analysis establishes the convergence of DSGT in nonsmooth nonconvex settings, it reveals that DSGT and DSGD are corresponding to the same limiting differential inclusion. Empirically, however, Figures \ref{fig:dsgt}-\ref{fig:dsgt4} demonstrate that DSGT consistently converges faster than DSGD. Bridging this gap between theory and practice, especially in establishing explicit convergence rates that distinguish between DSGT and DSGD in the nonsmooth nonconvex regime, remains a critical direction for future research. }

Further attention is also needed on the convergence of decentralized methods with time-varying networks, as well as those that incorporate asynchronous or compressed communication.  Additionally, we are also eager to explore routes that presume relaxed conditions for convergence analysis of decentralized methods in solving \eqref{Prob_DOP}. For instance, closed-measure approaches \cite{bianchi2021closed,bolte2022long,le2024nonsmooth}  are developed for establishing the convergence properties of stochastic subgradient-based methods, with relaxed conditions on the step-sizes and evaluation noises. How to employ these closed-measure approaches for analyzing decentralized stochastic subgradient-based methods is worthy of investigating.

\section*{Acknowledgement}
The work of Professor Xin Liu was supported in part by the National Key R\&D Program of China (2023YFA1009300), the National Natural Science Foundation of China (12125108, 12226008, 12021001, 12288201), Key Research Program of Frontier Sciences, Chinese Academy of Sciences (ZDBS-LY-7022), and CAS-Croucher Funding Scheme for Joint Laboratories ``CAS AMSS-PolyU Joint Laboratory of Applied Mathematics: Nonlinear Optimization Theory, Algorithms and Applications''.

\newpage

\appendix

\section{Proofs of Theorem \ref{thm:DSGD2} and \ref{thm:DSGDM2}}\label{appendix:1}

In this part, we devote to showing the avoidance of spurious critical points for Algorithm \ref{alg:dsgd} and \ref{alg:DSGD-M}. In other words, these algorithms can avoid points that are critical for a specific conservative field but are not Clarke-critical points. We begin with the definition of the concept of almost everywhere $\ca{C}^1$ set-valued mapping. 

\begin{defin}
	\label{Defin_AS_C1}
	A measurable mapping $q: \bb{R}^{nd} \to \bb{R}^{nd}$ is almost everywhere $\ca{C}^1$ if for almost every ${\bm z} \in \bb{R}^{nd}$, $q$ is locally continuously differentiable in 
	a neighborhood of $z$. 
	
	Moreover, a set-valued mapping $\ca{Q} : \bb{R}^{nd} \rightrightarrows \bb{R}^{nd}$ is almost everywhere $\ca{C}^1$ if there exists an almost everywhere $\ca{C}^1$ mapping  $q:\bb{R}^{nd} \to \bb{R}^{nd}$ such that for almost every ${\bm z} \in \bb{R}^{nd}$, $\ca{Q}({\bm z}) = \{q({\bm z})\}$. 
\end{defin}

\subsection{Proof of Theorem \ref{thm:DSGD2}}\label{sec:app1}

For ease of presentation, we define 
\begin{equation*}
F_{i, \mathcal{B}_{i,k+1}}({\bm x}):= \frac{1}{|\mathcal{B}_{i, k+1}|}\sum_{{\bm s}_l\in\mathcal{B}_{i, k+1}}F_{i}({\bm x}; {\bm s}_{l}), \quad \text{ and } \quad  D_{F_{i},\mathcal{B}_{i,k+1}}({\bm x}):= \frac{1}{|\mathcal{B}_{i, k+1}|}\sum_{{\bm s}_l\in\mathcal{B}_{i, k+1}}D_{F_{i}}({\bm x}; {\bm s}_{l}).
\end{equation*}
Define the set-valued mappings $\ca{U}_{k}: \mathbb{R}^{nd} \rightrightarrows \mathbb{R}^{nd}$ as  
	\begin{equation*}
\begin{aligned}
\ca{U}_{k}({\bm x}) & := \left(D_{F_{1},\mathcal{B}_{1,k+1}}({\bm x}_{(1)}), \ldots, D_{F_{d},\mathcal{B}_{d,k+1}}({\bm x}_{(d)})\right)^{\top}  \\ 
            \end{aligned}
	\end{equation*}
Note that $(F_{i}(\cdot,{\bm s}_{l}),\D_{F_{i}}(\cdot, {\bm s}_l))$ has a $\ca{C}^r$ variational stratification \cite[Theorem 4]{bolte2021conservative}, hence the mappings $\{\ca{U}_{k}\}$ are almost everywhere $\ca{C}^1$ over $\mathbb{R}^{nd}$, denoted by $\ca{U}_{k}({\bm x})= \{ u_{k}({\bm x})\},$ for a.e. ${\bm x}\in \mathbb{R}^{nd}$. Furthermore, let 
\begin{equation*}
    \begin{aligned}
    \ca{Y} & := \left\{{\bm x} \in \mathbb{R}^{nd}: D_{F_{i}}({\bm x}; {\bm s}_{l}) \neq \{\nabla F_i({\bm x}, {\bm s}_{l})\}, \text{ for some } i\in [N] \text{ and some } {\bm s}_{l} \in \mathcal{S}_{i} \right\},\\
        \end{aligned}
\end{equation*}
which is a zero-measure subset. 

With a slight abuse of notation, we vectorize the sequence $\{{\bm X}_{k}\}$ by column-major order, and denote as ${\bm x}_{k} \in \mathbb{R}^{nd}$. Then  the vectorized sequence $\{{\bm x}_{k}\}$ generated by Algorithm \ref{alg:dsgd} satisfies
	\begin{equation}\label{eq:avoidance_update}
    \begin{aligned}
{\bm x}_{k+1} & \in{} ({\bm W}\otimes {\bm I}_{n}){\bm x}_{k} - c \hat{\eta}_k\ca{U}_{k}({\bm x}_{k}). \\
            \end{aligned}
	\end{equation}
Set ${\bm P}:= {\bm W}\otimes {\bm I}_{n}$,  $\mathcal{Q}_{k}:= \hat{\eta}_{k} \mathcal{U}_{k}$, and $q_{k}:= \hat{\eta}_{k} {u}_{k}$. The Jacobian of the mapping ${\bm P}{\bm x} - c q_{k}({\bm x})$ is 
\begin{equation}\label{Jacobian0}
{\bm J}_{c,k}({\bm x}) = {\bm W}\otimes {\bm I}_{n} - c \hat{\eta}_{k}  \Diag(\nabla^2 F_{1, \mathcal{B}_{1,k+1}}({\bm x}_{(1)}), \ldots, \nabla^2 F_{d, \mathcal{B}_{d,k+1}}({\bm x}_{(d)})).
\end{equation}
The determinant of ${\bm J}_{c,k}({\bm x})$ is a zero polynomial of $c$ if and only if there exists ${\bm v}= [{\bm v}_1^{\top}, \ldots, {\bm v}_{d}^{\top}]^{\top}$ such that  ${\bm v} \in \mathrm{Ker}({\bm W}\otimes {\bm I}_n)$ and ${\bm v}_{i} \in \mathrm{Ker}(\nabla^2 F_{i, \mathcal{B}_{i,k+1}}({\bm x}_{(i)}))$ for any $i\in [d]$. As each $F_{i}(\cdot, {\bm s}_{l})$ is a definable function, we can infer that such $({\bm x}_{(1)}, \ldots, {\bm x}_{(d)})$ constitutes a zero-measure subset of $\mathbb{R}^{nd}$. We refer to  $\Gamma_{k}$ as the full-measure subset of $\mathbb{R}^{nd}$, in which $\ca{U}_{k}({\bm x})= \{ u_{k}({\bm x})\},$ and $\det({\bm J}_{c,k}({\bm x}))$ is a non-trival $nd$-th order polynomial of $c$. This indicates the set
\begin{equation*}
    \{c \in \mathbb{R}:\det({\bm J}_{c,k}({\bm x})) =0 \}
\end{equation*}
is zero-measure in $\mathbb{R}$. According to Fubini's Theorem, there exists a full-measure subset $\ca{S}_k \subseteq \bb{R}_+$, such that for any $c \in \ca{S}_k$, $\{{\bm x} \in \Gamma_{k}:\det({\bm J}_{c,k}({\bm x})) \neq 0 \}$ is a full-measure subset in $\mathbb{R}^{nd}$. 

Applying the inverse function theorem, we can conclude that for any $c \in \ca{S}_k$ and any ${\bm x} \in \{x \in \Gamma_{k}:\det({\bm J}_{c,k}({\bm x})) \neq 0 \}$, ${\bm P}{\bm x} - c q_{k}({\bm x})$ is a  local diffeomorphism in a neighborhood of ${\bm x}$, denoted by $\mathbb{B}({\bm x}, \delta_{{\bm x}})$. As a result,
\begin{equation*} 
\{({\bm P} - c q_{k})(\mathbb{B}({\bm x}^i, \delta_{{\bm x}^i}))\}_{{\bm x} \in \{x \in \Gamma_{k}:\det({\bm J}_{c,k}({\bm x})) \neq 0 \}}
\end{equation*}
is an open covering of $\{{\bm x} \in \Gamma_{k}:\det({\bm J}_{c,k}({\bm x})) \neq 0 \}$. Based on Lindelof's lemma \cite{kelley2017general}, there exists countable $\{{\bm x}^i\}_{i \in \mathbb{N}_{+}}\subseteq \{{\bm x} \in \Gamma_{k}:\det({\bm J}_{c,k}({\bm x})) \neq 0 \}$, such that 
\begin{equation*}
    \{{\bm x} \in \Gamma_{k}:\det({\bm J}_{c,k}({\bm x})) \neq 0 \} \subseteq \bigcup_{i \in \mathbb{N}_{+}} ({\bm P} - c q_{k})(\mathbb{B}({\bm x}^i, \delta_{{\bm x}^i})).
\end{equation*}

Given any zero-measure subset $A \subseteq \mathbb{R}^{nd}$, since ${\bm P} - c q_{k}$ is a diffeomorphism over $\mathbb{B}({\bm x}^i, \delta_{{\bm x}^i})$, it can be deduced that the set $\{{\bm z} \in \mathbb{B}({\bm x}^i, \delta_{{\bm x}^i}), {\bm P}{\bm z} - c q_{k}({\bm z}) \in A\}$ is zero-measure. Therefore, the set $\bigcup_{i\in\mathbb{N}_{+}}\{{\bm z} \in \mathbb{B}({\bm x}^i, \delta_{{\bm x}^i}), {\bm P}{\bm z} - c q_{k}({\bm z}) \in A\}$ is a zero-measure subset, which yields that 
\begin{equation*}
\{{\bm z} \in \mathbb{R}^{nd}: {\bm P}{\bm z} - c q_{k}({\bm z}) \in A\} \subset  \bigcup_{i\in\mathbb{N}_{+}}\{{\bm z} \in \mathbb{B}({\bm x}^i, \delta_{{\bm x}^i}), {\bm P}{\bm z} - c q_{k}({\bm z}) \in A\} \cup  \{{\bm z} \in \Gamma_{k}:\det({\bm J}_{c,k}({\bm z})) \neq 0 \}^c  
\end{equation*}
is zero-measure. That is to say, for any $c \in \ca{S}_k$, the set $\{{\bm z} \in \mathbb{R}^{nd}: {\bm P}{\bm z} - c q_{k}({\bm z}) \in A\}$ is zero-measure. 

Let $\ca{S}:= \cap_{k\geq p}\ca{S}_k$, $A_{0}:= A = \mathcal{Y}, A_{k+1}:= \{{\bm x}\in \mathbb{R}^{nd}: {\bm P}{\bm z} - c q_{k}({\bm z}) \in A\}$, and $\hat{A}:=\cup_{k\geq 0} A_{k}$. It can be inferred that $\ca{S}$ and $\hat{A}^c$ are full-measure subsets in $\mathbb{R}$ and $\mathbb{R}^{nd}$, respectively. For any $c \in \ca{S}$, and any ${\bm x} \in \hat{A}^c$, we have ${\bm P}{\bm x} - c q_{k}({\bm x}) \in A^c $. Employing Fubini's theorem, the set $\{c \in\mathbb{R}, {\bm x}\in \mathbb{R}^{nd}, \cup_{k\geq 0} \{{\bm P}{\bm x} - c q_{k}({\bm x})\}\subseteq A^c\}$ is full-measure in $\mathbb{R}\times \mathbb{R}^{nd}$, and thus, there exists a full-measure subset $\hat{\ca{S}} \times \hat{\ca{K}}$,  whenever $(c, {\bm x}_0) \in \hat{\ca{S}} \times \hat{\ca{K}}$, it follows that $\cup_{k\geq 0} \{{\bm P}{\bm x} - c q_{k}({\bm x})\}\subseteq A^c$, i.e. $\{{\bm x}_{k}\} \in \mathcal{Y}^c$.

Furthermore, the update scheme of vectorized sequence $\{{\bm x}_k\}$ can be reformulated as
\begin{equation*}
		\left\{
		\begin{aligned}
			{\bm g}_k ={}& \left(\nabla F_{1,\mathcal{B}_{1,k+1}}({\bm x}_{1, k}), \ldots, \nabla F_{d,\mathcal{B}_{d,k+1}}({\bm x}_{d, k})\right)^{\top} , \\
			{\bm x}_{k+1} & \in{} ({\bm W}\otimes {\bm I}_{n}){\bm x}_{k} - c \hat{\eta}_k {\bm g}_{k}
		\end{aligned}
		\right.
	\end{equation*}
Therefore, by applying Theorem \ref{thm:dsgd} with $\D_f = \partial f$, we can choose $\ca{S} = \hat{\ca{S}} \cap (0,\alpha_c)$ and $\ca{K} = \hat{\ca{K}} \cap \X_0$, such that the claim in Theorem \ref{thm:DSGD2} holds.

\subsection{Proof of Theorem \ref{thm:DSGDM2}}

Let the set-valued mappings $\ca{U}_{1, k}: \mathbb{R}^{2nd} \rightrightarrows \mathbb{R}^{nd}$, $\ca{U}_{2, k}: \mathbb{R}^{2nd} \rightrightarrows \mathbb{R}^{2nd}$ be defined as  
	\begin{equation*}
\begin{aligned}
\ca{U}_{1,k}({\bm x}, {\bm y}) & :=\left(\partial \hat{\phi}({\bm y}), {\bm 0} \right)\\
\ca{U}_{2,k}({\bm x}, {\bm y}) & :=\left({\bm 0}, \tau ({\bm W}\otimes {\bm I}_n) {\bm y} - \tau \left(D_{F_{1},\mathcal{B}_{1,k+1}}({\bm x}_{(1)}), \ldots, D_{F_{d},\mathcal{B}_{d,k+1}}({\bm x}_{(d)})\right)^{\top}\right)\\
            \end{aligned}
	\end{equation*}
where $\hat{\phi}({\bm y}):=(\phi({\bm y}_{(1)}), \ldots, \phi({\bm y}_{(d)}))^{\top}$, $\phi: \mathbb{R}^{n} \to \mathbb{R}$ is convex and admits a unique minimizer at ${\bm 0}$. Notice that $(F_{i}(\cdot,{\bm s}_{l}),\D_{F_{i}(\cdot, {\bm s}_{l})})$ has a $\ca{C}^r$ variational stratification \cite[Theorem 4]{bolte2021conservative}, and $\hat{\phi}$ is differentiable on a full-measure subset of $\mathbb{R}^{nd}$, thus the mappings $\{\ca{U}_{j,k}\}, j=1,2$ are almost everywhere $\ca{C}^1$ over $\mathbb{R}^{2nd}$, the corresponding almost-everywhere-$\mathcal{C}^1$ mappings are denoted by $u_{j,k}({\bm z}), j=1,2$. Furthermore, let 
	\begin{equation*}
    \begin{aligned}
		\ca{Y} & := \left\{({\bm x}, {\bm y}) \in \mathbb{R}^{2nd}: \hat{\phi}({\bm y}) \text{ is not differentiable }, \D_{F_{i, l}}({\bm x}) \neq \{\nabla F_i({\bm x}, {\bm s}_{l})\}, \text{ for some } i\in [N] \right. \\
        & \left. \text{ and some } {\bm s}_{l} \in \mathcal{S}_{i} \right\},\\
        \end{aligned}
	\end{equation*}
which is a zero-measure subset in $\mathbb{R}^{2nd}$.

With the same notations in Section \ref{sec:app1}, the vectorized sequence $\{({\bm x}_{k}, {\bm y}_{k})\}$ generated by Algorithm \ref{alg:DSGD-M} satisfies
\begin{equation}\label{eq:avoidance_update1}
\begin{aligned}
({\bm x}_{k+\frac{1}{2}}, {\bm y}_{k + \frac{1}{2}}  ) & \in{} ({\bm W}\otimes {\bm I}_{n}\times {\bm I}_{nd})({\bm x}_{k}, {\bm y}_{k}) - c \hat{\eta}_k\ca{U}_{1,k}({\bm x}_{k}, {\bm y}_{k}),\\
\left({\bm x}_{k+1}, {\bm y}_{k + 1}  \right) & \in{} ({\bm I}_{nd} \times {\bm W}\otimes {\bm I}_{n})({\bm x}_{k+\frac{1}{2}}, {\bm y}_{k+\frac{1}{2}}) - c \hat{\eta}_k\ca{U}_{2,k}({\bm x}_{k+ \frac{1}{2}}, {\bm y}_{k+ \frac{1}{2}}),\\        \end{aligned}
\end{equation}
Set ${\bm P}_{1}:= {\bm W}\otimes {\bm I}_{n}\times {\bm I}_{nd}$, ${\bm P}_{2}:= {\bm I}_{nd} \times {\bm W}\otimes {\bm I}_{n}$, then the update scheme of vectorized sequence $\{({\bm x}_{k}, {\bm y}_{k})\}$ is 
\begin{equation*}
\begin{aligned}
({\bm x}_{k+1}, {\bm y}_{k + 1}) &{} = {\bm P}_2 ({\bm P}_1({\bm x}_{k}, {\bm y}_{k}) ) - c {\bm P}_2 (\hat{\eta}_k\ca{U}_{1,k}({\bm x}_{k}, {\bm y}_{k})) - c \hat{\eta}_k\ca{U}_{2,k}({\bm P}_{1}({\bm x}_{k}, {\bm y}_{k})- c \hat{\eta}_k\ca{U}_{1,k}({\bm x}_{k}, {\bm y}_{k}))  \\
&{} := {\bm P}({\bm x}_{k}, {\bm y}_{k}) - c \mathcal{Q}_{k}({\bm x}_{k}, {\bm y}_{k})
\end{aligned}
\end{equation*}
where $\mathcal{Q}_{k}$ is almost everywhere $\ca{C}^1$ over $\mathbb{R}^{2nd}$. 

The Jacobian of  mapping ${\bm P}_1({\bm x}, {\bm y}) - c \hat{\eta}_k u_{1, k}({\bm x}, {\bm y})$ is 
\begin{equation*}\label{Jacobian1}
{\bm J}_{1c}({\bm x}, {\bm y})= \begin{pmatrix}
{\bm W}\otimes {\bm I}_{n}& 0\\
0& {\bm I}_{nd} \\
 \end{pmatrix} - c \hat{\eta}_{k} \begin{pmatrix} 0 & \nabla^2\phi({\bm y}) \\ 0 & 0 \\ 
 \end{pmatrix},  
\end{equation*}
where $\nabla^2 \phi({\bm y})$ is well-defined over a full-measure set. And the Jacobian of the mapping ${\bm P}_2({\bm x}, {\bm y}) - c \hat{\eta}_k u_{2, k}({\bm x}, {\bm y})$ is 
\begin{equation*}\label{Jacobian2}
{\bm J}_{2c}({\bm x}, {\bm y}) = \begin{pmatrix}
{\bm I}_{nd}& 0\\
0&  {\bm W}\otimes {\bm I}_{n} \\
 \end{pmatrix} - c \hat{\eta}_{k} \begin{pmatrix} 0 & 0 \\ \tau \Diag(\nabla^2 F_{1, \mathcal{B}_{1,k+1}}({\bm x}_{(1)}), \ldots, \nabla^2 F_{d, \mathcal{B}_{d,k+1}}({\bm x}_{(d)})) & - \tau   {\bm W}\otimes {\bm I}_{n} \\ 
 \end{pmatrix},  
\end{equation*}
Henceforth, the Jacobian of  ${\bm P}({\bm x}, {\bm y}) - c \mathcal{Q}_{k}({\bm x}, {\bm y})$ is ${\bm J}_{2c}({\bm x}, {\bm y}){\bm J}_{1c}({\bm x}, {\bm y})$.

When ${\bm W}$ is invertible, then the determinant of ${\bm J}_{2}({\bm x}, {\bm y}){\bm J}_1({\bm x}, {\bm y})$ is a non-trivial polynomial of $c$ for almost everywhere  $({\bm x}, {\bm y})\in \mathbb{R}^{2nd}$. Leveraging similar proof techniques in Section \ref{sec:app1}, we can derive that there exists a full-measure subset $\hat{\ca{S}} \subseteq \bb{R}_+$, such that for any $c \in \hat{\ca{S}}$, there exists a full-measure subset $\hat{\ca{K}}$ of $\mathbb{R}^{nd} \times \mathbb{R}^{nd}$ with the property that, whenever we choose $({\bm x}_0, {\bm y}_0) \in \hat{\ca{K}}$, it satisfies that $\{({\bm x}_k, {\bm y}_k)\} \cap \ca{Y} = \varnothing$. Consequently, the update scheme of vectorized sequence $({\bm x}_k, {\bm y}_k)$ can be reformulated as
\begin{equation*}
	\left\{
	\begin{aligned}
            {\bm x}_{k+1} & \in{} ({\bm W}\otimes {\bm I}_{n}){\bm x}_{k} - c \hat{\eta}_k \nabla \phi({\bm y}_{k}), \\
            {\bm g}_{k+1} ={}& \left(\nabla F_{1,\mathcal{B}_{1,k+1}}({\bm x}_{1,k+1}), \ldots, \nabla F_{d,\mathcal{B}_{d,k+1}}({\bm x}_{d, k+1})\right)^{\top} , \\
			{\bm y}_{k+1} & \in{} (1-c \tau\hat{\eta}_k) ({\bm W}\otimes {\bm I}_{n}){\bm y}_{k} +c \tau\hat{\eta}_k {\bm g}_{k+1}.
		\end{aligned}
		\right.
	\end{equation*}
	Therefore, by applying Theorem \ref{thm:DSGD-M} with $\D_f = \partial f$, we can choose $\ca{S} = \hat{\ca{S}} \cap (0,\alpha_c), j=1,2$ and $\ca{K} = \hat{\ca{K}}\cap (\X_0 \times \mathbb{R}^{nd})$, such that the assertion in Theorem \ref{thm:DSGDM2} holds.

\section{Synthetic numerical examples on  avoiding spurious critical points}\label{appendix:2}
In this part, we present numerical experiments on a synthetic example to illustrate that, under random initialization, Algorithm~\ref{alg:DSGD-M} converges to a Clarke-critical point with probability~$1$. 

We consider a four-node ring network with a non-singular mixing matrix
\[
W=\frac{1}{3}
\begin{pmatrix}
1 & 1 & 0 & 1\\
1 & 1 & 1 & 0\\
0 & 1 & 1 & 1\\
1 & 0 & 1 & 1
\end{pmatrix}.
\]
Moreover, the local objective functions are defined as
\begin{equation*}
\begin{aligned}
f_1(x,y) &= \mathrm{ReLU}(x)+\mathrm{ReLU}(y),\\
f_2(x,y) &= \mathrm{ReLU}(x)+\mathrm{ReLU}(y),\\
f_3(x,y) &= \mathrm{ReLU}(-x)+\mathrm{ReLU}(-y),\\
f_4(x,y) &= \mathrm{ReLU}(-x-1) + (x+1) - \mathrm{ReLU}(x+1) 
          + \mathrm{ReLU}(-x-2) + (x+2) - \mathrm{ReLU}(x+2) \\
        &\quad + \mathrm{ReLU}(-y-1) + (y+1) - \mathrm{ReLU}(y+1),
\end{aligned}
\end{equation*}
respectively.

Note that $f_4(x,y)$ is identically zero. However, by following the same theoretical analysis in \cite{NEURIPS2022_a9077da4}, automatic differentiation (AD) yields
$\partial_x f_4(-1,y)=1$, $\partial_x f_4(-2,y)=1$, and $\partial_y f_4(x,-1)=1$ for any $x \in \bb{R}$. 
As a result, AD introduces spurious stationary points for 
$f=f_1+f_2+f_3+f_4$, which are 
\[
(-1,0),\ (0,-1),\ (-2,0),\ (-1,-1),\ (-2,-1).
\]

In our experiments, we randomly initialize $(x_0,y_0)\sim \mathcal{N}\!\bigl((0,0), I_2\bigr)$ in Algorithm~\ref{alg:DSGD-M}, and in each epoch we compute only one full-batch local gradient using AD. 
Figure~\ref{fig:countour} reports a randomly sampled iterate trajectory and the percentages over $100$ independent runs that converge to different critical points. 
A run is declared to converge to a critical point if the squared distance from its final iterate to that point is below $10^{-4}$. 
As shown in Figure~\ref{fig:countour}, almost all runs generated by Algorithm~\ref{alg:DSGD-M} converge to the true Clarke-critical point $(0,0)$, which is consistent with Theorem~\ref{thm:DSGDM2}.

\begin{figure*}[!t]
\centering

\subfloat[Iterate trajectory]{\includegraphics[width=2.8in]{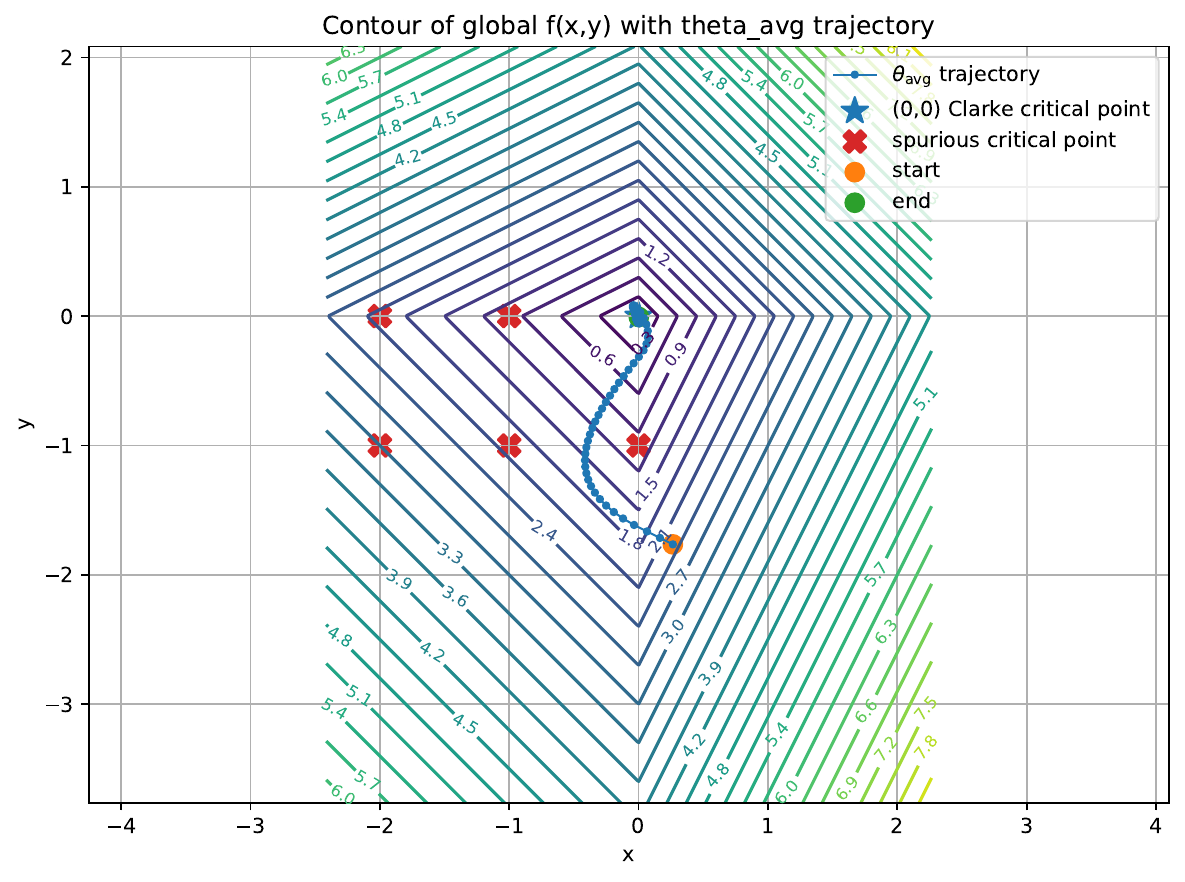}}
\hfil
\subfloat[The percentages on different types of the final results]{\includegraphics[width=3.1in]{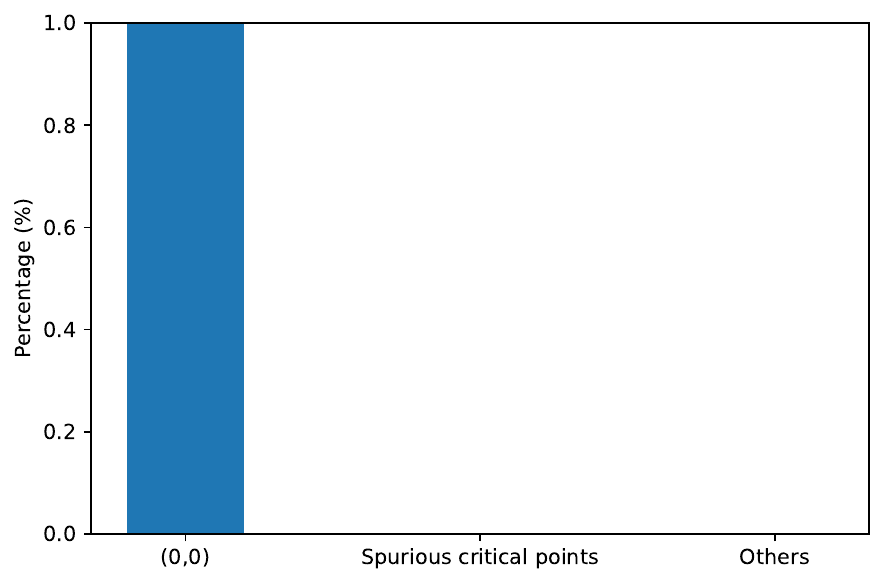}}
\caption{Numerical results of Algorithm \ref{alg:DSGD-M} on a toy example.}
\label{fig:countour}
\end{figure*}

\section{Developing Decentralized Generalized ADAM Methods with Convergence Guarantees}\label{appendix:adam}

Adam \cite{adam} uses a momentum term and an exponential moving average of squared gradients as an adaptive scaling factor, thereby applying coordinate-wise adaptive step sizes and updating parameters with bias correction. After this work, a number of efficient variants are developed, such as AdaBelief \cite{zhuang2020adabelief}, AMSGrad \cite{reddi2019convergence}, decentralized AMSGrad \cite{chen2023convergence}, etc. In the single-agent setting, \cite{xiao2023adam, ding2023adam} investigate the global convergence properties of Adam-family methods for nonsmooth objectives via a delicately constructed Lyapunov function.

In this part, we present a decentralized generalized Adam method in Algorithm \ref{alg:DAdam}, which fits into \eqref{Eq_Framework}. Notice  $b_{i, 1}, \ldots, b_{i, s}$ in Line 9 represent the block sizes in Adam-mini at node $i$.

Let ${\bm M}_{k}:= [{\bm m}_{1,k}, \ldots, {\bm m}_{d,k}]$, ${\bm V}_{k}:= [{\bm v}_{1,k}, \ldots, {\bm v}_{d,k}]$. We have the following theorem illustrating the convergence properties of Algorithm \ref{alg:DAdam}.

\begin{theo}\label{thm:DGAdam}
Suppose Assumption \ref{Assumption_obj} holds. $\{{\bm X}_k\}$ is generated by Algorithm \ref{alg:DAdam}. 
    \begin{enumerate}[label={(\arabic*)}]
        \item When Assumption \ref{Assumption_app1} holds, there exists $\alpha_{c}>0$, for any $c\in (0, \alpha_{c})$, it holds that almost surely, any cluster point of $\{{\bm X}_{k}\}$ is a $\D_f$-critical point of \eqref{Prob_DOP} and $\{f({\bm x}_{i,k}): k \in \bb{N}\}$ converges for any $i \in [d]$.
        \item When Assumption \ref{Assumption_app2} holds, for any $\varepsilon>0$, there exists $\alpha_{c}>0$, for any $c\in (0, \alpha_{c})$, 
        \begin{equation*}
            \mathbb{P}\left(\lim_{k \to \infty}\mathrm{dist}({\bm X}_{k}, \{{\bm X}\in \mathbb{R}^{n\times d}: {\bm X}={\bm x}{\bm 1}^{\top}, {\bm 0} \in D_f({\bm x})\}) =0  \right)\geq 1-\varepsilon
        \end{equation*}
and 
\begin{equation*}
\mathbb{P}(f({\bm x}_{i,k}) \textit { converges })\geq 1-\varepsilon, \textit{ }\forall i \in [d].
\end{equation*}
    \end{enumerate}
\end{theo}

\begin{algorithm}
\caption{DGAdam for solving \eqref{Prob_DOP}.}
\label{alg:DAdam}
\begin{algorithmic}[1] % 数字1确保行号是数字
\Require Initial point ${\bm x}_{0}\in \mathbb{R}^n$, initial momentum terms ${\bm m}_{0}, {\bm v}_0 \in \mathbb{R}^n$, momentum parameters $\tau_1, \tau_2 > 0$, an auxiliary function $\phi:\Rn \to \bb{R}$,  and a mixing matrix ${\bm W}$. 
\For{all $i\in [d]$ in parallel}
\State Set $k \gets 0$. Initialize ${\bm x}_{i,k}={\bm x}_0, {\bm m}_{i,k}={\bm m}_0, $; 
\While{not terminated}
\State Choose the step-size $\eta_{k}$;
\State Communicate and update the local variable 
    $${\bm x}_{i, k+1} \in \sum_{j\in \mathcal{N}_{i}}{\bm W}(i,j) {\bm x}_{j, k} - \eta_{k} \frac{1}{\sqrt{{\bm v}_{i, k} + \epsilon}} \odot ({\bm m}_{i,k} + \rho {\bm d}_{i,k});$$    
    \State Compute ${\bm d}_{i,k+1} \in \frac{1}{|\mathcal{B}_{i,k+1}|} \sum_{{\bm s}_l \in \mathcal{B}_{i,k+1}}  D_{F_{i,l}}({\bm x}_{i,k+1})$;
\State Randomly select a mini-batch $\mathcal{B}_{i,k} \subseteq \mathcal{S}_{i}$;
    \State Communicate and update momentum 
    $$
    {\bm m}_{i, k+1}  = (1-\tau_1 \eta_{k}) \sum_{j\in \mathcal{N}_{i}} {\bm W}(i,j) {\bm m}_{j,k}     + \tau_1 \eta_{k} {\bm d}_{i,k+1};
    $$
    \State Communicate and update second-moment estimate by one of the formats
    $$
    \begin{aligned}
    {\bm v}_{i, k+1}  & \in (1-\tau_2 \eta_{k})  \sum_{j\in \mathcal{N}_{i}} {\bm W}(i,j) {\bm v}_{j,k}     + \tau_2 \eta_{k} {\bm d}_{i, k+1} \odot {\bm d}_{i, k+1};  \hspace{9em} \triangleleft \text{ Adam }\\
    {\bm v}_{i, k+1}  & \in (1-\tau_2 \eta_{k})  \sum_{j\in \mathcal{N}_{i}} {\bm W}(i,j) {\bm v}_{j,k}     + \tau_2 \eta_{k} \Diag(\frac{{\bm 1}{\bm 1}^{\top}}{b_{i, 1}}, \ldots, \frac{{\bm 1}{\bm 1}^{\top}}{b_{i, s}}) {\bm d}_{i, k+1} \odot {\bm d}_{i, k+1}; \triangleleft \text{ Adam-mini }\\
    \end{aligned}
    $$
    \State Set $k \gets k + 1$;
\EndWhile
\EndFor
\State \Return ${\bm X}_k:= [{\bm x}_{1,k}, \ldots, {\bm x}_{d,k}]$;
\end{algorithmic}
\end{algorithm}

\begin{proof}
The compact format of Algorithm \ref{alg:DAdam} with Adam update can be written as
 \begin{equation}
        \label{Eq_Thm_DGAdam0}
        \begin{aligned}
            {\bm X}_{k+1} \in{}&  {\bm X}_k {\bm W} - \eta_k \frac{1}{\sqrt{{\bm V}_{k} + \epsilon {\bm I}}} \odot ({\bm M}_{k} + \rho {\bm D}_{k}), \\ 
            {\bm D}_{k+1} ={}& [{\bm d}_{1, k+1}, \ldots, {\bm d}_{d, k+1}],\\
            {\bm M}_{k+1} ={}& (1-\tau_1 \eta_{k}) {\bm M}_{k}{\bm W}+ \tau_1 \eta_{k} {\bm D}_{k+1},\\
            {\bm V}_{k+1} \in{}& (1-\tau_2\eta_{k}) {\bm V}_{k}{\bm W}+ \tau_2 \eta_{k} {\bm D}_{k+1}\odot {\bm D}_{k+1},\\
        \end{aligned}
    \end{equation}
However, if we directly define ${\bm Z}_k := \left[\begin{smallmatrix}
        {\bm X}_{k}\\
        {\bm M}_{k}\\ 
        {\bm V}_k
    \end{smallmatrix} \right]$, we cannot find a function $\Phi_i$ that satisfies Assumption \ref{Assumption_framework}.
    To address this issue, we introduce an auxiliary update scheme to  Algorithm \ref{alg:DAdam} parameterized by $K>0$. 
    \begin{equation}
        \label{Eq_Thm_DGAdam}
        \begin{aligned}
            {\bm X}_{k+1} \in{}&  {\bm X}_k {\bm W} - \eta_k \frac{1}{\sqrt{{\bm V}_{k} + \epsilon {\bm I}}} \odot ({\bm M}_{k} + \rho {\bm D}_{k}), \\ 
            {\bm D}_{k+1} ={}& [{\bm d}_{1, k+1}, \ldots, {\bm d}_{d, k+1}],\\
            {\bm M}_{k+1} ={}& (1-\tau_1 \eta_{k}) {\bm M}_{k}{\bm W}+ \tau_1 \eta_{k} {\bm D}_{k+1},\\
            \frac{1}{d}{\bm V}_{k+1}{\bm 1}_d {\bm 1}^{\top}_d \in{}& (1-\tau_2\eta_{k}) \frac{1}{d}{\bm V}_{k+1}{\bm 1}_d {\bm 1}^{\top}_d{\bm W}+ \tau_2 \eta_{k} \frac{1}{d}{\bm D}_{k+1}\odot {\bm D}_{k+1}{\bm 1}_d {\bm 1}^{\top}_d \mathbb{1}_{K}({\bm V}_{k}),\\
        \end{aligned}
    \end{equation}
    where $\mathbb{1}_{K}$ denotes the indication function of the subset $\{{\bm V}: \|{\bm V}\| \leq K \}$. 
    Denote  the filtration $\ca{F}_k := \sigma(\{{\bm X}_{l}: l\leq k\})$, $ D_{F_{\mathcal{B}_{i,k+1}}}({\bm x}_{i,k+1}):= \frac{1}{|\mathcal{B}_{i,k+1}|} \sum_{{\bm s}_l \in \mathcal{B}_{i,k+1}}  D_{F_{i,l}}({\bm x}_{i,k+1}) $, and let
    \begin{equation*}
    \mathcal{V}_i({\bm x}) := \begin{cases}\frac{1}{N}\sum_{k=0}^{N-1} D_{F_{\mathcal{B}_{i,k+1}}}({\bm x})\odot D_{F_{\mathcal{B}_{i,k+1}}}({\bm x}) &  \text{ for case (1), } \\ \mathbb{E}[D_{F_{\mathcal{B}_{i,k+1}}}({\bm x})\odot D_{F_{\mathcal{B}_{i,k+1}}}({\bm x})| \mathcal{F}_{k}] & \text{ for case (2), }\\
    \end{cases}
    \end{equation*}
    and
\begin{equation*}
{\bm Z}_k := \left[\begin{smallmatrix}
        {\bm X}_{k}\\
        {\bm M}_{k}\\ 
        \frac{1}{d}{\bm V}_{k}{\bm 1}_d {\bm 1}^{\top}_d
    \end{smallmatrix} \right], {\bm H}_k := \bb{E}\left[\left[\begin{smallmatrix}
       \frac{1}{\sqrt{{\bm V}_{k} + \epsilon {\bm I}}} \odot ({\bm M}_{k} + \rho {\bm D}_{k})\\
        \tau_1 {\bm M}_{k}{\bm W}- \tau_1 {\bm D}_{k+1}\\
         \tau_2  \frac{1}{d}{\bm V}_{k}{\bm 1}_d {\bm 1}^{\top}_d -\tau_2 \frac{1}{d}{\bm D}_{k+1}\odot {\bm D}_{k+1}{\bm 1}_d {\bm 1}^{\top}_d       
         \mathbb{1}_{K}({\bm V}_{k})
    \end{smallmatrix} \right] |\ca{F}_k\right], 
\end{equation*}    
\begin{equation*}   
\Xi_{k+1} := \left[\begin{smallmatrix}
        \frac{1}{\sqrt{{\bm V}_{k} + \epsilon {\bm I}}} \odot ({\bm M}_{k} + \rho {\bm D}_{k})\\
        \tau_1 {\bm M}_{k}{\bm W}- \tau_1 {\bm D}_{k+1}\\
         \tau_2  \frac{1}{d}{\bm V}_{k}{\bm 1}_d {\bm 1}^{\top}_d -\tau_2  \frac{1}{d}{\bm D}_{k+1}\odot {\bm D}_{k+1}{\bm 1}_d {\bm 1}^{\top}_d       
         \mathbb{1}_{K}({\bm V}_{k})
    \end{smallmatrix} \right] - {\bm H}_k.
\end{equation*}
Then \eqref{Eq_Thm_DGAdam} can be reshaped as ${\bm Z}_{k+1} \in {\bm Z}_k {\bm W} - \eta_k ({\bm H}_k + \Xi_{k+1})$, which aligns with the update format of \eqref{Eq_Framework}.

From the local boundedness of $\D_{f_i}$, we can conclude that $\{{\bm D}_k\}$ is uniformly bounded almost surely whenever $\{{\bm Z}_{k}\}$ is bounded, which further leads to the local boundedness of ${\bm H}_{k}$ and $\Xi_{k+1}$, i.e. Assumption \ref{Assumption_framework}(2) holds. Moreover, we set 
\begin{equation*}
\Phi_i^{(K)}({\bm x}, {\bm m}, {\bm v}) := \left[\begin{smallmatrix}
         (|{\bm v}| +\epsilon)^{-1/2}\odot ({\bm m} +\rho D_{f_i}({\bm x}))  \\
        \tau_1 {\bm m} - \tau_1 \D_{f_i}({\bm x}) \\
        \tau_2 {\bm v} -\tau_2 \mathcal{V}_i({\bm x})\mathbb{1}_{K}({\bm v})
\end{smallmatrix} \right].
\end{equation*}
\eqref{relation_H_PHI} holds vacuously for case (1) with a sufficiently large nonnegative $\delta_k$. Notice that $\mathbb{E}[\frac{1}{d}{\bm D}_{k+1}\odot {\bm D}_{k+1}{\bm 1}_d| \mathcal{F}_k] \in \frac{1}{d}\sum_{i=1}^d \mathcal{V}_{i}^{\|{\bm x}_{i,k}-{\bm x}_{i,k+1}\|}({\bm x}_{i,k}).$ We set $\delta_k$ to be an exponential moving average 
\begin{equation}\label{eq:delta_kk}
\begin{aligned}
\delta_k & = \eta_{k}\frac{\|{\bm M}_{k}+\rho {\bm D}_k\|}{\sqrt{\epsilon}} +2\sum_{j=0}^{k-1} \lambda_2^{k-1-j}\eta_{j} \frac{\|{\bm M}_{j}+\rho {\bm D}_j\|}{\sqrt{\epsilon}}\\
& + \tau_2 \eta_{k-1}\|{\bm D}_{k} \odot {\bm D}_k\| + \sum_{j=0}^{k-1}(\Pi_{i=0}^{j} (1-\tau_2 \eta_{k-1-i})) \lambda_2^{j+1} \tau_2\eta_{k-j-2} \|{\bm D}_{k-1-j} \odot {\bm D}_{k-j-1}\|,
\end{aligned}
\end{equation}
for case (2).  Then, we attain that 
\begin{equation*}
\Phi^{(K)}({\bm x}, {\bm m}, {\bm v}) = \left[\begin{smallmatrix}
         (|v| +\epsilon)^{-1/2}\odot (m +\rho D_{f}({\bm x}))  \\
        \tau_1 {\bm m} - \tau_1 D_{f}({\bm x}) \\
        \tau_2 {\bm v} -\tau_2 \conv\left(\frac{1}{d}\sum_{i=1}^d\mathcal{V}_i({\bm x})\right)\mathbb{1}_{K}({\bm v})
\end{smallmatrix} \right].
\end{equation*}
Since set-valued mapping $\conv\left(\frac{1}{d}\sum_{i=1}^d\mathcal{V}_i({\bm x})\right)$ is a nonnegative, convex compact valued, and graph-closed, by Proposition 5.3 in \cite{xiao2023convergence}, we have $\psi^{(K)}({\bm x}, {\bm m}, {\bm v}) = f({\bm x}) +\frac{1}{2\tau_1} \langle {\bm m}, (|{\bm v}| +\epsilon )^{-1/2}\odot{\bm m} \rangle + K\max\{0, \|{\bm v}\|-K\}$ is a coercive Lyapunov function of differential inclusion $(\frac{d {\bm x}}{d t}, \frac{d{\bm m}}{d t}, \frac{d{\bm v}}{d t}) \in - \Phi^{(K)}({\bm x}, {\bm m}, {\bm v})$, with stable  set $\mathcal{A}_{K}=\{({\bm x}, {\bm m}, {\bm v})\in \mathbb{R}^{n} \times \mathbb{R}^{n} \times \mathbb{R}^{n}: {\bm 0}\in D_{f}({\bm x}), {\bm m}=0, \|{\bm v}\|\leq K\}$, which
verifies Assumption \ref{Assumption_framework}-(3). Moreover,  Assumption \ref{Assumption_framework}-(4) directly follows from the definability of $f$ and $\mathcal{A}_{K}= \{ ({\bm x}, {\bm m}, {\bm v}): {\bm 0} \in D_{\psi^{(K)}}({\bm x}, {\bm m}, {\bm v})\}$. 

For case (1), we set $\mathcal{U}_{i,j}({\bm x}, {\bm m}, {\bm v}):= \left[\begin{smallmatrix}
         (|{\bm v}| +\epsilon)^{-1/2}\odot ({\bm m} +\rho  D_{F_{\mathcal{B}_{i,j}}}({\bm x}))  \\
        \tau_1 {\bm m} - \tau_1 D_{F_{\mathcal{B}_{i,j}}}({\bm x})) \\
        \tau_2 {\bm v} -\tau_2  D_{F_{\mathcal{B}_{i,j+1}}}({\bm x}) \odot D_{F_{\mathcal{B}_{i,j+1}}}({\bm x})
 \mathbb{1}_{K}({\bm v})
\end{smallmatrix} \right]$, where $\mathcal{B}_{i,j}$, $0\leq j \leq N-1$ are subsets of equal size. We also set $\rho_k$ similar to $\delta_k$ in \eqref{eq:delta_kk}. Then \eqref{Eq_Thm_DGAdam} is a special form of \eqref{Eq_Framework} with \eqref{eq:reshuf}. Assumption \ref{Assumption_app1} togather with the fact $\frac{1}{N}\sum_{i=0}^{N-1} \mathcal{U}_{i,j}({\bm x}, {\bm m}, {\bm v}) = \Phi_{i}^{(K)}({\bm x}, {\bm m}, {\bm v})$ implies that Assumption \ref{asp:reshuffling} is satisfied for case (1). For case (2), we choose $\xi_{i, k}:=\{\textit{Draw out } \mathcal{B}_{i,k} \textit{ from agent } i\}$, and 
\begin{equation*}
\chi_{i}({\bm x}, {\bm m}, {\bm v}, \xi_{i, k})\in \left[\begin{smallmatrix}
         (|v| +\epsilon)^{-1/2}\odot (\rho D_{F_{\mathcal{B}_{i,k}}}({\bm x})- \rho D_{f_i}({\bm x})  \\
      \tau_1 D_{f_i}({\bm x})  - \tau_1 D_{F_{\mathcal{B}_{i,k}}}({\bm x}) \\
       \tau_2 \frac{1}{d}\sum_{i=1}^d\mathcal{V}_i({\bm x})\mathbb{1}_{K}({\bm v})  -\tau_2 \frac{1}{d}\sum_{i=1}^d D_{F_{\mathcal{B}_{i,k}}}({\bm x}) \odot D_{F_{\mathcal{B}_{i,k}}}({\bm x})
 \mathbb{1}_{K}({\bm v})
\end{smallmatrix} \right],
\end{equation*}
\eqref{Eq_Thm_DGAdam} is a particular form of \eqref{Eq_Framework} with \eqref{eq:chi}, and Assumption \ref{Assumption_stochastic} is satisfied for case (2). Plugging Lemma \ref{lem:2} and Theorem \ref{theo:withreplace_thm1}, we can prove that $\delta_{k}, \rho_k \to 0$, as $k \to \infty$ with arbitrary high probability. 

Applying Theorem \ref{thm:reshuffing} and \ref{theo:withreplace_thm1}, we derive that \eqref{Eq_Thm_DGAdam} is restricted in some bounded set for case (1) and high-probability bounded for case (2) for any $K>0$, which implies that Algorithm \ref{alg:DAdam} coincides with \eqref{Eq_Thm_DGAdam} for some sufficiently large $K$. Employing Theorem \ref{theo:reshuff} and \ref{theo:with}, we get the desired results in Theorem \ref{thm:DGAdam}. 

For decentralized Adam-mini, we denote $\Sigma_i:= \Diag(\frac{{\bm 1}{\bm 1}^{\top}}{b_{i, 1}}, \ldots, \frac{{\bm 1}{\bm 1}^{\top}}{b_{i, s}})$, and replace the update schemes of $\frac{1}{d}{\bm V}_{k+1}{\bm 1}_d {\bm 1}^{\top}_d$ and $\mathcal{V}_i(x)$ by 
\begin{equation*}
 \frac{1}{d}{\bm V}_{k+1}{\bm 1}_d {\bm 1}^{\top}_d \in{} (1-\tau_2\eta_{k}) \frac{1}{d}{\bm V}_{k+1}{\bm 1}_d {\bm 1}^{\top}_d{\bm W}+ \tau_2 \eta_{k} \frac{1}{d}[\Sigma_1 {\bm d}_{1, k+1}\odot {\bm d}_{1, k+1}, \ldots, \Sigma_d {\bm d}_{d, k+1}\odot {\bm d}_{d, k+1}]{\bm 1}_d {\bm 1}^{\top}_d \mathbb{1}_{K}({\bm V}_{k}) 
\end{equation*}
and
 \begin{equation*}
    \mathcal{V}_i({\bm x}) := \begin{cases}\frac{1}{N}\sum_{k=0}^{N-1} \Sigma_i D_{F_{\mathcal{B}_{i,k+1}}}({\bm x})\odot D_{F_{\mathcal{B}_{i,k+1}}}({\bm x}) &  \text{ for case (1), } \\ \mathbb{E}[\Sigma_i D_{F_{\mathcal{B}_{i,k+1}}}({\bm x})\odot D_{F_{\mathcal{B}_{i,k+1}}}({\bm x})| \mathcal{F}_{k}] & \text{ for case (2). }\\
    \end{cases}
    \end{equation*}
Correspondingly, we add $\Sigma_i$ before the term $D_{F_{\mathcal{B}_{i,k}}}({\bm x}) \odot D_{F_{\mathcal{B}_{i,k}}}({\bm x})$ in the  expression of $\mathcal{U}_{i,j}$ and $\chi_i$. Then, using the same approach as in the above proof of decentralized Adam, we can obtain the desired results.
  
\end{proof}

\bibliographystyle{IEEEtran}
\bibliography{ref}

@inproceedings{ojasiewicz1965EnsemblesS,
  title={Ensembles semi-analytiques},
  author={Stanisław Łojasiewicz},
  year={1965},
  url={https://api.semanticscholar.org/CorpusID:118074635}
}

@article{polyak1964some,
  title={Some methods of speeding up the convergence of iteration methods},
  author={Polyak, Boris T},
  journal={Ussr computational mathematics and mathematical physics},
  volume={4},
  number={5},
  pages={1--17},
  year={1964},
  publisher={Elsevier}
}

@article{chen2023symbolic,
	title={Symbolic discovery of optimization algorithms},
	author={Chen, Xiangning and Liang, Chen and Huang, Da and Real, Esteban and Wang, Kaiyuan and Liu, Yao and Pham, Hieu and Dong, Xuanyi and Luong, Thang and Hsieh, Cho-Jui and others},
	journal={arXiv preprint arXiv:2302.06675},
	year={2023}
}

@article{castera2021inertial,
  title={An inertial newton algorithm for deep learning},
  author={Castera, Camille and Bolte, J{\'e}r{\^o}me and F{\'e}votte, C{\'e}dric and Pauwels, Edouard},
  journal={The Journal of Machine Learning Research},
  volume={22},
  number={1},
  pages={5977--6007},
  year={2021},
  publisher={JMLRORG}
}

@inproceedings{lu2019gnsd,
  title={GNSD: A gradient-tracking based nonconvex stochastic algorithm for decentralized optimization},
  author={Lu, Songtao and Zhang, Xinwei and Sun, Haoran and Hong, Mingyi},
  booktitle={2019 IEEE Data Science Workshop (DSW)},
  pages={315--321},
  year={2019},
  organization={IEEE}
}

@article{bolte2021conservative,
	title={Conservative set valued fields, automatic differentiation, stochastic gradient methods and deep learning},
	author={Bolte, J{\'e}r{\^o}me and Pauwels, Edouard},
	journal={Mathematical Programming},
	volume={188},
	number={1},
	pages={19--51},
	year={2021},
	publisher={Springer}
}

@article{le2024nonsmooth,
  title={Nonsmooth nonconvex stochastic heavy ball},
  author={Le, Tam},
  journal={Journal of Optimization Theory and Applications},
  pages={1--21},
  year={2024},
  publisher={Springer}
}

@inproceedings{bernstein2018signsgd,
  title={SignSGD: Compressed optimisation for non-convex problems},
  author={Bernstein, Jeremy and Wang, Yu-Xiang and Azizzadenesheli, Kamyar and Anandkumar, Animashree},
  booktitle={International Conference on Machine Learning},
  pages={560--569},
  year={2018},
  organization={PMLR}
}

@book{clarke1990optimization,
	title={Optimization and nonsmooth analysis},
	author={Clarke, Frank H},
	volume={5},
	year={1990},
	publisher={SIAM}
}

@article{benaim2005stochastic,
	title={Stochastic approximations and differential inclusions},
	author={Bena{\"\i}m, Michel and Hofbauer, Josef and Sorin, Sylvain},
	journal={SIAM Journal on Control and Optimization},
	volume={44},
	number={1},
	pages={328--348},
	year={2005},
	publisher={SIAM}
}

@article{davis2020stochastic,
	title={Stochastic subgradient method converges on tame functions},
	author={Davis, Damek and Drusvyatskiy, Dmitriy and Kakade, Sham and Lee, Jason D},
	journal={Foundations of Computational Mathematics},
	volume={20},
	number={1},
	pages={119--154},
	year={2020},
	publisher={Springer}
}

@article{bolte2007clarke,
	title={Clarke subgradients of stratifiable functions},
	author={Bolte, J{\'e}r{\^o}me and Daniilidis, Aris and Lewis, Adrian and Shiota, Masahiro},
	journal={SIAM Journal on Optimization},
	volume={18},
	number={2},
	pages={556--572},
	year={2007},
	publisher={SIAM}
}

@article{bianchi2022convergence,
	title={Convergence of constant step stochastic gradient descent for non-smooth non-convex functions},
	author={Bianchi, Pascal and Hachem, Walid and Schechtman, Sholom},
	journal={Set-Valued and Variational Analysis},
	pages={1--31},
	year={2022},
	publisher={Springer}
}

@article{krizhevsky2009learning,
	title={Learning multiple layers of features from tiny images},
	author={Krizhevsky, Alex and Hinton, Geoffrey and others},
	year={2009},
	publisher={Toronto, ON, Canada}
}

@inproceedings{he2016deep,
	title={Deep residual learning for image recognition},
	author={He, Kaiming and Zhang, Xiangyu and Ren, Shaoqing and Sun, Jian},
	booktitle={Proceedings of the IEEE conference on computer vision and pattern recognition},
	pages={770--778},
	year={2016}
}

@article{xiao2023adam,
  author  = {Nachuan Xiao and Xiaoyin Hu and Xin Liu and Kim-Chuan Toh},
  title   = {Adam-family Methods for Nonsmooth Optimization with Convergence Guarantees},
  journal = {Journal of Machine Learning Research},
  year    = {2024},
  volume  = {25},
  number  = {48},
  pages   = {1--53},
  url     = {http://jmlr.org/papers/v25/23-0576.html}
}

@article{bolte2020mathematical,
  title={A mathematical model for automatic differentiation in machine learning},
  author={Bolte, J{\'e}r{\^o}me and Pauwels, Edouard},
  journal={Advances in Neural Information Processing Systems},
  volume={33},
  pages={10809--10819},
  year={2020}
}

@article{bianchi2021closed,
  title={A closed-measure approach to stochastic approximation},
  author={Bianchi, Pascal and Rios-Zertuche, Rodolfo},
  journal={arXiv preprint arXiv:2112.05482},
  year={2021}
}

@article{bolte2022long,
  title={Long term dynamics of the subgradient method for Lipschitz path differentiable functions},
  author={Bolte, J{\'e}r{\^o}me and Pauwels, Edouard and Rios-Zertuche, Rodolfo},
  journal={Journal of the European Mathematical Society},
  year={2022}
}

@article{duchi2018stochastic,
  title={Stochastic methods for composite and weakly convex optimization problems},
  author={Duchi, John C and Ruan, Feng},
  journal={SIAM Journal on Optimization},
  volume={28},
  number={4},
  pages={3229--3259},
  year={2018},
  publisher={SIAM}
}

@book{borkar2009stochastic,
  title={Stochastic approximation: a dynamical systems viewpoint},
  author={Borkar, Vivek S},
  volume={48},
  year={2009},
  publisher={Springer}
}

@article{xiao2023convergence,
  title={Convergence Guarantees for Stochastic Subgradient Methods in Nonsmooth Nonconvex Optimization},
  author={Xiao, Nachuan and Hu, Xiaoyin and Toh, Kim-Chuan},
  journal={arXiv preprint arXiv:2307.10053},
  year={2023}
}

@article{ruszczynski2020convergence,
  title={Convergence of a stochastic subgradient method with averaging for nonsmooth nonconvex constrained optimization},
  author={Ruszczy{\'n}ski, Andrzej},
  journal={Optimization Letters},
  volume={14},
  number={7},
  pages={1615--1625},
  year={2020},
  publisher={Springer}
}

@article{sundhar2010distributed,
  title={Distributed stochastic subgradient projection algorithms for convex optimization},
  author={Sundhar Ram, S and Nedi{\'c}, Angelia and Veeravalli, Venugopal V},
  journal={Journal of optimization theory and applications},
  volume={147},
  pages={516--545},
  year={2010},
  publisher={Springer}
}

@article{pillai2005perron,
  title={The Perron-Frobenius theorem: some of its applications},
  author={Pillai, S Unnikrishna and Suel, Torsten and Cha, Seunghun},
  journal={IEEE Signal Processing Magazine},
  volume={22},
  number={2},
  pages={62--75},
  year={2005},
  publisher={IEEE}
}

@article{xiao2004fast,
  title={Fast linear iterations for distributed averaging},
  author={Xiao, Lin and Boyd, Stephen},
  journal={Systems \& Control Letters},
  volume={53},
  number={1},
  pages={65--78},
  year={2004},
  publisher={Elsevier}
}

@article{xiao2006distributed,
  title={Distributed average consensus with time-varying metropolis weights},
  author={Xiao, Lin and Boyd, Stephen and Lall, Sanjay},
  journal={Automatica},
  volume={1},
  year={2006}
}

@article{shi2015extra,
  title={Extra: An exact first-order algorithm for decentralized consensus optimization},
  author={Shi, Wei and Ling, Qing and Wu, Gang and Yin, Wotao},
  journal={SIAM Journal on Optimization},
  volume={25},
  number={2},
  pages={944--966},
  year={2015},
  publisher={SIAM}
}

@article{nedic2018network,
  title={Network topology and communication-computation tradeoffs in decentralized optimization},
  author={Nedi{\'c}, Angelia and Olshevsky, Alex and Rabbat, Michael G},
  journal={Proceedings of the IEEE},
  volume={106},
  number={5},
  pages={953--976},
  year={2018},
  publisher={IEEE}
}

@incollection{benaim2006dynamics,
  title={Dynamics of stochastic approximation algorithms},
  author={Bena{\"\i}m, Michel},
  booktitle={Seminaire de probabilites XXXIII},
  pages={1--68},
  year={2006},
  publisher={Springer}
}

@article{nedic2009distributed,
  title={Distributed subgradient methods for multi-agent optimization},
  author={Nedic, Angelia and Ozdaglar, Asuman},
  journal={IEEE Transactions on Automatic Control},
  volume={54},
  number={1},
  pages={48--61},
  year={2009},
  publisher={IEEE}
}

@article{zeng2018nonconvex,
  title={On nonconvex decentralized gradient descent},
  author={Zeng, Jinshan and Yin, Wotao},
  journal={IEEE Transactions on signal processing},
  volume={66},
  number={11},
  pages={2834--2848},
  year={2018},
  publisher={IEEE}
}

@article{lian2017can,
  title={Can decentralized algorithms outperform centralized algorithms? a case study for decentralized parallel stochastic gradient descent},
  author={Lian, Xiangru and Zhang, Ce and Zhang, Huan and Hsieh, Cho-Jui and Zhang, Wei and Liu, Ji},
  journal={Advances in neural information processing systems},
  volume={30},
  year={2017}
}

@article{swenson2022distributed,
  title={Distributed stochastic gradient descent: Nonconvexity, nonsmoothness, and convergence to local minima},
  author={Swenson, Brian and Murray, Ryan and Poor, H Vincent and Kar, Soummya},
  journal={The Journal of Machine Learning Research},
  volume={23},
  number={1},
  pages={14751--14812},
  year={2022},
  publisher={JMLRORG}
}

@article{pu2021distributed,
  title={Distributed stochastic gradient tracking methods},
  author={Pu, Shi and Nedi{\'c}, Angelia},
  journal={Mathematical Programming},
  volume={187},
  pages={409--457},
  year={2021},
  publisher={Springer}
}

@article{zhang2019decentralized,
  title={Decentralized stochastic gradient tracking for non-convex empirical risk minimization},
  author={Zhang, Jiaqi and You, Keyou},
  journal={arXiv preprint arXiv:1909.02712},
  year={2019}
}

@article{koloskova2021improved,
  title={An improved analysis of gradient tracking for decentralized machine learning},
  author={Koloskova, Anastasiia and Lin, Tao and Stich, Sebastian U},
  journal={Advances in Neural Information Processing Systems},
  volume={34},
  pages={11422--11435},
  year={2021}
}

@inproceedings{yu2019linear,
  title={On the linear speedup analysis of communication efficient momentum SGD for distributed non-convex optimization},
  author={Yu, Hao and Jin, Rong and Yang, Sen},
  booktitle={International Conference on Machine Learning},
  pages={7184--7193},
  year={2019},
  organization={PMLR}
}

@article{gao2020periodic,
  title={Periodic stochastic gradient descent with momentum for decentralized training},
  author={Gao, Hongchang and Huang, Heng},
  journal={arXiv preprint arXiv:2008.10435},
  year={2020}
}

@inproceedings{chen2023convergence,
  title={On the convergence of decentralized adaptive gradient methods},
  author={Chen, Xiangyi and Karimi, Belhal and Zhao, Weijie and Li, Ping},
  booktitle={Asian Conference on Machine Learning},
  pages={217--232},
  year={2023},
  organization={PMLR}
}

@article{di2016next,
  title={Next: In-network nonconvex optimization},
  author={Di Lorenzo, Paolo and Scutari, Gesualdo},
  journal={IEEE Transactions on Signal and Information Processing over Networks},
  volume={2},
  number={2},
  pages={120--136},
  year={2016},
  publisher={IEEE}
}

@article{xin2021improved,
  title={An improved convergence analysis for decentralized online stochastic non-convex optimization},
  author={Xin, Ran and Khan, Usman A and Kar, Soummya},
  journal={IEEE Transactions on Signal Processing},
  volume={69},
  pages={1842--1858},
  year={2021},
  publisher={IEEE}
}

@article{chen2021distributed,
  title={On distributed nonconvex optimization: Projected subgradient method for weakly convex problems in networks},
  author={Chen, Shixiang and Garcia, Alfredo and Shahrampour, Shahin},
  journal={IEEE Transactions on Automatic Control},
  volume={67},
  number={2},
  pages={662--675},
  year={2021},
  publisher={IEEE}
}

@article{boyd2004fastest,
  title={Fastest mixing Markov chain on a graph},
  author={Boyd, Stephen and Diaconis, Persi and Xiao, Lin},
  journal={SIAM review},
  volume={46},
  number={4},
  pages={667--689},
  year={2004},
  publisher={SIAM}
}

@article{pauwels2021incremental,
  title={Incremental Without Replacement Sampling in Nonconvex Optimization},
  author={Pauwels, Edouard},
  journal={Journal of Optimization Theory and Applications},
  volume={190},
  number={1},
  pages={274--299},
  year={2021},
  publisher={Springer}
}

@inproceedings{yuan2021decentlam,
  title={DecentLaM: Decentralized momentum SGD for large-batch deep training},
  author={Yuan, Kun and Chen, Yiming and Huang, Xinmeng and Zhang, Yingya and Pan, Pan and Xu, Yinghui and Yin, Wotao},
  booktitle={Proceedings of the IEEE/CVF International Conference on Computer Vision},
  pages={3029--3039},
  year={2021}
}

@article{shi2015proximal,
  title={A proximal gradient algorithm for decentralized composite optimization},
  author={Shi, Wei and Ling, Qing and Wu, Gang and Yin, Wotao},
  journal={IEEE Transactions on Signal Processing},
  volume={63},
  number={22},
  pages={6013--6023},
  year={2015},
  publisher={IEEE}
}

@article{wang2022decentralized,
  title={Decentralized optimization over the Stiefel manifold by an approximate augmented Lagrangian function},
  author={Wang, Lei and Liu, Xin},
  journal={IEEE Transactions on Signal Processing},
  volume={70},
  pages={3029--3041},
  year={2022},
  publisher={IEEE}
}

@inproceedings{di2019distributed,
  title={Distributed stochastic nonconvex optimization and learning based on successive convex approximation},
  author={Di Lorenzo, Paolo and Scardapane, Simone},
  booktitle={2019 53rd Asilomar Conference on Signals, Systems, and Computers},
  pages={1--5},
  year={2019},
  organization={IEEE}
}

@article{niu2021distributed,
  title={A distributed stochastic proximal-gradient algorithm for composite optimization},
  author={Niu, Youcheng and Li, Huaqing and Wang, Zheng and L{\"u}, Qingguo and Xia, Dawen and Ji, Lianghao},
  journal={IEEE Transactions on Control of Network Systems},
  volume={8},
  number={3},
  pages={1383--1393},
  year={2021},
  publisher={IEEE}
}

@article{xin2021stochastic,
  title={A stochastic proximal gradient framework for decentralized non-convex composite optimization: Topology-independent sample complexity and communication efficiency},
  author={Xin, Ran and Das, Subhro and Khan, Usman A and Kar, Soummya},
  journal={arXiv preprint arXiv:2110.01594},
  year={2021}
}

@inproceedings{xiao2023one,
  title={A one-sample decentralized proximal algorithm for non-convex stochastic composite optimization},
  author={Xiao, Tesi and Chen, Xuxing and Balasubramanian, Krishnakumar and Ghadimi, Saeed},
  booktitle={Uncertainty in Artificial Intelligence},
  pages={2324--2334},
  year={2023},
  organization={PMLR}
}

@article{jiang2017collaborative,
  title={Collaborative deep learning in fixed topology networks},
  author={Jiang, Zhanhong and Balu, Aditya and Hegde, Chinmay and Sarkar, Soumik},
  journal={Advances in Neural Information Processing Systems},
  volume={30},
  year={2017}
}

@inproceedings{george2019distributed,
  title={Distributed stochastic gradient method for non-convex problems with applications in supervised learning},
  author={George, Jemin and Yang, Tao and Bai, He and Gurram, Prudhvi},
  booktitle={2019 IEEE 58th Conference on Decision and Control (CDC)},
  pages={5538--5543},
  year={2019},
  organization={IEEE}
}

@article{sahinoglu2024online,
  title={Online Optimization Perspective on First-Order and Zero-Order Decentralized Nonsmooth Nonconvex Stochastic Optimization},
  author={Sahinoglu, Emre and Shahrampour, Shahin},
  journal={arXiv preprint arXiv:2406.01484},
  year={2024}
}

@article{bottou2018optimization,
  title={Optimization methods for large-scale machine learning},
  author={Bottou, L{\'e}on and Curtis, Frank E and Nocedal, Jorge},
  journal={SIAM review},
  volume={60},
  number={2},
  pages={223--311},
  year={2018},
  publisher={SIAM}
}

@book{shapiro2021lectures,
  title={Lectures on stochastic programming: modeling and theory},
  author={Shapiro, Alexander and Dentcheva, Darinka and Ruszczynski, Andrzej},
  year={2021},
  publisher={SIAM}
}

@article{jain2017non,
  title={Non-convex optimization for machine learning},
  author={Jain, Prateek and Kar, Purushottam and others},
  journal={Foundations and Trends{\textregistered} in Machine Learning},
  volume={10},
  number={3-4},
  pages={142--363},
  year={2017},
  publisher={Now Publishers, Inc.}
}

@article{mateos2012distributed,
  title={Distributed recursive least-squares: Stability and performance analysis},
  author={Mateos, Gonzalo and Giannakis, Georgios B},
  journal={IEEE Transactions on Signal Processing},
  volume={60},
  number={7},
  pages={3740--3754},
  year={2012},
  publisher={IEEE}
}

@article{cohen2016distributed,
  title={Distributed learning algorithms for spectrum sharing in spatial random access wireless networks},
  author={Cohen, Kobi and Nedi{\'c}, Angelia and Srikant, Rayadurgam},
  journal={IEEE Transactions on Automatic Control},
  volume={62},
  number={6},
  pages={2854--2869},
  year={2016},
  publisher={IEEE}
}

@article{bolognani2014distributed,
  title={Distributed reactive power feedback control for voltage regulation and loss minimization},
  author={Bolognani, Saverio and Carli, Ruggero and Cavraro, Guido and Zampieri, Sandro},
  journal={IEEE Transactions on Automatic Control},
  volume={60},
  number={4},
  pages={966--981},
  year={2014},
  publisher={IEEE}
}

@article{josz2024global,
  title={Global stability of first-order methods for coercive tame functions},
  author={Josz, C{\'e}dric and Lai, Lexiao},
  journal={Mathematical Programming},
  volume={207},
  number={1},
  pages={551--576},
  year={2024},
  publisher={Springer}
}

@article{josz2023lyapunov,
  title={Lyapunov stability of the subgradient method with constant step size},
  author={Josz, C{\'e}dric and Lai, Lexiao},
  journal={Mathematical Programming},
  volume={202},
  number={1},
  pages={387--396},
  year={2023},
  publisher={Springer}
}

@article{bianchi2019constant,
  title={Constant step stochastic approximations involving differential inclusions: stability, long-run convergence and applications},
  author={Bianchi, Pascal and Hachem, Walid and Salim, Adil},
  journal={Stochastics},
  volume={91},
  number={2},
  pages={288--320},
  year={2019},
  publisher={Taylor \& Francis}
}

@article{bolte2024inexact,
  title={Inexact subgradient methods for semialgebraic functions},
  author={Bolte, J{\'e}r{\^o}me and Le, Tam and Moulines, {\'E}ric and Pauwels, Edouard},
  journal={arXiv preprint arXiv:2404.19517},
  year={2024}
}

@article{bian2009subgradient,
  title={Subgradient-based neural networks for nonsmooth nonconvex optimization problems},
  author={Bian, Wei and Xue, Xiaoping},
  journal={IEEE Transactions on Neural Networks},
  volume={20},
  number={6},
  pages={1024--1038},
  year={2009},
  publisher={IEEE}
}

@book{borkar2008stochastic,
  title={Stochastic approximation: a dynamical systems viewpoint},
  author={Borkar, Vivek S and Borkar, Vivek S},
  volume={9},
  year={2008},
  publisher={Springer}
}

@article{gao2023distributed,
  title={Distributed stochastic gradient tracking methods with momentum acceleration for non-convex optimization},
  author={Gao, Juan and Liu, Xin-Wei and Dai, Yu-Hong and Huang, Yakui and Gu, Junhua},
  journal={Computational Optimization and Applications},
  volume={84},
  number={2},
  pages={531--572},
  year={2023},
  publisher={Springer}
}

@article{bolte2021nonsmooth,
  title={Nonsmooth implicit differentiation for machine-learning and optimization},
  author={Bolte, J{\'e}r{\^o}me and Le, Tam and Pauwels, Edouard and Silveti-Falls, Tony},
  journal={Advances in neural information processing systems},
  volume={34},
  pages={13537--13549},
  year={2021}
}

@article{de2020random,
  title={Random reshuffling is not always better},
  author={De Sa, Christopher M},
  journal={Advances in Neural Information Processing Systems},
  volume={33},
  pages={5957--5967},
  year={2020}
}

@book{kelley2017general,
  title={General topology},
  author={Kelley, John L},
  year={2017},
  publisher={Courier Dover Publications}
}

@article{huang2023distributed,
  title={Distributed random reshuffling over networks},
  author={Huang, Kun and Li, Xiao and Milzarek, Andre and Pu, Shi and Qiu, Junwen},
  journal={IEEE Transactions on Signal Processing},
  volume={71},
  pages={1143--1158},
  year={2023},
  publisher={IEEE}
}

@article{huang2023distributed2,
  title={Distributed random reshuffling methods with improved convergence},
  author={Huang, Kun and Zhou, Linli and Pu, Shi},
  journal={arXiv preprint arXiv:2306.12037},
  year={2023}
}

@article{adam,
author = {Kingma, Diederik and Ba, Jimmy},
year = {2014},
month = {12},
pages = {},
title = {Adam: A Method for Stochastic Optimization},
journal = {International Conference on Learning Representations}
}

@article{zhuang2020adabelief,
  title={Adabelief optimizer: Adapting stepsizes by the belief in observed gradients},
  author={Zhuang, Juntang and Tang, Tommy and Ding, Yifan and Tatikonda, Sekhar C and Dvornek, Nicha and Papademetris, Xenophon and Duncan, James},
  journal={Advances in neural information processing systems},
  volume={33},
  pages={18795--18806},
  year={2020}
}

@article{reddi2019convergence,
  title={On the convergence of adam and beyond},
  author={Reddi, Sashank J and Kale, Satyen and Kumar, Sanjiv},
  journal={arXiv preprint arXiv:1904.09237},
  year={2019}
}

@article{ding2023adam,
  title={Adam-family methods with decoupled weight decay in deep learning},
  author={Ding, Kuangyu and Xiao, Nachuan and Toh, Kim-Chuan},
  journal={arXiv preprint arXiv:2310.08858},
  year={2023}
}

@article{wang2025decentralized,
  title={A decentralized proximal gradient tracking algorithm for composite optimization on riemannian manifolds},
  author={Wang, Lei and Bao, Le and Liu, Xin},
  journal={Journal of Machine Learning Research},
  volume={26},
  number={106},
  pages={1--37},
  year={2025}
}

@article{condat2024locodl,
  title={Locodl: Communication-efficient distributed learning with local training and compression},
  author={Condat, Laurent and Maranjyan, Artavazd and Richt{\'a}rik, Peter},
  journal={arXiv preprint arXiv:2403.04348},
  year={2024}
}

@article{xu2023parallel,
  title={Parallel and distributed asynchronous adaptive stochastic gradient methods},
  author={Xu, Yangyang and Xu, Yibo and Yan, Yonggui and Sutcher-Shepard, Colin and Grinberg, Leopold and Chen, Jie},
  journal={Mathematical Programming Computation},
  volume={15},
  number={3},
  pages={471--508},
  year={2023},
  publisher={Springer}
}

@inproceedings{sun2024role,
  title={On the role of server momentum in federated learning},
  author={Sun, Jianhui and Wu, Xidong and Huang, Heng and Zhang, Aidong},
  booktitle={Proceedings of the AAAI Conference on Artificial Intelligence},
  volume={38},
  number={13},
  pages={15164--15172},
  year={2024}
}

@article{zhou2025decentralized,
  title={Decentralized Nonconvex Composite Federated Learning with Gradient Tracking and Momentum},
  author={Zhou, Yuan and Shi, Xinli and Li, Xuelong and Zhong, Jiachen and Wen, Guanghui and Cao, Jinde},
  journal={arXiv preprint arXiv:2504.12742},
  year={2025}
}

@article{Bylinkin2024AcceleratedMW,
  title={Accelerated Methods with Compressed Communications for Distributed Optimization Problems under Data Similarity},
  author={Dmitry Bylinkin and Aleksandr Beznosikov},
  journal={ArXiv},
  year={2024},
  volume={abs/2412.16414},
  url={https://api.semanticscholar.org/CorpusID:274982476}
}

@article{Liu2025CompressedDM,
  title={Compressed Decentralized Momentum Stochastic Gradient Methods for Nonconvex Optimization},
  author={Wei Liu and Anweshit Panda and Ujwal Pandey and Christopher Brissette and Yikang Shen and George M. Slota and Naigang Wang and Jiewei Chen and Yangyang Xu},
  journal={Trans. Mach. Learn. Res.},
  year={2025},
  volume={2025},
  url={https://api.semanticscholar.org/CorpusID:280545843}
}

@inproceedings{Yan2023CompressedDP,
  title={Compressed Decentralized Proximal Stochastic Gradient Method for Nonconvex Composite Problems with Heterogeneous Data},
  author={Yonggui Yan and Jiewei Chen and Pin-Yu Chen and Xiaodong Cui and Songtao Lu and Yangyang Xu},
  booktitle={International Conference on Machine Learning},
  year={2023},
  url={https://api.semanticscholar.org/CorpusID:257232829}
}

@article{vogels2023beyond,
  title={Beyond spectral gap (extended): The role of the topology in decentralized learning},
  author={Vogels, Thijs and Hendrikx, Hadrien and Jaggi, Martin},
  journal={arXiv preprint arXiv:2301.02151},
  year={2023}
}

@article{song2022communication,
  title={Communication-efficient topologies for decentralized learning with $ o (1) $ consensus rate},
  author={Song, Zhuoqing and Li, Weijian and Jin, Kexin and Shi, Lei and Yan, Ming and Yin, Wotao and Yuan, Kun},
  journal={Advances in Neural Information Processing Systems},
  volume={35},
  pages={1073--1085},
  year={2022}
}

@inproceedings{NEURIPS2022_a9077da4,
 author = {Bolte, Jerome and Pauwels, Edouard and Vaiter, Samuel},
 booktitle = {Advances in Neural Information Processing Systems},
 editor = {S. Koyejo and S. Mohamed and A. Agarwal and D. Belgrave and K. Cho and A. Oh},
 pages = {26404--26417},
 publisher = {Curran Associates, Inc.},
 title = {Automatic differentiation of nonsmooth iterative algorithms},
 url = {https://proceedings.neurips.cc/paper_files/paper/2022/file/a9077da44185792cb63599cc9e0357bc-Paper-Conference.pdf},
 volume = {35},
 year = {2022}
}

@article{bolte2025inexact,
  title={Inexact subgradient methods for semialgebraic functions: J. Bolte et al.},
  author={Bolte, J{\'e}r{\^o}me and Le, Tam and Moulines, Eric and Pauwels, Edouard},
  journal={Mathematical Programming},
  pages={1--27},
  year={2025},
  publisher={Springer}
}

\vfill

\end{document}